\documentclass[draft]{amsart}
\usepackage{luatex85}
\usepackage[normalem]{ulem}
\numberwithin{equation}{section}
\usepackage{fontspec}
\newfontfamily\cyrfont{FreeSerif.otf}
\def\ggamma{\text{\cyrfont ѵ}}

\usepackage{longtable}
\usepackage[all]{xy}
\usepackage{amssymb}
\usepackage{comment}
\usepackage{stmaryrd}
\usepackage{xcolor}
\usepackage{mathtools}
\usepackage{enumitem}

\definecolor{Vino}{rgb}{0.356,0,0}
\definecolor{Ruta}{rgb}{0.309, 0.459, 0.208}
\usepackage[final,colorlinks,linkcolor=Vino,citecolor=Ruta]{hyperref}
\let\cal\mathcal
\def\Ascr{{\cal A}}
\def\Mscr{{\cal B}}
\def\Bscr{{\cal B}}
\def\Cscr{{\cal C}}
\def\Dscr{{\cal D}}
\def\Escr{{\cal S}}
\def\Eescr{{\cal E}}

\def\Gscr{{\cal G}}
\def\Hscr{{\cal H}}
\def\Iscr{{\cal I}}

\def\Kscr{{\cal K}}
\def\Lscr{{\cal L}}
\def\Mscr{{\cal M}}
\def\Nscr{{\cal N}}
\def\Oscr{{\cal O}}
\def\Pscr{{\cal P}}

\def\Sscr{{\cal S}}
\def\Tscr{{\cal T}}

\let\blb\mathbb
\def\CC{{\blb C}} 
\def\DD{{\blb D}}

\def \PP{{\blb P}}

\def \ZZ{{\blb Z}}
\def \TT{{\blb T}}
\def \NN{{\blb N}}
\def \RR{{\blb R}}

\def\id{\text{id}}

\def\pr{\mathop{\text{pr}}\nolimits}

\def\Res{\operatorname{Res}}

\def\Lotimes{\overset{L}{\otimes}}
\def\quot{/\!\!/}

\def\Mod{\operatorname{Mod}}
\def\mod{\operatorname{mod}}
\def\Gr{\operatorname{Gr}}

\def\gr{\operatorname{gr}}
\def\Lie{\mathop{\text{Lie}}}

\def\Qch{\operatorname{Qch}}
\def\coh{\mathop{\text{\upshape{coh}}}}

\def\gr{\operatorname {gr}}
\def\Spec{\operatorname {Spec}}
\def\Rep{\operatorname {Rep}}

\def\Hom{\operatorname {Hom}}

\def\End{\operatorname {End}}

\def\RHom{\operatorname {RHom}}

\def\uRHom{\operatorname {R\mathcal{H}\mathit{om}}}

\def\supp{\operatorname {supp}}
\def\relint{\operatorname {relint}}
\def\im{\operatorname {im}}

\def\ker{\operatorname {ker}}

\def\End{\operatorname {End}}

\def\id{{\operatorname {id}}}

\def\rk{\operatorname {rk}}

\def\r{\rightarrow}

\def\exist{\exists}

\def\uRHom{\operatorname {R\mathcal{H}\mathit{om}}}

\def\hf{\mathfrak{h}}
\newtheorem{lemma}{Lemma}[section]
\newtheorem{proposition}[lemma]{Proposition}
\newtheorem{theorem}[lemma]{Theorem}
\newtheorem{corollary}[lemma]{Corollary}

\newtheorem{assumption}[lemma]{Assumption}
\newtheorem{observation}[lemma]{Observation}
\newtheorem{convention}[lemma]{Convention}

\newtheorem*{fact}{Fact}

\theoremstyle{definition}

\newtheorem{example}[lemma]{Example}
\newtheorem{definition}[lemma]{Definition}

\newtheorem{question}[lemma]{Question}

{

\newtheorem{step}{Step}

}

\theoremstyle{remark}

\newtheorem{remark}[lemma]{Remark}

\newtheorem{notation}[lemma]{Notation}

\newdimen\uboxsep \uboxsep=1ex
\def\uboxn#1{\vtop to 0pt{\hrule height 0pt depth 0pt\vskip\uboxsep
\hbox to 0pt{\hss #1\hss}\vss}}

\def\uboxs#1{\vbox to 0pt{\vss\hbox to 0pt{\hss #1\hss}
\vskip\uboxsep\hrule height 0pt depth 0pt}}

\def\Sym{\operatorname{Sym}}

\let\oldmarginpar\marginpar
\def\marginpar#1{\oldmarginpar{\raggedright \tiny \baselineskip 0pt \lineskip 0pt #1}}

\def\Sym{\operatorname{Sym}}

\def\Sym{\operatorname{Sym}}
\def\Ob{\operatorname{Ob}}

\def\TT{\mathbb{T}}

\def\sJ{\mathsf{J}}

\def\Re{\operatorname{Re}}

\def\mm{\mathrm{m}}
\def\MMa{\Pscr}
\def\Ma{P}
\def\Sol{\operatorname{Sol}}
\def\rh{\operatorname{rh}}
\def\Perv{\operatorname{Perv}}
\def\an{\operatorname{an}}
\def\cs{\operatorname{cs}}
\def\DR{\operatorname{DR}}
\def\iPP{\overline{\MMa}}
\def\iP{\overline{\Ma}}
\def\Im{\operatorname{Im}}
\def\Re{\operatorname{Re}}

\def\kkappa{\zeta}
\def\jota{\kappa}
\def\KS{\operatorname{KS}}
\def\Schob{\operatorname{Schob}}
\def\nres{\operatorname{nres}}
\title[Perverse schobers and GKZ systems]{Perverse schobers and GKZ systems}
\author[\v{S}pela \v{S}penko and Michel Van den Bergh]{\v{S}pela \v{S}penko and Michel
  Van den Bergh} 
\address[\v{S}pela \v{S}penko]{D\'epartement de Math\'ematique, Universit\'e Libre de Bruxelles, Campus de la Plaine CP 213, Bld du Triomphe, B-1050 Bruxelles}
\email{spela.spenko@vub.be}
\address[Michel Van den Bergh]{Vakgroep Wiskunde, Universiteit Hasselt, Universitaire Campus \\
  B-3590 Diepenbeek} 
\email{michel.vandenbergh@uhasselt.be}
\email{michel.van.den.bergh@vub.be}
\thanks{The second author is a senior researcher at the Research
  Foundation Flanders (FWO).  
 This project has received funding from the European Research Council (ERC) under the European Union's Horizon 2020 research and innovation programme (grant agreement No 885203).
This project was also supported by the FWO grant G0D8616N: ``Hochschild cohomology and
  deformation theory of triangulated categories''.}
\begin{document}

\begin{abstract}
  Perverse schobers are categorifications of perverse sheaves.  In
  prior work we constructed a perverse schober on a partial
  compactification of the stringy K\"ahler moduli space (SKMS)
  associated by Halpern-Leistner and Sam to a quasi-symmetric
  representation of a reductive group. When the group is a torus the
  SKMS corresponds to the complement of the GKZ discriminant locus
  (which is a hyperplane arrangement in the quasi-symmetric case shown by Kite).  We show here that a suitable variation of the
  perverse schober we constructed provides a categorification of the
  associated GKZ hypergeometric system in the case of non-resonant parameters. 
  As an intermediate result we give a description of  the monodromy of such ``quasi-symmetric'' GKZ hypergeometric systems. 
\end{abstract}
\maketitle
\tableofcontents
\section{Introduction}
The classical Riemann-Hilbert correspondence yields an equivalence
between the triangulated category of (regular holonomic) $D$-modules
and that of constructible sheaves
\cite{Deligne7,Kashiwara,Mebkhout}. Under this equivalence the abelian
category of regular holonomic D-modules, more or less the same as
systems of linear partial differential equations with regular
singularities, corresponds to the abelian category of perverse sheaves
\cite{Kashiwara2,BBD}.
Recently, guided by mirror symmetry applications, Kapranov and Schechtman \cite{KapranovSchechtmanSchobers} have introduced categorifications
of perverse sheaves and called these perverse schobers. It is intended that perverse schobers would serve as coefficient data in the construction of Fukaya like categories.


\medskip

The aim of this paper is to show that a variant of the perverse schober
constructed in \cite{SVdB10} provides a categorification of the
 GKZ hypergeometric system in the case of toric data
associated to a quasi-symmetric representation (see below).

As an intermediate result we obtain a general formula for the
monodromy of such ``quasi-symmetric'' GKZ systems which we believe is new. The result applies to all classical, one 
variable, hypergeometric equations (generalizing in particular \cite{BeukersHeckman}), and also to some of the classical higher dimensional GKZ systems such as Lauricella $F_D$
\cite[\S9]{Beukers} and Appell $F_1$ \cite[p.94]{bod2013algebraic} (for those particular systems we recover the results from \cite{DeligneMostow,Picard,Terada}).

\subsection{Perverse schobers}
\emph{Perverse schobers} \cite{BondalKapranovSchechtman,KapranovSchechtmanSchobers} are categorifications of perverse sheaves on suitably stratified topological spaces. 
There is no general definition for perverse schobers but they have been
defined for example on complex vector spaces stratified by complexified real hyperplane
arrangements \cite{KapranovSchechtmanSchobers}. This is accomplished by categorifying the work of \cite{KapranovSchechtman}. 
Work is also ongoing in the case of Riemann surfaces
\cite{Donovan,DKSS}.

\medskip

We discuss the case of hyperplane arrangements. Let $V$ be a real
vector space and let $\Hscr$ be an affine hyperplane arrangement in~$V$. Let $\Cscr$ be the open cell complex on $V$ induced by $\Hscr$, ordered
by $C'<C$ iff $C'\subset \bar{C}$ and let $\Cscr^0\subset \Cscr$ be
the set of chambers. To this data one associates a particular kind of perverse schober
on $V_\CC$ which is called an $\Hscr$-schober in \cite{BondalKapranovSchechtman} and whose definition is
recalled in \S\ref{subsec:reminder}. 

\medskip

For the purpose of this
introduction we will just mention that the data defining an
$\Hscr$-schober consists of a family of triangulated categories
$(\Eescr_C)_{C\in \Cscr}$ such that $\Eescr_{C'}\subset \Eescr_C$ if
$C<C'$. Extra conditions are imposed on this data which imply in
particular that the collection $(\Eescr_C)_{C\in \Cscr^0}$ defines a
representation of the fundamental groupoid of
$V_\CC\setminus \Hscr_\CC$ in the category of triangulated
categories. Informally: a local system of triangulated categories on
$V_\CC\setminus \Hscr_\CC$.

\medskip

The \emph{decategorification} $K^0_\CC(\Eescr)$ of an $\Hscr$-schober $\Eescr$ assigns
the vector space $E_C:=K^0(\Eescr_C)_\CC$ to $C\in \Cscr$. The definition of an $\Hscr$-schober is such
that $(E_C)_{C\in \Cscr}$ is precisely the data required by  \cite{KapranovSchechtman}
to define a perverse sheaf on $V_\CC$ which is smooth with respect to $\Hscr_\CC$.
We will denote this perverse sheaf by $\tilde{K}^0_\CC(\Eescr)$.

\subsection{A perverse schober using Geometric Invariant Theory}
\label{sec:HLS}
In the paper \cite{HLSam} Halpern-Leistner and Sam observed that the Stringy 
K\"ahler Moduli Space (SKMS) of suitable GIT quotients is given by  the complement of a (toric) hyperplane arrangement and they used this observation to
construct a local system of triangulated categories on the SKMS.  In \cite{SVdB10}, reviewed in \S\ref{subsec:construct}, we were able to extend this local system
to an $\Hscr$-schober (denoted by $\Sscr^c$ below).

\medskip

For the convenience of the reader we briefly give some indications about the construction.
We restrict ourselves to the torus case as this is the case we will be concerned
with in this paper.
Let $T$ be a torus acting faithfully on a representation~$W$ with weights $(b_i)_i\in X(T)$. 
In the rest of this introduction, and in most of the paper, we also assume that $W$ is \emph{quasi-symmetric} \cite{SVdB},
i.e.  $\sum_{b_i\in \ell} b_i=0$
for each line $\ell\subset X(T)_\RR$, passing through the origin.

\medskip

To the data $(T,W)$ \cite{HLSam} associates a
hyperplane arrangement $\Hscr$ in the real vector space $X(T)_\RR$
consisting of the translates by elements of $X(T)$ of the hyperplanes in
$X(T)_\RR$ spanned by the faces of the zonotope
$\Delta:=\sum_i [-1/2,0]b_i\subset X(T)_\RR$.  The schober $\Sscr^c$
constructed in \cite{SVdB10} is such that for $C\in \Cscr$,
$\Sscr^c_C$ is the thick subcategory of $D(W/T)$ 
 generated by
$\chi\otimes \Oscr_{W}$ for $\chi \in (\nu+{\Delta})\cap X(T)$ with
$\nu\in -C$ chosen arbitrarily.\footnote{In \cite{SVdB10} our signs were slightly different. $\nu$ was an element of $C$ (instead of $-C$) and the weights of $W$ were $(-b_i)_i$ (instead of $(b_i)_i$).}  We note that the $\Sscr^c_C$ are
derived categories of suitable noncommutatve resolutions of
$W\quot T$ originally constructed in \cite{SVdB}. 
For $C\in \Cscr^0$ they are derived equivalent to certain GIT quotient stacks and these were used in   \cite{HLSam} in their description of the
  local system on $X(T)_\CC\setminus \Hscr_\CC$. 

\begin{example} \label{ex:example1}
Let $T=G_m$. We identify $X(T)$ with $\ZZ$ and we write $(n)$ for the one-dimensional $T$-representation with weight $n\in \ZZ$.  Let $W=(-1)\oplus (-1)\oplus (1)\oplus (1)$ and put $P_n=(n)\otimes \Oscr_{W}\in D(W/T)$.
Then $\Hscr=\ZZ$. If $C=]a,a+1[\in \Cscr^0$ then $\Sscr^c_C$ is the full thick subcategory of $D(W/T)$ generated $P_{-a-1}$, $P_{-a}$.
If $C=\{a\}\in \Cscr$ then $\Sscr^c_C$ is generated by $P_{-a-1}$, $P_{-a}$, $P_{-a+1}$.
\end{example}
The $\Hscr$-schober $\Sscr^c$ is in fact invariant under translation by $X(T)\subset X(T)_\RR$. Hence we can think of it as a perverse schober on the quotient
$X(T)_\CC/X(T)$ and the latter may be canonically identified with the dual torus $T^\ast$ via the map
\[
X(T)_\CC=Y(T^\ast)_\CC=\Lie(T^\ast) \xrightarrow{\exp(2\pi i-)} T^\ast
\]
where $Y(T^\ast)=X(T^\ast)^\ast$ is the group of cocharacters of $T^\ast$.
In particular $\tilde{K}_{\CC}^0(\Sscr^c)$ descends to a perverse sheaf on $T^\ast$ which we denote by $S^c$. 
\subsection{The mirror picture}
Since we are in a toric setting the ``$B$-data'' given by $(b_i)_{i=1,\ldots,d}$ has associated ``$A$-data'', referred to as the \emph{Gale dual} of $B$. Let $\TT^d=(\CC^\ast)^d$ be 
the torus which acts coordinate wise on $W$. Then by faithfulness we may assume that $T$ acts on $W$ via an inclusion $T\subset \TT$. Put $H=\TT/T$.
We obtain an exact sequence
\[
0\r X(H)\r X(\TT)\xrightarrow{B} X(T)\r 0
\]
where $B$ sends the standard basis $(e_i)_{i=1,\ldots,d}$ of $X(\TT)$ to $(b_i)_{i=1,\ldots,d}$. Then we obtain a morphism $A$ by dualizing this sequence
\[
0\r Y(T)\r Y(\TT)\xrightarrow{A} Y(H)\r 0.
\]
Usually $A$ is identified with the sequence of $H$-cocharacters given by $a_i=A(e^\ast_i)$. 
\subsection{The GKZ system}
Gelfand, Kapranov and Zelevinsky discovered 
a captivating common generalisation of many classical hypergeometric differential equations introduced by Euler, Gauss, Horn, Appell,\dots, 
which is now known as the \emph{GKZ (hypergeometric) system}. The GKZ system lives on $\TT^\ast$ and depends on $A$ and in addition on a continuous parameter $\alpha\in Y(H)_\CC$. 
It is weakly $H^\ast$-invariant (see \S\ref{sec:weq})  and hence
with some care it may be regarded as living on $\TT^\ast/H^\ast=T^\ast$ (see
\S\ref{subsec:splitting} below for the precise procedure).
In our setting the GKZ system has regular singularities \cite{{Adolphson}} and so its solution complex $\bar{P}(\alpha)$ is a family of perverse sheaves on $T^\ast$ parametrized by 
$\alpha\in Y(H)_\CC$. 

\medskip

The GKZ system is most easily understood for parameters which are sufficiently generic. Let $\Iscr\subset Y(H)_\RR$ be the real hyperplane arrangement
consisting of the translates by elements of $Y(H)$ of the hyperplanes in
$Y(H)_\RR$ spanned by the facets of the cone $\RR_{\ge 0}A$.
 We say that $\alpha\in Y(H)_\CC$ is \emph{non-resonant} if $\alpha\not\in \Iscr_\CC$. The following is a basic result in the theory of GKZ systems
(see \S\ref{subsec:gkzperverse})
\begin{proposition} \cite{GKZEuler} \label{prop:intrononres}
Assume $\alpha \in Y(H)_\CC$ is non-resonant. Then $\bar{P}(\alpha)$ is a simple perverse sheaf, in particular it is the 
\emph{intermediate extension} 
of its corresponding local system. 
If $\alpha,\alpha'\in Y(H)_\CC$ are non-resonant and $\alpha-\alpha'\in Y(H)$ then $\bar{P}(\alpha)\cong \bar{P}(\alpha')$.
\end{proposition}

\begin{example}\label{eq:Gauss} Let $(T,W)$ be as in Example \ref{ex:example1}. In that case $T^\ast=\CC^\ast= \PP^1\setminus \{0,\infty\}$.
After a suitable identification $Y(H)_\CC\cong \CC^3$ and choosing $\alpha=(a,b,c)\in \CC^3$ the corresponding GKZ system on $T^\ast$ is the 
\emph{Gaussian hypergeometric equation}
\[
z(1-z)\frac{d^2w}{dz^2}+[c-(a+b+1)z]\frac{dw}{dz}-abw=0.
\]
The non-resonant parameters are those $a,b,c$  for which $a,b,a-c,b-c$ are all non-integers. 
\end{example}

\subsection{Kite's result}
The singular locus of the GKZ system on $T^\ast$ is defined by the so-called \emph{principal $A$-discriminant} $E_A$  \cite{GKZhyper}.
This paper started with following beautiful observation in \cite{Kite}.
\begin{theorem}[\protect{\cite[Proposition 4.1]{Kite}, \S\ref{subsec:GKZdislocus}}] \label{kite:intro}
The image of $\Hscr_\CC$ (see \S\ref{sec:HLS}) in $T^\ast$ is, up to a translation $\bar{\tau}$ (see \S\ref{subsec:GKZdislocus}), equal to  $V(E_A)$, the vanishing locus of $E_A$.
\end{theorem}
This immediately suggests the question:
\begin{question} \label{qu:question}
What is the relationship, if any,  between the perverse sheaves $\bar{P}(\alpha)$, $\alpha \in Y(H)_\CC$, and the perverse sheaf $S^c$?
\end{question}
Alas we immediately see an issue. $S^c$ does not depend on any continuous parameters whereas $\bar{P}(\alpha)$ does. So 
to make sense of the question we either have to pick a specific $\alpha$ or we have to 
change the definition of $S^c$ so that it also depends on continuous parameters. In this paper we take the latter approach.

\subsection{Creating a family}
Since $X(T)$ is a quotient of $X(\TT)$ we may view $X(\TT)$ as acting on the vector space $X(T)_\CC$.
We first replace the $X(T)$-equivariant $\Hscr$-schober $\Sscr^c$  by an $X(\TT)$-equivariant
$\Hscr$-scho\-ber $\bar{\Sscr}^c$ built from subcategories of $D(W/\TT)$.
More precisely for $C\in \Cscr$, $\bar{\Sscr}_C^c$ is defined as the thick subcategory of $D(W/\TT)$ spanned by $\chi\otimes \Oscr_{W}$ 
for $\chi\in X(\TT)$
such that the image of $\chi$ in $X(T)$ is in $\nu+{\Delta}$ for $\nu\in -C$.
Note that $X(H)=\ker(X(\TT)\r X(T))$ acts trivially on $\Cscr$. In other words $X(H)$ acts on $\bar{\Sscr}_C^c$ for every $C\in \Cscr$. 

\medskip

We now change our point of view a bit.
We choose
a splitting $\TT=T\times H$ and we  consider $\bar{\Sscr}^c$ as an $X(T)$-equivariant schober on $X(T)_\CC$, equipped with an additional
$X(H)$-action. The reward for doing this is that the decategorification $K^0_\CC(\bar{\Sscr}^c)$ of $\bar{\Sscr}^c$ is now built from modules over the group ring $\CC[X(H)]$ of $X(H)$. Since $\CC[X(H)]\cong \CC[H]$ (the coordinate ring of $H$) we may specialize
$K^0_\CC(\bar{\Sscr}^c)$  at an element $h$ of $H$. Denote the result by $K^0_h(\bar{\Sscr}^c)$. Via \cite{KapranovSchechtman} we
obtain a corresponding perverse sheaf on $X(T)_\CC$ which we denote by $\tilde{K}^0_h(\bar{\Sscr}^c)$. After descent under $X(T)$ this yields a perverse sheaf
of $X(T)_\CC/X(T)=T^\ast$ which we denote by $S^c(h)$.
 We recover $S^c$ defined above as $S^c(\bold{1})$ where $\bold{1}$ denotes the unit element of $H$. 
\subsection{Main result}
The following result fulfils the promise made at the start of the introduction. 
\begin{theorem}[Theorem \ref{thm:connect}] \label{mainth:intro}
Assume that $\alpha\in Y(H)_\CC$ is non-resonant and put $h=\exp(2\pi i\alpha)\in H$.
Then 
$
\bar{P}(\alpha)\cong \bar{\tau}^*S^c(h^{-1})
$ where $\bar{\tau}$ is the translation introduced in Theorem \ref{kite:intro}.
\end{theorem}
Now note that while Theorem \ref{mainth:intro} is in the spirit of Question \ref{qu:question}, it does not actually answer the latter. Indeed  as we have said $S^c=S^c(\bold{1})$ and
$\exp(2\pi i\alpha)=\bold{1}$ if and
only if $\alpha\in Y(H)$ and such $\alpha$ are, in some sense, as far away from being non-resonant as possible.  The case of resonant parameters is currently
work in progress and will be discussed in a future paper.
\begin{remark} \label{rem:repar}
Theorem \ref{mainth:intro} suggests that it should be possible to parametrize the GKZ system by $H$, instead of by the covering space $Y(H)_\CC$. This
is indeed possible, analytically, if we restrict ourselves to the non-resonant part of $Y(H)_\CC$. See \S\ref{subsec:parametricdescent}.
\end{remark}
\subsection{Discussion}
Theorem \ref{mainth:intro} connects two (families of) perverse sheaves, one defined via algebra and one defined via analysis.
A perverse sheaf is generically a (shifted) local system.
Restricting ourselves to local systems, results in the spirit of Theorem \ref{mainth:intro}, but with different specifics, appear elsewhere.
We give some examples.
\begin{itemize}
\item 
In \cite{BHMB} a local system ``at infinity''  obtained from the derived categories of suitable toric GIT quotients
is shown to be the same as a local system at infinity corresponding to a GKZ system.
This result is not
  restricted to the quasi-symmetric case.
\item In \cite{ABM} a result like Theorem \ref{mainth:intro} is proved for the standard resolution of suitable Slodowy slices. The corresponding
derived category is equipped
with an action of the affine braid group and hence may be regarded as a local system of triangulated categories on the complement of a hyperplane arrangement.
After decategorification, this local system is shown to be the same as a local system defined using quantum cohomology (which is equipped with a connection).
It would be very interesting to apply these techniques in the setting of Theorem \ref{mainth:intro}. 
\item Bridgeland observes that the moduli space of stability conditions of a triangulated category often arises as a covering space 
of a nice space  such that 
 the
  central charge map  descends to a multivalued map which satisfies an interesting differential equation
  whose monodromy comes from  the action of autoequivalences. For
  example, \cite{BridgelandYuSutherland} (see also \cite{Ikeda}) considers the
   Calabi-Yau completions \cite{Gi,Keller20} of the $A_2$-quiver. It is
  shown that the associated spaces of stability conditions, modulo suitable autoequivalences, are given by $\mathfrak{h}^{\text{reg}}/S_3$ where
$\mathfrak{h}$ is the two-dimensional irreducible representation of $S_3$ and $\mathfrak{h}^{\text{reg}}\subset \mathfrak{h}$ is the hyperplane complement 
on which $S_3$ acts freely. The central charge, which is a multivalued holomorphic map $\mathfrak{h}^{\text{reg}}/S_3\r \CC^2$, satisfies a hypergeometric equation \cite[eq. (14)]{BridgelandYuSutherland}.
\end{itemize}
\subsection{Outline of proof}
The proof of Theorem \ref{mainth:intro} consists of three major steps.
\begin{step}
\label{step:step1}
We compute the monodromy of the perverse sheaves $\bar{P}(\alpha)$. To this end we follow  the method introduced by Beukers in \cite{Beukers}
but in our setting we are able to get a precise formula for the full monodromy and not just the local monodromy as in loc.\ cit. 
We believe that this result is interesting in its own right but it is a bit too technical to state here. See Theorem \ref{th:mainth1}.
\end{step}
\begin{step}
\label{step:step2}
We compute the monodromy of the perverse sheaf $S^c(h)$ and show that it is the same as the monodromy of $\bar{P}(\alpha)$ modulo the reparametrisation
indicated in Theorem \ref{mainth:intro}. This is done by using suitable complexes introduced in \cite{SVdB}. See Proposition \ref{lem:supportschobermonodromy1}.
\end{step}
\begin{step} \label{step:step3}
We prove that $S^c(h)$ is the intermediate extension of the local system it defines, which allows us to conclude by combining Proposition \ref{prop:intrononres} with Steps \ref{step:step1},\ref{step:step2}.

To accomplish Step \ref{step:step3} we first use some combinatorics to show $S^c(h)$ has no quotients supported on $\Hscr_\CC/X(T)$.
We then show that  $K^0_h(\bar{\Sscr}^c)$ is nearly self dual, implying that $S^c(h)$ also has no subobjects supported on
$\Hscr_\CC/X(T)$. Therefore it has the intermediate extension property. 

The self duality is shown by defining a subschober $\bar{\Sscr}^f\subset \bar{\Sscr}^c$ of
$\bar{\Sscr}^c$ where $\bar{\Sscr}^f_C$ consists of those objects in $\bar{\Sscr}^c_C$ which are supported on the nullcone of $W$. 
Then $K^0(\bar{\Sscr}_C^f)$ is both:  \emph{isomorphic} to $K^0(\bar{\Sscr}_C^c)$, after a suitable localization, and  (2) \emph{dual} to $K^0(\bar{\Sscr}_C^c)$, via the Euler form.
See \S\ref{sec:intermediate}.
\end{step}

\section{Acknowledgement}
We thank Roman Bezrukavnikov for informing us about \cite{ABM}, Lev Borisov  for helpful discussions and Igor Klep for help with references.
We also thank  Tom Bridgeland for enlightening comments about the context of our results.

A large part of this work was carried out at the Max Planck Institute for Mathematics in Bonn in the fall of 2019. The authors are very
grateful for wonderful working conditions provided by the institute. 

We  acknowledge the referees for insightful comments and suggestions.

\section{Notation and conventions}
A list of notations is given at the end of the paper.
\subsection{Identifications for tori}
\label{sec:id}
Below we routinely use some identifications associated with algebraic tori. We summarize these below.
Let $L$ be a finitely generated free abelian group and let $T:=L\otimes_{\ZZ}\CC^\ast$ be the corresponding algebraic torus.
Then we have canonical identifications
\begin{align*}
\Lie(T)&=L\otimes_{\ZZ} \CC\\
X(T)&=L^\ast\\
Y(T)&=L.
\end{align*}
In particular we obtain a canonical isomorphism
\[
Y(T)_\CC\cong \Lie(T).
\]
One checks that the resulting composition
\[
Y(T)\r Y(T)_\CC\cong \Lie(T)
\]
sends $\lambda:G_m\r T$ to $d\lambda:\CC=\Lie(G_m)\r \Lie(T)$, evaluated
at $1$.
The exponential exact sequence
\[
0\r \ZZ\r \CC\xrightarrow{e^{2\pi i-}} \CC^\ast\r 0
\]
after tensoring with $L$ becomes an exact sequence
\[
0\r Y(T)\r \Lie(T)=Y(T)_\CC\xrightarrow{e^{2\pi i-}} T\r 0.
\]
The dual torus $T^\ast$ is defined as $L^\ast\otimes_{\ZZ}\CC^\ast$. Hence we have identifications
\begin{align*}
\Lie(T^\ast)&=\Lie(T)^\ast\\
X(T^\ast)&=Y(T)^\ast\\
Y(T^\ast)&=X(T)^\ast.
\end{align*}
\begin{remark} Below we will usually silently make the identifications
  given above.
\end{remark}
\subsection{Setting}\label{subsec:setting}
Throughout we fix an exact sequence of free abelian groups  
\begin{equation}
\label{eq:settingses1}
0\r L\xrightarrow{B^\ast} \ZZ^d\xrightarrow{A} N\r 0.
\end{equation}
with $n:=\rk L$. We do not fix  bases for $L$, $N$ but the middle term $\ZZ^d$ 
will be equipped with its tautological basis (denoted by $(e_i)_i$) and this basis plays an essential role in the theory.

The map denoted by $B^\ast$ is the dual of the corresponding map $B:\ZZ^d\r L^\ast$. We write
$b_i=B(e_i)\in L^\ast$, $a_i=A(e_i)\in N$. Following custom we will usually identify $A,B$ with the sequences of elements $(a_i)_i$, $(b_i)_i$.

Tensoring \eqref{eq:settingses1} with $-\otimes_{\ZZ}\CC^\ast$ we obtain an exact sequence of algebraic tori
\begin{equation}\label{eq:settingses}
1\to T\xrightarrow{B^\ast}\TT\xrightarrow{A} H\to 1
\end{equation}
with $\dim T=n$, $\dim \TT=d$.
Let $W\cong \CC^d$ and let $\TT=(\CC^*)^d$ act coordinate-wise on~$W$. Then $T$ acts on $W$ via the map $B^\ast$. 

From \eqref{eq:settingses} we obtain:
\begin{equation}\label{eq:seschar}
0\to X(H)\xrightarrow{A^\ast} X(\TT)\xrightarrow{B} X(T)\to 0.
\end{equation}
In this interpretation the $b_i\in X(T)\cong L^*$ are the weights for the $T$-action on $W$. 
Dualizing \eqref{eq:seschar}, we obtain 
\begin{equation*}\label{eq:sec1par}
0\to Y(T)\xrightarrow{B^\ast} Y(\TT)\xrightarrow{A} Y(H)\to 0
\end{equation*}
which is just an avatar of \eqref{eq:settingses1}.

Below we make the identification (see \S\ref{sec:id})
\begin{equation}\label{eq:ext}
X(T)_\CC/X(T)=Y(T^*)_\CC/Y(T^*)=\Lie(T^*)/Y(T^*)\overset{e^{2\pi i-}}{\cong} T^*.
\end{equation}
\subsection{Quasi-symmetry}\label{subsec:quasi-symmetric}  
We say that $W$ is {\em
  quasi-symmetric} if for every line $0\in\ell\in X(T)_\RR$,
$\sum_{b_i\in\ell}b_i=0$. 
\subsection{Affine hyperplane arrangements}
\label{sec:genhyp}
We will encounter hyperplane arrangements in several different contexts
so we introduce some general notions related to them.

Below let $\Hscr$ be an affine hyperplane arrangement\footnote{We
  always silently assume that affine hyperplane arrangements are
  locally finite and that their corresponding central arrangements are
  finite.}  in an $n$-dimensional real vector space~$V$. Following
custom we will often confuse $H\in \Hscr$ with a specific affine
function (``equation'') $H:V\r \RR$ defining $H$.

For $H\in \Hscr$ we let $H_0$ be the parallel translate of $H$  which passes through the origin and we define the \emph{central hyperplane arrangement} $\Hscr_0$ corresponding
to $\Hscr$ as
\[
\Hscr_0=\{H_0\mid H\in \Hscr\}.
\]
If $H\in \Hscr$ then the corresponding complex hyperplane $\Hscr_\CC\subset V_\CC$ is given by  the same equation which defines $H$. To imagine $H_\CC$ one may
use the following concrete interpretation
\[
H_\CC=\{x_0+ix_i\in V+iV\mid H(x_0)=0, H_0(x_1)=0\}=H\times i H_0.
\]
Note that $H_\CC$ has codimension two in $V_\CC$.
The \emph{complex hyperplane arrangement} $\Hscr_\CC\subset V_\CC$ is defined by
\[
\Hscr_\CC=\{H_\CC\mid H\in \Hscr\}.
\]
 A vector $\rho\in V$
is \emph{generic} if it is not parallel to any of the hyperplanes in
$\Hscr$.

The closures of the connected components of $V\setminus \Hscr$ are convex polytopes. The collection of the (relatively open) faces 
of these polytopes is denoted by $\Cscr$.
This set is partially ordered by 
  $C_1\le C_2$ iff $C_1\subset \overline{C_2}$. 
  We denote $C_1\wedge C_2=\relint(\overline{C_1}\cap \overline{C_2})$. 
  By $\Cscr^0$ we denote the set of chambers, i.e.\ the polytopes of dimension $n$, in $\Cscr$. 
A triple of faces $(C_1,C_2,C_3)\in \Cscr$ is {\em collinear} if there exists $C'\leq C_1,C_2,C_3$ and there exist points $c_i\in C_i$ such that $c_2\in[c_1,c_3]$. 
\subsection{A hyperplane arrangement in {\boldmath $X(T)_\RR$}}
\label{sec:hyperplane}
Unless otherwise specified we will use the notation $(\Hscr,\Cscr)$ in a concrete sense which we now outline. 
Put
\begin{equation}
\label{eq:Delta}
\Sigma=\left\{\sum_i \beta_i b_i\mid \beta_i\in
(-1,0)\right\},\qquad \Delta=(1/2)\overline{\Sigma}.
\end{equation}
Denote by $(H_i)_i$ the affine hyperplanes in $X(T)_\RR$ spanned by
the facets of $\Delta$. Put 
\begin{equation}
\label{eq:Hscrdef}
\Hscr=\bigcup_i(-H_i +X(T)).
\end{equation}

The following lemma is rather easy to check.
\begin{lemma}\label{lem:genericB}
$\rho\in X(T)_\RR$ is generic (see \S\ref{sec:genhyp}) if and only if $\rho$ does not lie in any (proper) subspace spanned by $b_i$'s. If
$W$ is furthermore  quasi-symmetric, then this is also equivalent to the condition that $\rho$ does not lie in the boundary of any cone spanned by $b_i$'s. 
\end{lemma}

For $C\in \Cscr$ (defined as in \S\ref{sec:genhyp}) we introduce the following notation. Let
\[\mathcal{L}_{C}=(\nu+\Delta) \cap X(T)
\] 
for an arbitrary
\(\nu \in C\left(\mathcal{L}_{C} \text { does not depend on } \nu,
  \text{ see \cite{HLSam}\cite[\S 5.1]{SVdB10}}\right)\). It follows from loc.cit.\ (see \cite[Lemma 5.3]{SVdB10}) that 
\begin{equation}\label{eq:inclusion_Lscr_C}
C'<C =\!\Rightarrow \Lscr_{C}\subset \Lscr_{C'}.
\end{equation}

\subsection{Non-resonance condition}
\label{sec:nonresonance}
Consider the real central hyperplane arrangement $\Iscr_0$ in $\mathfrak{h}=Y(H)_\RR$ spanned by the facets of $\RR_{\ge 0}A$ and let $\Iscr$ be the real affine hyperplane arrangement $Y(H)+\Iscr_0$. 
 We say that $\alpha$ is \emph{non-resonant}
if $\alpha\not\in \Iscr_\CC$. Also let $\Iscr'_0$ be the real central hyperplane arrangement consisting of hyperplanes spanned by subsets of $A$ and let
$\Iscr'=Y(H)+\Iscr'_0$. 
We say that \(\alpha\) is {\em totally non-resonant} if $\alpha\not\in \Iscr'_\CC$. Clearly $\Iscr_0\subset \Iscr'_0$. So totally non-resonant implies non-resonant.

We will say that $h\in H$ is {\em (totally) non-resonant} if $(\log h)/(2\pi i)\in \mathfrak{h}$ is (totally) non-resonant. Note that it does not matter which branch of $\log h$ we choose. 

We will denote by $H^{\nres}$, $\mathfrak{h}^{\nres}$ the (analytically open) sets of non-resonant elements of $H$ and $\mathfrak{h}$.

\section{The GKZ hypergeometric system}
We now discuss the celebrated Gelfand-Kapranov-Zelevinsky system of differential equations \cite{GKZhyper}. It is well-known that the theory of said system 
becomes substantially easier under the following assumptions which we now make.
\begin{assumption} 
\label{ass:setting} Except when otherwise specified  we assume 
$\sum_i b_i=0$. 
It is easy to see that the latter condition is equivalent to the assumption that there exists $h\in N^\ast$ such that {$\forall i:\langle h,a_i\rangle = 1$}.
\end{assumption}
\begin{remark}
\label{eq:setup}
Assumption \ref{ass:setting} is implied by quasi-symmetry.
It follows from Assumption \ref{ass:setting} that $W$ is \emph{unimodular}.
In our current setup we also have $\ZZ A=N$ and $\ZZ B=L^\ast\cong X(T)$. Furthermore under Assumption \ref{ass:setting}, $\Delta$ as introduced in \eqref{eq:Delta} is centrally symmetric
and in particular the minus sign in the definition of $\Hscr$ in \eqref{eq:Hscrdef} is superfluous.
\end{remark}
\subsection{The GKZ system of differential equations}
\label{ssec:gkz}
Let $x_i$, $1\leq i\leq d$, be the coordinates on $\TT^*=(\CC^*)^d$, we write $\partial_i=\partial/\partial x_i$. 
For $l\in L$ write $B^\ast(l)=(l_i)_{i=1}^d$ and put
\[
\Box_l=\prod_{l_i>0}\partial_i^{l_i}-\prod_{l_i<0}\partial_i^{-l_i}.
\]
Let $\hf={\rm Lie}(H)$ so that $\hf^*={\rm Lie}(H^*)$. As $H^\ast\subset \TT^\ast$,  $H^\ast$ acts by translation on~$\TT^\ast$.
Therefore
$\hf^*$ acts by derivations on $\Oscr_{\TT^*}$. If $\phi\in \mathfrak{h}^\ast$ then we write $E_\phi$
for the corresponding derivation.

Put $\Dscr=\Dscr_{\TT^\ast}$. The GKZ system of differential equation with parameter $\alpha\in Y(H)_\CC=\hf$ corresponds to the following $\Dscr$-module
\begin{equation}
\label{eq:GKZ}
\MMa(\alpha)=\Dscr\bigg/\left(\sum_{l\in X(T^*)}\Dscr\Box_l+\sum_{\phi\in\hf^*}\Dscr(E_\phi-\phi(\alpha))\right).
\end{equation}

\begin{theorem}
\cite[Theorem 3.9]{Adolphson}
\label{rem:holonomic}
The $\Dscr$-module $\MMa(\alpha)$ is holonomic with regular singularities.
\end{theorem}
\subsection{Non-resonance}\label{subsec:nonresonance}
The
GKZ system  is {\em non-resonant} if the parameter $\alpha$ is non-resonant (see \S\ref{sec:nonresonance}).  
Let $\mathfrak{v}$ be the normalized volume\footnote{The normalized volume is the volume divided by the fundamental volume of the lattice $N\cap \{\langle h,-\rangle=1\}$ where $h\in N^\ast$ is as in Assumption \ref{ass:setting}.} of the convex
hull of $A$. In the non-resonant case~$\mathfrak{v}$ equals the rank of the GKZ
system~\cite{GKZhyper,GKZcorr}.  In the quasi-symmetric case (see
\S\ref{subsec:quasi-symmetric})~$\mathfrak{v}$ also equals the number of integral points in
$\varepsilon + (1/2)\Sigma$ for arbitrary generic (see \S\ref{sec:hyperplane})~$\varepsilon \in X(T)_\RR$
(see e.g. \cite[Corollary 4.2, Remark 4.3]{HLSam} together with \cite{BorisovHorja} (explicitly stated in \cite[Theorem A.1]{SVdB4})).\footnote{Lev Borisov  showed us a nice direct combinatorial proof of this result.}

\subsection{Descent of the GKZ system}\label{subsec:splitting}
The GKZ system is a weakly $H^*$-equivariant 
$\Dscr$-\-module on $\TT^\ast$ with character $\alpha$ (see Definition \ref{def:character}), so it represents an object
in $\Qch_\alpha(H^\ast,\Dscr_{\TT^\ast})$ (see loc.\ cit.).

By Corollary \ref{cor:maincor}, choosing a splitting $\iota:\TT\to T$ of $B^*$ in \eqref{eq:settingses} and denoting by a slight abuse of notation $\iota:T^*\to \TT^*$ also adjoint to $\iota$, 
$\Dscr$-module pullback 
\begin{equation}
\label{eq:iota}
\iota^*:\Qch_\alpha(H^\ast,\Dscr_{\TT^\ast})\to \Qch(\Dscr_{T^\ast})
\end{equation}
is an equivalence of (abelian) categories. Below we will be concerned with 
\[\iPP(\alpha):=\iota^\ast \MMa(\alpha)\in \Dscr_{T^\ast}.\]

\begin{remark} \label{rem:difference}
$\iPP(\alpha)$ depends on the splitting $\iota$, but in a relatively weak way. The difference of two splittings $\iota$, $\iota'$ can be regarded as a map $\delta:N\r L$ which
then induces a corresponding map $\delta:T^\ast \r H^\ast$. Then it follows from Lemma \ref{the-sections-they-are-a-changin} that $\iota^{\prime \ast}\MMa(\alpha)$ differs
from $\iota^{\ast}\MMa(\alpha)$ by tensoring with the invertible $\Dscr$-module generated by the multi-valued function $\theta_{\alpha,\delta}$ on $T^\ast$ which satisfies
$\theta_{\alpha,\delta}\circ e^{2\pi i-}=e^{2\pi i \langle \delta(-),\alpha\rangle}$ where we have lifted $\delta$ to a linear map $\mathfrak{t}^\ast\r \mathfrak{h}^\ast$.
\end{remark}
\subsection{The GKZ discriminant locus}\label{subsec:GKZdislocus}
The  singular
locus of $\iPP(\alpha)$ is given by the principal $A$-discriminant $E_A$
\cite{GKZhyper}. If $W$ is quasi-symmetric then by \cite[Proposition 4.1]{Kite}  it is the 
image of a hyperplane
arrangement in $X(T)_\CC$ under the identification $X(T)_\CC/X(T)\cong T^\ast$ \eqref{eq:ext}.
More precisely, 
\[
(\Hscr_\CC+\kkappa)/X(T)\overset{e^{2\pi i-}}{\cong}V(E_A)
\]
 where $\Hscr$ is the hyperplane arrangement defined in \S\ref{sec:hyperplane} and
$\kkappa:=-\frac{i}{2\pi}\sum_i (\log{|n_j|})b_j$ with $b_j=n_j l_\ell$
for a generator $l_\ell$ of the rank one lattice $X(T)\cap \ell$ where
$\ell\in X(T)_\RR$ is the line through the origin which contains $b_j$.
\subsection{Reminder on the Riemann-Hilbert correspondence}
\label{sec:reminder}
For a smooth qua\-si-compact separated scheme $X/\CC$ of pure dimension $d$ 
let $\Mod_{\rh}(\Dscr_X)$ be the full subcategory of $\Qch(\Dscr_X)$
consisting of regular holonomic $\Dscr_X$-modules. 
Moreover, 
let $D^b_{\rh}(\Dscr_X)$ be the bounded derived category of $\Dscr_X$-modules
with regular holonomic cohomology. 
\begin{remark} \label{rem:compat}
The category $D^b_{\rh}(\Dscr_X)$ is compatible with all standard
$\Dscr$-module operations (e.g.\ \cite[Theorem 6.1.5]{hottabook}). Moreover, 
the category $\Mod_{\rh}(\Dscr_X)$ is closed under subquotients in $\Qch(\Dscr_X)$
(see \cite[\S6.1]{hottabook}).
\end{remark}
Let $D^b_{\cs}(\CC_{X^{\an}})$ be the bounded derived category of sheaves
of vector spaces on $X^{\an}$ with constructible cohomology (with
respect to an algebraic stratification).
Also let $\Perv(X)\subset D^b_{\cs}(\CC_{X^{\an}})$ be the category of perverse sheaves on $X^{\an}$. Recall
that $\Perv(X)$ is the heart of a $t$-structure. We denote the corresponding cohomology
by ${}^pH(-)$.
Put
\[
\Sol_X:D^b(\Dscr_X)\r D(\CC_{X^{\an}}):\Mscr\mapsto \uRHom_{\Dscr^{\an}_X}(\Mscr^{\an},\Oscr_{X^{\an}}).
\]
By the
celebrated Riemann-Hilbert correspondence $\Sol$ restricts to an equivalence of
triangulated categories\footnote{$(-)^\circ$ denotes the opposite.}
\[
\Sol_X:D^b_{\rh}(\Dscr_X)\r D_{\cs}(\CC_X)^\circ
\]
which sends the natural $t$-structure on $D^b_{\rh}(\Dscr_X)$ to the $d$-shifted perverse one
on $D_{\cs}(\CC_X)^\circ$ and hence
restricts further to an equivalence of abelian categories
\[
\Sol_X[d]:\Mod_{\rh}(\Dscr_X)\r \Perv(X)^\circ
\]
(see e.g. \cite[Proposition 4.7.4, Corollary 7.2.4, Theorem 7.2.5]{hottabook}).

\begin{lemma} \label{lem:solcompatibility}
Let $f:Y\r X$ be a morphism between smooth quasi-compact separated schemes over $\CC$ 
of pure dimensions $d_Y$, $d_X$ and let $Lf^\ast$ be the unshifted $\Dscr$-module pullback (denoted by $Lf^\circ$ in \cite{Borel}).
Then
\begin{equation}
\label{eq:solcompat}
\Sol_Y\circ Lf^\ast=Lf^{\an,\ast}\circ \Sol_X
\end{equation}
as functors $D^b_{\rh}(\Dscr_X)\r D^b_{\cs}(\CC_{X^{\an}})$.
\end{lemma}
\begin{proof} Put  $d_{Y,X}=d_Y-d_X$.
Let $\DR_X:D^b_{\rh}(\Dscr_X)\r D^b_{\cs}(\CC_{X^{\an}})$ be the De Rham functor \cite[\S4]{hottabook} and let $f^!:D^b(\Dscr_X)\r D^b(\Dscr_Y)$ be the \emph{shifted} $\Dscr$-module pullback functor; i.e. $f^!=Lf^\ast[d_{Y,X}]$. By \cite[\S14.5(4)]{Borel},
\[
\DR_Y\circ f^!=f^{\an,!}\circ \DR_X.
\]
Moreover, by \cite[Prop 4.7.4]{hottabook}, we have $\DR_X\cong \Sol_X(\DD_X(-))[d_X]$ where
$\DD_X=\uRHom_{\Dscr_X}(-,\Dscr_X[d_X])$ and by \cite[Corollary 4.6.5]{hottabook},
$\DD_X^{\an}(\Sol_X (-)[d_X])=\Sol_X(\DD_X (-))[d_X]$ where $\DD_X^{\an}:=\uRHom(-,\CC_{X^{\an}}[2d_X])$ is
the Verdier duality functor. Finally $f^{\an,!}$ and $Lf^{\an,\ast}$ are related in the usual
way by $Lf^{\an,\ast}=\DD^{\an}_Y\circ f^{\an,!}\circ \DD^{\an}_X$. By combining these ingredients
one obtains the formula \eqref{eq:solcompat}.
\end{proof}
\begin{remark} If $f:Y\r X$ is a closed immersion then \eqref{eq:solcompat} says
informally that $\Dscr$-module pullback corresponds to restriction of solutions.
\end{remark}
\subsection{The GKZ perverse sheaf}\label{subsec:gkzperverse}

By Theorem \ref{rem:holonomic}, $\MMa(\alpha)\in \Mod_{\rh}(\Dscr_X)$. We put $\Ma(\alpha)=\Sol_{\TT^\ast}(\MMa(\alpha))[d]$.
We call $\Ma(\alpha)$ the \emph{GKZ perverse sheaf}, with parameter~$\alpha$. For further reference we recall the following results.
\begin{proposition} \label{prop:Palpha}
Assume that $\alpha,\alpha'$ are non-resonant.
\begin{enumerate}
\item
$\Ma(\alpha)$ is a simple 
perverse sheaf. Moreover, $\Ma(\alpha)\cong
\Ma(\alpha')$  
if and only if $\alpha-\alpha'\in
N$.
\item The analogous statements in the category of $\Dscr$-modules hold for $\MMa(\alpha)$.
\end{enumerate}
\end{proposition}
\begin{proof}
The statements about perverse sheaves and about $\Dscr$-modules are
  equivalent by the Riemann-Hilbert correspondence.  The first claim
  (for perverse sheaves) is \cite{GKZEuler} (for the explicit
  statement see \S.4.7. Proof of Theorem 2.11 in loc.cit.). 
(See also \cite[Theorem 4.1]{SchulzeWalther} in the context of $\Dscr$-modules.) 
   The
  second claim (for $\Dscr$-modules) is \cite[Corollary 2.6]{Saito},
  see also \cite[Thm 6.9.1]{Dwork}, 
\cite[Theorem 2.1]{Beukers2}.
\end{proof}

We may deduce the corresponding claims for $\iPP(\alpha)$ (see \S\ref{subsec:splitting}).
\begin{proposition}\label{prop:iPPgeneral}
The following holds for $\iPP(\alpha)=\iota^*\MMa(\alpha)$.
\begin{enumerate}
\item 
$L\iota^*\MMa(\alpha)=\iPP(\alpha)$.
\item
$\iPP(\alpha)$ has regular singularities.
\item
If $\alpha$ is non-resonant then $\iPP(\alpha)$ is a simple $\Dscr_{T^\ast}$-module. 
\item If \label{it4:iPPgeneral}
$\alpha,\alpha'$  are  non-resonant and
$\alpha-\alpha'\in
N$ then
$\iPP(\alpha)\cong
\iPP(\alpha')$.
\end{enumerate}
\end{proposition}
\begin{proof}
The properties (except for (1)) are  consequences of the preceding results.
\begin{enumerate}
\item This follows from Lemma \ref{lem:inverse} below.
\item See Remark \ref{rem:compat} and Theorem \ref{rem:holonomic}. 
\item This follows from Proposition \ref{prop:Palpha}(2) and \eqref{eq:iota}.
\item 
Let $\alpha$, $\alpha'$ be as in the statement of this proposition. By Proposition  \ref{prop:Palpha}(2)
we then have $\MMa(\alpha)\cong \MMa(\alpha')$. Hence $\iota^\ast\MMa(\alpha)\cong \iota^\ast\MMa(\alpha')$.
\qedhere 
\end{enumerate}
\end{proof}

\begin{lemma}\label{lem:generaliP} Put ${}^p \iota^\ast={}^pH^0(L\iota^{\an,\ast}):\Perv(\TT^\ast)\r \Perv(T^\ast)$. Then one has
\begin{equation}
\label{eq:ip}
\iP(\alpha):={}^p\iota^\ast P(\alpha)=\Sol_{T^*}(\iPP(\alpha))[\dim T].
\end{equation}
Moreover, $\iP(\alpha)$ satisfies the analogues of Proposition \ref{prop:iPPgeneral} (3)(4).
\end{lemma}
\begin{proof} The formula \eqref{eq:ip} may be deduced from Lemma \ref{lem:solcompatibility} using Proposition \ref{prop:iPPgeneral}(1).
The other claims follow from  Proposition \ref{prop:iPPgeneral} by the Riemann-Hilbert correspondence.
\end{proof}
\begin{corollary}\label{cor:intermediateextension} 
If $j:T^*\setminus V(E_A) \hookrightarrow T^*$ is the embedding then 
$j^{\an,*}\iP(\alpha)=\Sol_{T^*\setminus V(E_A)}(j^*\iPP(\alpha))[\dim T]$. 
 If $\alpha$ is non-resonant 
then
$\iP(\alpha)\cong j_{!*}(j^*\iP(\alpha))$, where
 $j_{!*}$ is the
intermediate extension. 
\end{corollary}

\begin{proof}
The first claim follows from Lemma \ref{lem:inverse}. 
For the second claim first note that $\iP(\alpha)$ is simple by Lemma \ref{lem:generaliP}. 
As  $\iP(\alpha)$ is not supported on $V(E_A)$, the conclusion then follows from \cite[\S4.3]{BBD} (see also \cite[Proposition 8.2.5(i), Corollary 8.2.10]{hottabook}). 
\end{proof}

\subsection{Parametric descent of GKZ systems}
\label{subsec:parametricdescent}
The results in this section, which we will not use in the sequel, were announced in Remark \ref{rem:repar}.

It is clear from the definition of $\Pscr(\alpha)$ (see \eqref{eq:GKZ}) that we may in fact define a relative $\Dscr$-module $\Pscr$ over $\TT^\ast\times \mathfrak{h}\r \mathfrak{h}$ whose fiber in $\alpha$ is equal to $\Pscr(\alpha)$. However Proposition \ref{prop:Palpha}(2) strongly suggests 
we should be able to descend the restriction of $\Pscr$ to $\mathfrak{h}^{\nres}$ to a relative $\Dscr$-module for the projection  $\TT^\ast\times H^{\nres}\r H^{\nres}$ whose fibers are still $\Pscr(\alpha)$. This is not possible algebraically but it is possible analytically.
\begin{proposition}
\label{prop:descent}
There exists a coherent $\Dscr^{\text{an}}_{\TT^\ast\times H^{\nres}/H^{\nres}}$-module $\tilde{\Pscr}$ such that the fiber of $\tilde{\Pscr}$ over $h\in H^{nres}$ is equal
to $\Pscr^{\text{an}}((\log h)/(2\pi i))$.
\end{proposition}

We will start with an algebraic result.
Put
\begin{equation*}
\MMa=(\Dscr_{\TT^\ast}\boxtimes_\CC \Oscr_{\mathfrak{h}})\biggr/\left(\sum_{l\in X(T^*)}\Dscr\Box_l+\sum_{\phi\in\hf^*}\Dscr(E_\phi-\phi)\right).
\end{equation*}
Then for $\alpha\in \mathfrak{h}$ we have $\Pscr(\alpha)\cong \Pscr_\alpha$.

\medskip

For $v\in \mathfrak{h}$ denote by $\tau_v$ the translation by $v$ on $\mathfrak{h}$. We use the same symbol for the corresponding translation on $\TT^\ast\times \mathfrak{h}$. We put a $\mathfrak{h}$-grading on the sections of $\Dscr_{\TT}$ such that $|\partial_i|=-a_i$. We then have for a section $f$ of $\Dscr_{\TT^\ast}$:
\[
[E_\phi,f]=\langle\phi,|f|\rangle f.
\]
From this it is easy to see that there are well-defined $\Dscr_{\TT^\ast}$-morphisms
\[
(\cdot f):\Pscr(\alpha)\r \Pscr(\alpha-|f|):\bar{D}\mapsto \overline{Df}
\]
which may be obtained as restrictions of the $\Dscr_{\TT^\ast}\boxtimes_\CC \Oscr_{\mathfrak{h}}$-morphism given by
\[
(\cdot f):\Pscr\r \tau_{-|f|}^\ast \Pscr : \bar{D}\mapsto \overline{Df}.
\]

\begin{proposition}\label{eq:family}
The morphism $(\cdot \partial_i):\Pscr\r \tau_{a_i}^\ast\Pscr$ becomes an isomorphism
when pulled back to\footnote{Note that while $\mathfrak{h}^{\nres}$ is not an open subscheme of $\mathfrak{h}$ it is still a perfectly good
noetherian affine scheme, albeit not of finite type.}  $\mathfrak{h}^{\nres}$.
\end{proposition}

\begin{proof}
This follows from the proof of \cite[Theorem 2.1]{Beukers2}. Let $i_\eta:\eta\r \TT^\ast$ be the generic point of $\TT^\ast$.
For any specific non-resonant $\alpha$ Beukers constructs an element $P_\alpha\in i^\ast_\eta\Dscr_{\TT^\ast}$ such that 
\begin{equation}
\label{eq:Palpha}
P_\alpha \partial_i \equiv 1 \mod \text{(relations of $\Pscr(\alpha)$)}.
\end{equation}
Going through the proof one sees:
\begin{enumerate}
\item The constructed $P_\alpha$ is actually a section of $\Dscr_{\TT^\ast}$.
\item The construction of $P_\alpha$ is polynomial in $\alpha$ 
except that one has to invert a finite number of times the evaluation of a
 hyperplane in $\mathfrak{h}\setminus \mathfrak{h}^{\nres}$ on $\alpha$.
\end{enumerate}
To check these assertions consult the first display on p.34 in loc.\ cit.

\medskip

If follows that one may consider the
$\alpha$ as variables and so one obtains a section $P$ of 
$\Dscr_{\TT^\ast}\boxtimes_\CC \Oscr_{\mathfrak{h}^{\nres}}$
which yields $P_\alpha$ when restricted to $\TT^\ast \times \alpha$ such that one 
furthermore has
\[
P \partial_i \equiv 1 \mod \text{(relations of $\Pscr$)}.
\]
Finally we may assume that $P$ is homogeneous for the $\mathfrak{h}$-grading (by dropping the terms not of degree $-|\partial_i|$). Hence 
if $\Pscr^{\nres}$ denotes the pullback of $\Pscr$ to $\TT^\ast \times \mathfrak{h}^{\nres}$ then we have morphisms
\[
\Pscr^{\nres}\xrightarrow{(\cdot \partial_i)} \tau^\ast_{a_i} \Pscr^{\nres} \xrightarrow{(\cdot P)} \Pscr^{\nres}
\]
whose composition is the identity. It now suffices to invoke Lemma \ref{lem:algebraic} below.
\end{proof}
\begin{lemma}
\label{lem:algebraic}
Let $A$ be a ring and let $M$ be a noetherian left $A$-module. Let $\phi$ be an automorphism of $A$ and assume that there are $A$-module morphisms
\[
M\xrightarrow{u} {}_\phi M\xrightarrow{v} M
\]
whose composition is the identity. Then $u$, $v$ are isomorphisms.
\end{lemma}
\begin{proof}
$uv$ is an idempotent morphism ${}_\phi M\r {}_\phi M$ whose image is isomorphic to $M$ (as $vu=\id$). Hence we obtain an isomorphism ${}_\phi M\cong M\oplus X$ for some left $A$-module $X$. Or $M\cong {}_\theta M\oplus Y$ for $\theta:=\phi^{-1}$ and $Y:={}_\theta X$. Iterating we get
\[
M\cong Y\oplus {}_\theta Y\oplus {}_{\theta^2} Y\oplus \cdots \oplus {}_{\theta^n} M 
\]
where $Y\subset Y\oplus {}_\theta Y\subset \cdots$ represents an infinite ascending chain of submodules of $M$, contradicting the noetherianity of $M$.
\end{proof}
\begin{proof}[Proof of Proposition \ref{prop:descent}] The isomorphisms  $(\cdot \partial_i):\Pscr^{\nres}\r \tau_{a_i}^\ast\Pscr^{\nres}$   exhibited in Proposition \ref{eq:family}
provide descent data for the action of $Y(H)\subset \mathfrak{h}$ on
$\TT^\ast\times\mathfrak{h}^{\nres}$ by translation. This descent cannot be done in the Zariski topology (cf. Remark \ref{rem:trivialcase} below) but it can be done in in the analytic topology since for this topology the action is discrete. 
This proves the result.
\end{proof}
\begin{remark} \label{rem:trivialcase} 
The relative $\Dscr$-module $\tilde{\Pscr}$
  exhibited in Proposition \ref{prop:descent} is not algebraic.  In
  fact the situation is already interesting in the case that $T=1$,
  $\dim H=\dim \TT=1$. In that case 
  $H\times \TT^\ast\cong \CC^\ast\times \CC^\ast$ and the underlying
  $\Oscr^{\text{an}}_{\TT^\ast\times H^{\nres}}$-module of
  $\tilde{\Pscr}$ is the restriction of the analytic line bundle
  $\Lscr$ on $\TT^\ast\times H$ which is locally over an open contractible subset $U$ of $H$ of the form 
\[
\Lscr_U:=\Lscr\mid (\TT^\ast\times U)=x^{\hat{h}}\Oscr^{\an}_{\TT^\ast\times U}
\]
where $\hat{h}=(\log h)/(2\pi i)\in \Gamma(U,\Oscr^{\an}_U)$, and where $x^{\hat{h}}$ is a symbol satisfying $x^{\hat{h}+1}=x x^{\hat{h}}$
(so that $\Lscr_U$ is canonically independent of the chosen branch of $\log h$). 
One checks that $\Lscr$ has no global section so it is not algebraic. In fact $c_2(\Lscr)\in H^2(\TT^\ast \times H,\ZZ)\cong \ZZ$ is non-trivial
so $\Lscr$ is not flat. Hence
  it cannot be made into a module over
  $\Dscr^{\operatorname{an}}_{\TT^\ast\times H}$, as one perhaps naively might hope.
\end{remark}
 \section{Solutions of the GKZ system}
\subsection{Mellin-Barnes solutions}
\label{sec:MB}
In \cite{Beukers} Beukers
shows that so-called {\em
  Mellin-Barnes integrals} 
 satisfy the GKZ 
system. Let $\gamma\in \CC^d=Y(\TT)_\CC=X(\TT^\ast)_\CC$ be such that
$\alpha=A(\gamma)$. 

Let us first recall the definition in loc.\ cit. For $\sigma\in Y(T)_\CC$  such that\footnote{It is possible to choose suitable $\sigma$ except when $\Re\gamma_i\in \NN$ and $b_i=  0$ for some $i$. Below
we will not care about this case. See Convention \ref{convention}.
}
 $\Re(\gamma_j+\langle b_j,\sigma\rangle)\not\in \NN$ 
the Mellin-Barnes integral is formally defined
as
\begin{equation}
\label{eq:formalMB}
M(v_1,\dots,v_d)=\int_{\sigma+i Y(T)_\RR}\prod_{j=1}^d\Gamma(-\gamma_j-\langle b_j,s\rangle)v_j^{\gamma_j+\langle b_j, s\rangle}ds.
\end{equation}
The condition on $\sigma$ is to guarantee that
the integrand does not have poles on the integration domain. For the (in)dependence of $\sigma,\gamma$ see Lemma \ref{lem:independenceofsigma} below.

\medskip

Differentiating $\eqref{eq:formalMB}$ under the integral sign with respect to the $v_i$
yields the relation $(E_\phi-\phi(\alpha))M=0$ for $\phi\in \mathfrak{h}^\ast$, which is part of the GKZ system.

\subsubsection{Making the Mellin-Barnes integral single-valued}\label{subsubsec:MBsingle}
Note that $M$ is defined on~$\TT^*$ (more precisely on its region of convergence, see \S\ref{sec:convergencedomain} below) but it is multi-valued due to the fact that the exponentials $v_j^u:= e^{\log(v_j) u}$ are multi-valued. To make $M$
single-valued we write $v=e^{2\pi i\hat{v}}$ for $\hat{v}\in \Lie \TT^\ast=Y(\TT^\ast)_\CC=X(\TT)_\CC$
and we express the integrand in terms of $\hat{v}$. 
Writing the single-valued version of $M$ as $\hat{M}$, i.e.
$\hat{M}(\hat{v})=M(v)$, we have
\begin{equation}
\label{eq:mbmv}
\hat{M}(\hat{v})=\int_{\sigma+i Y(T)_\RR} e^{2\pi i(\langle \hat{v},\gamma\rangle+\langle B\hat{v},s\rangle)}
\prod_{j=1}^d\Gamma(-\gamma_j-\langle b_j,s\rangle)
ds.
\end{equation}

\subsubsection{Convergence domain}
\label{sec:convergencedomain}
By \cite[Corollary
4.2]{Beukers} the integral defining $\hat{M}$ 
converges absolutely if\footnote{Formally the description of the convergence domain in \cite{Beukers} is slightly different from ours. However under the standing Assumption \ref{ass:setting} both descriptions are equivalent.}
 $B(\operatorname{Re} \hat{v})\in (1/2)\Sigma$.

\subsubsection{Other GKZ relations}\label{subsubsec:otherGKZrelations}
\begin{theorem}\protect{\cite[Theorem 3.1]{Beukers}}\label{thm:beukers}
If 
\begin{equation}
\label{eq:beukers}
\Re (\gamma_i + \langle b_i,\sigma\rangle)<0\quad \text{for all $1 \leq i \leq d$}, 
\end{equation}
then $M$ also satisfies
$\Box_l M = 0$ for all $l \in L$.
\end{theorem}
Note: in loc. cit. $\gamma$ is assumed to be real, but the proof works equally well for complex $\gamma$.
\begin{lemma}\label{rem:condition} For a given $\alpha$, there exists $\gamma\in \CC^d$ 
with $A\gamma=\alpha$ and $\sigma\in Y(T)_\CC$ such that \eqref{eq:beukers} holds if and only if $\Re\alpha$ is in $\RR_{<0}A$.
\end{lemma}
\begin{proof} We need to determine when there exists $\gamma\in \CC^d$ with $A\gamma=\alpha$ and $\sigma\in Y(T)_\CC$ such that $\forall i:\Re (\gamma_i + \langle b_i,\sigma)\rangle<0$.
Now we have $\sum_i (\gamma_i+\langle b_i,\sigma\rangle)a_i=\sum_i\gamma_ia_i=\alpha$. In other words we may assume $\sigma=0$ and we have to look for $\gamma$ satisfying $\forall i:\Re\gamma_i<0$ and $\sum_i \gamma_i a_i=\alpha$. Such $\gamma$ exists if and only of $\alpha$ is as in the statement of lemma.
\end{proof}
\begin{lemma}\label{lem:independenceofsigma}
For given $\alpha$ with $\Re \alpha\in \RR_{<0}A$, let $\sigma\in Y(T)_\CC,\gamma\in \CC^d$ with $A\gamma=\alpha$. 
Assume that $\sigma,\gamma$ satisfy \eqref{eq:beukers}. 
Then the corresponding Mellin-Barnes integral only depends on $\alpha$.
\end{lemma}
\begin{proof}
  We first fix $\gamma$ and vary $\sigma\r \sigma'$ (preserving \eqref{eq:beukers}). Then the integral
  does not change by the fact that the domain
  $\{\sigma+t(\sigma'-\sigma)+iY(\TT)_\RR\mid t\in [0,1]\}$ does not
  contain any pole of the integrand. Now we keep $\sigma$ fixed and we vary
  $\gamma\r \gamma'$ (keeping $A\gamma'=\alpha$) such that
  $(\sigma,\gamma')$ still satisfies \eqref{eq:beukers}. Then we can find  $\sigma'$  such that
  $\forall i:(\gamma_i+\langle b_i,\sigma\rangle)_i=(\gamma'_i+\langle
  b_i,\sigma'\rangle)_i$ (as $A(\gamma-\gamma')=0$ and hence $\gamma-\gamma'\in B^*Y(T)$).
  Hence $(\gamma',\sigma')$ satisfies \eqref{eq:beukers}.  Then we use
  the already established independence of $\sigma$.
\end{proof}

\begin{convention}\label{convention}
For $\Re \alpha \in \RR_{<0}A$ we 
write $M^{\alpha}$ (or $\hat{M}^{\alpha}$) 
for the MB-integral corresponding to $\gamma\in Y(\TT)_\CC\cong \CC^d$ such that $\alpha=A\gamma$, 
and $\sigma\in Y(T)_\CC$ such that $\forall i:\Re(\gamma_i+\langle b_i,\sigma\rangle)<0$.  
This notation is justified by Lemma \ref{lem:independenceofsigma}, as the MB integral is independent
of $\sigma$, $\gamma$ satisfying \eqref{eq:beukers} for $A\gamma=\alpha$ with $\alpha$ fixed. 
If we locally fix $\alpha$  then 
 we usually write $M:=M^{\alpha}$ (or $\hat{M}:=\hat{M}^{\alpha}$). 
\end{convention}
The next lemma and the ensuing remark will be used below.
\begin{lemma}\label{lem:Hartogs} 
The integral  $\hat{M}^{\alpha}$ depends holomorphically on the parameter~$\alpha$ with $\Re \alpha\in\RR_{<0}A$.  
\end{lemma}
\begin{proof} 
By Lemma \ref{lem:independenceofsigma}, it suffices to fix $\sigma$ satisfying \eqref{eq:beukers} for some $\gamma$ and prove that $\hat{M}^\alpha$ depends holomorphically on the parameter $\gamma$ on the domain  where $\eqref{eq:beukers}$ is satisfied for $\gamma$. We then write $\hat{M}^\gamma$ instead of $\hat{M}^\alpha$. 
To accomplish this we use Morera's theorem. Let $D$ be the domain as in the lemma. Then we need to check that the integral of $\hat{M}^\gamma$ as a function of $\gamma$ vanishes when integrating over a closed piecewise $C^1$-curve $C$ in $D$. 
At this point we apply the estimate \cite[\S 4]{Beukers} of the absolute value of the integrand of $\hat{M}^\gamma$, which allows us to apply the Fubini theorem, and exchange the integrals over $C$ and over $\sigma+iY(T)_\RR$. As the integrand of $\hat{M}^\gamma$ is an analytic function of $\gamma$ (on $D$), its integral over $C$ vanishes. Thus, the condition of Morera's theorem holds.
\end{proof}
\begin{remark} \label{rem:Hartogs}
Below we will extend Mellin-Barnes integrals analytically outside the convergence domain of their defining integrals. 
By the generalized Hartogs' lemma  \cite[Theorem 5, Ch. VII]{Bochner-Martin}, the conclusion of Lemma \ref{lem:Hartogs} remains valid for such extended functions.
\end{remark}

\subsubsection{Restriction of Mellin-Barnes solutions}
\label{subsec:descentMB}
From \eqref{eq:mbmv} one obtains
\begin{equation}
\label{eq:transformationlaw}
\hat{M}(\hat{v}+A^\ast w)=e^{2\pi i \langle w,\alpha\rangle} \hat{M}(\hat{v}).
\end{equation}
We use the
splitting $\iota$ from \S\ref{subsec:splitting}, which induces
a corresponding splitting of  $B:X(\TT)_\CC\r X(T)_\CC$. 
Then $\hat{M}\circ \iota$ becomes a function which is defined on $(1/2)\Sigma\times iX(T)_\RR$.

Of course $\hat{M}\circ \iota$ depends on the splitting $\iota$. With $\delta$ being the difference of splittings $\iota$ and $\iota'$ as
in Remark \ref{rem:difference}
we find
\[
(\hat{M}\circ \iota')(\hat{v})=e^{2\pi i \langle \delta\hat{v},\alpha\rangle} (\hat{M}\circ \iota)(\hat{v}).
\]
So $\hat{M}\circ \iota'$ and $\hat{M}\circ \iota$ differ by the character $e^{2\pi i\langle \delta(-),\alpha\rangle}$ of $X(T)_\CC$ which agrees nicely
with Remark \ref{rem:difference}.

We now fix $\iota$ and write $\hat{M}$ for $\hat{M}\circ \iota$. We will also think of $\hat{M}$ as a multi-valued function $M$ on $X(T)_\CC/X(T)\xrightarrow[\cong]{e^{2\pi i-}} T^\ast$.
It will always be clear from the context on which spaces $M$ and $\hat{M}$ are defined.

\subsubsection{Basis of solutions}\label{subsubsec:basisofsolutions} We assume that \(W\) 
is quasi-symmetric.  We give a basis of
solutions 
of the restricted  GKZ system $\iPP(\alpha)$ on a dense open subset of $T^\ast\setminus V(E_A)$
 (see \S\ref{subsec:GKZdislocus}).
To make the solutions univalued we will work on a dense open subset of the corresponding  covering space
\[
(X(T)_\RR\setminus \Hscr)\times i X(T)_\RR\subset X(T)_\CC\setminus (\kkappa+\Hscr_\CC)\longrightarrow T^\ast\setminus V(E_A)
\]
which by \S\ref{sec:hyperplane} has a cell decomposition
\[
(X(T)_\RR\setminus \Hscr)\times  i X(T)_\RR=\bigcup_{C\in \Cscr^0} C\times i X(T)_\RR.
\]
\begin{proposition}\label{prop:basis}
Assume that $\alpha\in Y(H)_\CC$ is non-resonant and $\Re\alpha\in \sum \RR_{<0}A$.
For \(\chi \in X(T)\) put 
\[
\hat{M}_{\chi}(x):=\hat{M}(x-\chi)
\]
(see Convention \ref{convention}). Then the collection of functions
\begin{equation}
\label{eq:MBnotation}
\Mscr_C:=\left\{\hat{M}_{\chi} | \chi \in \mathcal{L}_{C}\right\}
\end{equation}
where $\Lscr_C$ is as defined in \S\ref{sec:hyperplane},
gives a basis of solutions of the pullback of the GKZ
system $\iPP(\alpha)$ to $C \times i X(T)_\RR$ for $C\in \Cscr^0$.
\end{proposition}
\begin{proof}
Let $C\in \Cscr^0$. Recall that $\Lscr_C=(\nu+\Delta)\cap X(T)$ for $\nu\in C$.
If \(\chi \in \mathcal{L}_{C}\)
then \(\nu \in \chi+\Delta\)
(as $\Delta=-\Delta$ by unimodularity) and hence \(C \subset \chi+\Delta.\)
Moreover, as $C\in\Cscr$ is a chamber 
 it does not lie on the boundary of $\chi+\Delta$. 
By \S\ref{sec:convergencedomain}, \(\hat{M}\) is defined on \( (1 / 2) \Sigma \times i X(T)_{\mathbb{R}}.\) 
Hence $\hat{M}_\chi$ is defined on $C\times i X(T)_\RR$.

The set \(\left\{\hat{M}_{\chi} | \chi \in \mathcal{L}_{C}\right\}\)
is linearly independent by \cite[Lemma 5.3.1]{NilssonPassareTsikh}
(see also \cite[Proposition 4.3)]{Beukers}).  As for non-resonant
$\alpha$ the rank of the GKZ system equals
\(\left|\mathcal{L}_{C}\right|=D\)
(see \S\ref{subsec:nonresonance}), the conclusion follows by Theorem \ref{thm:beukers} (using the assumption $\alpha\in \RR_{<0} A$, and recalling the Convention \ref{convention}).
\end{proof}

\subsection{Power series solutions}\label{subsec:powerseries}
The GKZ system has a formal solution given by \cite{GKZhyper} 
\[
\Phi_\gamma(v_1,\ldots,v_d)=\sum_{l\in L}\prod_{i=1}^d\frac{v_i^{(B^*l)_i+\gamma_i}}{\Gamma((B^*l)_i+\gamma_i+1)}.
\]
This function is multi-valued but it can be made single-valued by considering it as a function $\hat{\Phi}_\gamma$ on $X(\TT)_\CC$, as we did for $M$. We have
the formula
\[
\hat{\Phi}_\gamma(\hat{v})=\sum_{l\in L} \frac{e^{2\pi i\langle \hat{v},B^*l+\gamma\rangle}}{\prod_{i=1}^d
\Gamma((B^*l)_i+\gamma_i+1)}.
\]

Let 
 $I=\{i_1,\dots,i_n\}$ be a subset of $\{1,\dots,d\}$ such that
  $b_{i_1},\dots,b_{i_n}$ are linearly independent in $X(T)_\RR$. We choose $\gamma_I$ as in
\S\ref{sec:MB} (i.e. $\alpha=A(\gamma_I)$) such that $\gamma_{I,i}\in \ZZ$
for $i\in I$. This gives us $|\det((b_i)_{i\in I})|$ choices for
$\gamma_I$ modulo $L$. 

\begin{lemma}\label{lem:convdom}\cite[Proposition 16.2]{Beukers3}\cite[\S 3.3, \S 3.4]{Stienstra}
Let $\rho\in \sum_{i\in I} \RR_{>0} b_i\subset X(T)_\RR$.  
Then $\Phi_{\gamma_I}$ converges on an open neighbourhood of 
\[
D_\rho:=\{v\in (\CC^*)^d\mid \exist {\tilde{\rho}}\in B^{-1}(\rho), 0<t< t_\rho:\forall i:|v_i|=t^{\tilde{\rho}_i}
\}\]
for a suitable $0<t_\rho\ll 1$.
\end{lemma}

\begin{corollary}\label{cor:convdom}Let $\rho\in \sum_{i\in I} \RR_{>0} b_i\subset X(T)_\RR$. 
Then $\hat{\Phi}_{\gamma_I}$ converges on an open neighbourhood of 
\[
\hat{D}_\rho:=\{\hat{v}\in \CC^d\mid \exist {\tilde{\rho}}\in B^{-1}(\rho), u>u_\rho :{\rm Im}\,\hat{v}=
u\tilde{\rho}\}\]
for a suitable $u_\rho\gg 0$.
\end{corollary}

For $\rho\in X(T)_\RR$ we define
\[
\Iscr_\rho=\left\{J\subset \{1,\dots,d\}\mid |J|=n,\,\RR\{b_i\mid i\in J\}=X(T)_\RR,  
\,\rho=\sum_{j\in J} \beta_j b_j\;\text{for }\beta_j>0\right\}.
\]
If $\Iscr_\rho$ is non-empty then we call $\rho$ a \emph{convergence direction}.

\medskip

We let $\hat{\Iscr}_\rho$ be the multiset with $\det((b_i)_{i\in I})$ copies of $I\in \Iscr_\rho$. To each copy we associate a distinct $\gamma_I$ (i.e. $\gamma_I$ does not depend only on $I\in \Iscr_\rho$, but on $I\in \hat{\Iscr}_\rho$). For simplicity, we denote $\Phi_I$, $\hat{\Phi}_I$ for $\Phi_{\gamma_I}$, $\hat{\Phi}_{\gamma_I}$, $I\in \hat{\Iscr}_\rho$.  
\begin{proposition}\cite{GKZhyper} (see also \cite[\S2]{Beukers})\label{prop:solutionsatinfinity} 
If \(\alpha\) is totally non-resonant and $\rho$ is generic (see \S\ref{sec:hyperplane}, Lemma \ref{lem:genericB}), then \(\left\{\Phi_{I} | I \in \hat{\mathcal{I}}_{\rho}\right\}\) is a basis of solutions of the GKZ system on an open neighbourhood of  
$D_\rho$.\footnote{The assumption of total non-resonance enters here, and guarantees that all solutions at infinity are logarithm free, which as it will be seen later, simplifies the computation of the monodromy.}
\end{proposition}

\subsubsection{Restriction of power series solutions}
\label{subsec:descentps}
Let us write  $\hat{\Phi}_\gamma$ for $\hat{\Phi}_\gamma\circ \iota$ where $\iota:X(T)_\CC\r X(\TT)_\CC$ is as in \S\ref{sec:MB}. 
Then one checks  for $x\in X(T)_\CC$
\[
\hat{\Phi}_\gamma(x)=\sum_{l\in Y(T)}\frac{e^{2\pi i\langle x,l+\iota{\gamma}\rangle}}{\prod_{i=1}^d\Gamma((B^*l)_i+\gamma_i+1)}
\]
where $\iota{\gamma}\in Y(T)_\CC$, denoting, by a slight abuse of notation,   $\iota:Y(\TT)_\CC\r Y(T)_\CC$ also the adjoint to $\iota$. 
Note that for $p\in X(T)$ 
\begin{equation}\label{eq:shiftofpowerseries}
\hat{\Phi}_\gamma(x+p)=e^{2\pi i\langle p,\iota{\gamma}\rangle}\hat{\Phi}_\gamma(x)
\end{equation}
as $\langle p,l\rangle\in \ZZ$ for $l\in Y(T)$.

\begin{corollary}\label{cor:convdomT}
Let the setting be as in Lemma \ref{lem:convdom}.
Then $\hat{\Phi}_{\gamma_I}$ converges on an open neighbourhood of 
\[
\hat{D}_\rho:=\{x\in X(T)_\CC \mid {\rm Im\,}x=u\rho\text{ for }u> u_\rho\}
\]
for a suitable $u_\rho\gg 0$.
\end{corollary}
\begin{proof}
By Corollary \ref{cor:convdom}, $\hat{\Phi}_{\gamma_I}$ converges for $x\in X(T)_\CC$ such that
\[
\Im(\iota(x))\in ]u_\rho,\infty[B^{-1}(\rho)
\]
for a suitable $u_\rho\gg 0$. 
This is equivalent to
\[
\iota(\Im x)\in B^{-1}(]u_\rho,\infty[\rho)
\]
and this is equivalent to $\Im x=B(\iota(\Im x))\in ]u_\rho,\infty[\rho$. 
\end{proof}
\section{Monodromy of the GKZ system in the quasi-symmetric case}
In this section assume throughout that $W$ is quasi-symmetric. We give a very explicit description of the representation of the fundamental groupoid given by the local system
determined by the GKZ system.
We do this by connecting the MB integral solutions
with the power series
solutions. This approach is similar to the one used in \cite{Beukers} from which we took our inspiration.
\subsection{The fundamental groupoid of the complement of a hyperplane arrangement}
\label{subsec:fundgd}
Let $(\Hscr,\Cscr,V,\ldots)$ be as in \S\ref{sec:genhyp}. There are a number of presentations available \cite{Delignetresses,KapranovSchechtman,HLSam,Salvetti} for the fundamental groupoid $\Pi_1(V_\CC\setminus \Hscr_\CC$). We will use the presentation from \cite{KapranovSchechtman}. 

\medskip

Let $C_1,C_2\in\Cscr^0$ be chambers with $\dim C_1\wedge C_2=n-1$. 
Denote $C_0=C_1\wedge C_2$ and let $H$ be an equation for the hyperplane in $\Hscr$ containing $C_0$ which is strictly positive on $C_2$. 

For every $C\in \Cscr^0$, choose $\rho_C\in C$ and put
$\rho_1:=\rho_{C_1}\in C_1$, $\rho_2:=\rho_{C_2}\in C_2$ and for
$\ell\in V\setminus H_0$ consider the path $\ggamma_\ell$ connecting $\rho_1$ to
$\rho_2$ via the following line segments
\[
\rho_1\r i\ell+\rho_1\r i\ell+\rho_2\r \rho_2.
\]
Then $\ggamma_\ell$, $\ggamma_{\ell'}$ are homotopic iff
$H_0(\ell), H_0(\ell')$ have the same sign. So
up to homotopy this gives us two paths
$\rho_1\r \rho_2$ in
$V_\CC\setminus \Hscr_\CC$. 
We pick $\ell$ such that $H_0(\ell)>0$ and write $\ggamma_{C_1C_2}=\ggamma_\ell$.
\begin{definition} Consider the  abstract groupoid $\Pi(\Hscr)$ with the following presentation:
\begin{enumerate}
\item An object for every $C\in \Cscr^0$.
\item A morphism $\ggamma_{C_1C_2}:C_1\r C_2$ for every $C_1,C_2\in \Cscr^0$ with $C_1\wedge C_2\neq \emptyset$, such that $\ggamma_{CC}=\id$.
\item Relations of the form $\ggamma_{C_1C_3}=\ggamma_{C_2C_3}\ggamma_{C_1C_2}$ for collinear (\S\ref{sec:genhyp})
  triples  $C_1$, $C_2$, $C_3\in \Cscr^0$. 
\end{enumerate}
\end{definition}
It is easy to see that $\Pi(\Hscr)$ is generated by $\ggamma_{C_1C_2}$ for couples $C_1$, $C_2$ which share a facet.
\begin{proposition} \cite[Proposition 9.11]{KapranovSchechtman}\label{prop:KSmonodromy} 
There is an equivalence of groupoids
\[
\Pi(\Hscr)\r \Pi_1(V_\CC\setminus \Hscr_\CC)
\]
sending $C\in \Cscr^0$ to $\rho_C\in C$ and $\ggamma_{C_1C_2}:C_1\r C_2$ such that $\dim C_1\wedge C_2=n-1$
to $\ggamma_{C_1C_2}:\rho_{C_1}\r \rho_{C_2}$.
\end{proposition}
\subsection{Fundamental groupoids of quotient spaces}
\label{subsubsec:equivariance}
If $\Mscr$ is a groupoid and $\Gscr$ is a group acting on $\Mscr$ then the semi-direct
product $\Mscr\rtimes \Gscr$ has the same objects as $\Mscr$ and is obtained by freely adjoining morphisms $g_m:m\r gm$ to $\Mscr$,
for $g\in \Gscr$, $m\in \Ob(\Mscr)$ subject to the relations
\begin{enumerate}
\item $e_m=\id_m$, for $e\in \Gscr$ the identity element.
\item $h_{gm}\cdot g_m=(hg)_m$ for $g,h\in \Gscr$, $m\in \Ob(\Mscr)$.
\item $g(f)\cdot g_m=g_n\cdot f$ for $g\in \Gscr$, $f:m\r n$ in $\Mscr$.
\end{enumerate}
It is easy to see that the construction of $\Mscr\rtimes \Gscr$ is
compatible with equivalences of groupoids.  If $\Gscr$ acts freely and
discretely on a topological space $Y$ then there is an equivalence of groupoids
\begin{equation}
\label{eq:quotientspace}
\Pi_1(Y)\rtimes \Gscr \cong 
\Pi_1(Y/\Gscr)
\end{equation}
 (see \cite[Chapter
11]{BrownRonald}\cite[\S6]{HLSam}). For use below we make this more concrete.

Recall that a $\Gscr$-equivariant local system $L$ on $Y$ is a local system equipped with \emph{descent data}, i.e. isomorphisms
\[
u_g:L\r g^\ast L
\]
for all $g\in \Gscr$, satisfying the standard cocycle condition. Then $L$ descends to a local system $\bar{L}$ on $Y/\Gscr$ such that $\pi^\ast(\bar{L})\cong L$ where
$\pi:Y\r Y/\Gscr$ is the quotient morphism.
\begin{lemma} \label{eq:concrete}
 The  $\Pi_1(Y)\rtimes \Gscr$ representation $\Lscr$ corresponding to $\bar{L}$ via the equivalence \eqref{eq:quotientspace} 
 is the following:
\begin{enumerate}
\item $\Lscr_y= L_y$ for $y\in Y=\Ob(\Pi_1(Y)\rtimes \Gscr)$. 
\item The $\Pi_1(Y)$-action on $\coprod_{y\in Y} \Lscr_y$ is obtained from the fact that $L$ is a local system on $Y$.
\item If $g\in \Gscr$ then the corresponding morphism $\Lscr(g_y):\Lscr_y=L_y\r \Lscr_{gy}=L_{gy}$ is obtained by specializing $u_g$ at $y$
(using the fact that $(g_\ast L)_y=L_{gy}$).
\end{enumerate}
\end{lemma}
Now let $V,\Hscr,\ldots$ be as above and assume that $V$ is equipped with an affine,  
$\Hscr$-preserving, group action by a group $\Gscr$.  In that case $V_\CC\setminus \Hscr_\CC$ and 
$\Cscr$ are of course also preserved. If $\Gscr$ acts freely and discretely then from 
Proposition \ref{prop:KSmonodromy} and \eqref{eq:quotientspace} 
we obtain equivalences of groupoids
\begin{equation}\label{eq:groupoideq}
\Pi(\Hscr)\rtimes \Gscr\xrightarrow{\cong} \Pi_1(V_\CC\setminus \Hscr_\CC)\rtimes \Gscr \xrightarrow{\cong} \Pi_1((V_\CC\setminus\Hscr_\CC)/\Gscr).
\end{equation}

\subsection{Statement of the main result}\label{subsec:mainresult}
We now let $\Hscr,\Cscr,\ldots$ have again their standard
meaning. After choosing a splitting $\iota:\TT\to T$ of
$B^\ast:T\r \TT$, the GKZ system defines a local system on
$(X(T)_\CC\setminus (\Hscr_\CC+\kkappa))/X(T)$ (cfr
\S\ref{subsec:GKZdislocus}). By \eqref{eq:groupoideq} 
we
have equivalences 
\begin{equation}
\label{eq:groupoids}
\Pi(\Hscr)\rtimes X(T)\xrightarrow[\cong]{} \Pi_1(V_\CC\setminus \Hscr_\CC)\rtimes X(T)\xrightarrow[\cong]{\text{translation}} \Pi_1((X(T)_\CC\setminus (\Hscr_\CC+\kkappa))/X(T)).
\end{equation}
In other words, the GKZ system yields a representation of $\Pi(\Hscr)\rtimes \Gscr$ which we will now describe explicitly. First we introduce some notation.
Let $C'<C$ with $\dim C=n$, $\dim C'=n-1$. Let $L\in \Hscr$ be the hyperplane spanned by $C'$ and assume that $L$ is represented by an equation
which is strictly positive on $C$. Then we put
\begin{equation}\label{eq:JC0C1}
J_{C'C}=\{i\in \{1,\ldots,d\}\mid L_0(b_i) >0\}.
\end{equation}
Below we write the analytic continuation along a path $\ggamma$ as $\ggamma(-)$.
\begin{theorem} \label{th:mainth1}
Assume that $\alpha\in Y(H)_\CC$ is non-resonant and satisfies $\Re\alpha\in \RR_{<0}A$ and let~$M$ be the representation of $\Pi(\Hscr)\rtimes X(T)$ corresponding 
to the local system given by the solutions of
 the GKZ system $\bar{\Pscr}(\alpha)$ via \eqref{eq:groupoids}.
\begin{enumerate}
\item For $C\in \Cscr^0$ we have
\[
M(C)=\Mscr_C
\]
(using the notation \eqref{eq:MBnotation}).
\item For $C_1,C_2\in \Cscr^0$ such that $\dim C_1\wedge C_2=n-1$ write $C_0=C_1\wedge C_2$. The map $\ggamma_{C_1C_2}:M(C_1)\r M(C_2)$
evaluated on $\hat{M}_\chi\in M(C_1)$ is given by
\begin{multline}
\label{eq:Mchiexpress}
M(\ggamma_{C_1C_2})(\hat{M}_\chi)=\\
\begin{cases}
\hat{M}_\chi&\text{if $\chi\in \Lscr_{C_1}\cap \Lscr_{C_2}$},\\
\sum_{\emptyset\neq\sJ\subset J_{C_0C_2}} (-1)^{|\sJ|+1}\left(\prod_{j\in \sJ}e^{-2\pi i \gamma_j}\right){\hat{M}_{\chi+\sum_{j\in \sJ}b_j}}
&\text{if $\chi\in \Lscr_{C_1}\setminus \Lscr_{C_2}$},
\end{cases}
\end{multline}
where  $\gamma$ is the unique element of $Y(\TT)_\CC$ such that $A\gamma=\alpha$ and $\iota \gamma=0$.
\item $M(\mu_C)(\hat{M}_\chi)=\hat{M}_{\chi+\mu}$ for $\mu\in X(T)$ and $\chi\in \Lscr_C$, $C\in \Cscr^0$.
\end{enumerate}
\end{theorem}
The proof of this theorem will be carried out in the remainder of this section.
\subsection{Reminder on the splitting}
\label{sec:splitting}
As already stated above, like in \S\ref{subsec:splitting}  we choose a splitting  $\iota:\TT\to T$ of $B^\ast:T\r \TT$ and we consider the pullback of the GKZ system under $\iota:X(T)_{\CC}\r  X(\TT)_{\CC}$ and we do the same for the Mellin-Barnes solutions (see \S\ref{subsec:descentMB}) and the formal power series solutions (see \S\ref{subsec:descentps}).
\subsection{Connecting MB solutions to power series solutions}\label{subsec:connecting}
\begin{proposition}
\label{eq:connecting}
Assume that $\alpha\in Y(H)_\CC$ is totally non-resonant and $\Re\alpha\in \RR_{<0}A$.
Let $\rho\in X(T)_\RR$ be generic  (c.f.\ \S\ref{sec:hyperplane}) and $C\in \Cscr^0$. We have on $\hat{D}_\rho\cap(C\times iX(T)_\RR)$ 
\begin{equation}\label{eq:MPhi}
\hat{M}_\chi=\sum_{I\in \hat{\Iscr}_\rho}e^{-2\pi i\langle \chi,\iota{\gamma_I}\rangle}{\hat{\Phi}}^s_I
\end{equation}
where $\hat{\Phi}^s_I=a \hat{\Phi}_I$  for suitable $a\in \CC^\ast$, depending only on $I$, $\alpha$ and $\rho$. 
\end{proposition}
\begin{proof}
We have
$\hat{M}^\alpha=\sum_{I\in \hat{\Iscr}_\rho}a_I\hat{\Phi}_I$
 for $a_I\neq 0$ ($\hat{M}^\alpha$ as in Convention \ref{convention}) 
by Propositions \ref{prop:basis}, \ref{prop:solutionsatinfinity} and  \eqref{eq:shiftofpowerseries}. 
Put  $\hat{\Phi}^s_I=a_I\hat{\Phi}_I$.
 Then \eqref{eq:MPhi} follows from the  definition of $M_\chi$ and \eqref{eq:shiftofpowerseries}.
\end{proof}
Below we write 
$\Pscr_{\rho}=\{\hat{\Phi}^s_I\mid I\in \hat{\Iscr}_\rho\}$.
\subsection{Monodromy}
The fundamental groupoid of $X(T)_\CC\setminus (\zeta+\Hscr_\CC)$ acts on Mellin-Barnes solutions by analytic continuation. Following
literally the equivalences in \eqref{eq:groupoids} we have to carry out the analytic continuation
for $\ggamma_{C_1C_2}+\kkappa$ for chambers $C_1$, $C_2$ sharing a face. However this turns out to
be technically slightly inconvenient. We therefore observe that $\ggamma_{C_1C_2}$ was
defined as $\ggamma_\ell$ for suitable $\ell$ and by taking $\|\ell\|\gg 0$ we may assume
that $\ggamma_{C_1C_2}$ and  $\ggamma_{C_1C_2}+\kkappa$ are homotopy equivalent.
So below we assume that~$\ell$ is taken in this way and we will do the analytic continuation
for $\ggamma_{C_1C_2}$, rather than for its translated version. Also for technical reasons we further assume that  $\ell$ is generic (c.f.\ \S\ref{sec:hyperplane}).
\begin{lemma}\label{lem:path}
  The path $\ggamma_{C_1C_2}$  (for $\|\ell\|\gg 0$ as was assumed above) intersects $\hat{D}_{\ell}$ and moreover stays within
$C_1\times iX(T)_\RR\cup \hat{D}_{\ell} \cup C_2\times iX(T)_\RR$. In particular, it intersects 
  the common convergence domain of $\Pscr_{\ell}$ as well
as the domains of definition of $\Mscr_{C_1}$ and $\Mscr_{C_2}$.
\end{lemma}
\begin{proof}
This follows from Corollary \ref{cor:convdomT}.
 \end{proof}
\begin{proof}[Proof of Theorem \ref{th:mainth1}]
  (1) is simply Proposition \ref{prop:basis}. (3) follows from the
  fact that $\mu\in X(T)$ acts on $X(T)_\CC$ by the corresponding
  translation $\tau_\mu$. Unravelling the descent data for the pullback 
  of $\bar{\Pscr}(\alpha)$ to $X(T)_\CC\setminus \Hscr_\CC$ we see by
  Lemma \ref{eq:concrete}(3) that
  $M(\mu_C)(\hat{M}_\chi)=\tau^\ast_{-\mu}(\hat{M}_\chi)=\hat{M}_\chi\circ
  \tau_{-\mu}=\hat{M}_{\chi+\mu}$.

Now we concentrate on (2). 
For $\hat{M}_\chi\in \Mscr_{C_1}$ we
set
$
\hat{{M}}^{C_{1}C_2}_\chi:=\ggamma_{C_1,C_2}(\hat{M}_\chi)
$ (on $C_2\times i X(T)_\RR$).

If  $\chi\in \Lscr_{C_1}\cap \Lscr_{C_2}$ then $\hat{M}_\chi$ is defined on $(C_1\cup C_0\cup C_2)\times iX(T)_\RR$ and so no analytic continuation is necessary.
This proves the first case of (2).

Now we proceed to the second case: $\chi\in \Lscr_{C_1}\setminus\Lscr_{C_2}$.
The fact that $\chi+\sum_{j\in \sJ}b_j\in \Lscr_{C_2}$ for $\sJ\neq \emptyset$, so that \eqref{eq:Mchiexpress} is well-defined, follows from Lemma \ref{lem:inC2} below.

To continue the proof, we first reduce to the case that
$\alpha$ is totally non-resonant.  Assume that $\alpha$ is
non-resonant. Then there exists a sequence of totally non-resonant
$(\alpha^i)_i$ with $\Re \alpha^i\in \RR_{<0} A$ which converges to $\alpha$ (as the set
of totally non-resonant parameters is dense in the set of non-resonant
parameters). By Remark \ref{rem:Hartogs}, we may compute the expansion
$\hat{M}_\chi^{C_1C_2}$ for
$\chi\in \Lscr_{C_1}\setminus \Lscr_{C_2}$ in the basis $\Mscr_{C_2}$, for
$\alpha$, as a limit of the corresponding expansions for
$\alpha^i$ (note that we have to adapt\footnote{We may think of $\gamma$ as $\jota\alpha$ where $\jota:Y(H)\to Y(\TT)$ is the splitting of $A$ such that $\iota\jota=0$.} $\gamma$ to $\alpha^i$). Thus, if \eqref{eq:Mchiexpress} holds for $\alpha^i$, it
also holds for $\alpha$.

\medskip

From now we assume that $\alpha$ is totally non-resonant.  Put $J=J_{C_0C_2}$. Analytically continuing the elements
of $\Mscr_{C_1}$, $\Mscr_{C_2}$ along the path $\ggamma_{C_1C_2}$ we can write them as linear combinations
of elements in $\Pscr_{\ell}$ (recall that we have assumed that $\ell$ is generic so that Proposition \ref{eq:connecting} applies with $\rho=\ell$). Hence by \eqref{eq:MPhi} we have
to prove
\[
\sum_{I\in \hat{\Iscr}_{\ell}}e^{-2\pi i\langle \chi,\iota{\gamma_I}\rangle}{\hat{\Phi}}^s_I=
\sum_{\emptyset\neq\sJ\subset J} (-1)^{|\sJ|+1}\prod_{j\in \sJ}e^{-2\pi i \gamma_j}\sum_{I\in \hat{\Iscr}_{\ell}}
e^{-2\pi i\langle \chi+\sum_{j\in \sJ}b_j,\iota{\gamma_I}\rangle}\hat{\Phi}^s_I.
\]
Equivalently, for all $I\in \hat{\Iscr}_{\ell}$
\[
e^{-2\pi i\langle \chi,\iota{\gamma_I}\rangle}=
\sum_{\emptyset\neq\sJ\subset J} (-1)^{|\sJ|+1}\left(\prod_{j\in \sJ}e^{-2\pi i \gamma_j}\right)e^{-2\pi i\langle \chi+\sum_{j\in \sJ}b_j,\iota{\gamma_I}\rangle}.
\]
Or simply
\[
1=
\sum_{\emptyset\neq\sJ\subset J} (-1)^{|\sJ|+1}\prod_{j\in \sJ}e^{-2\pi i (\gamma_j+\langle b_j,\iota{\gamma_I}\rangle)}
\]
which may be rewritten as
\[
0=\prod_{j\in J} (1-e^{-2\pi i (\gamma_j+\langle b_j,\iota{\gamma_I}\rangle)})
\]
and using Lemma \ref{lem:gamma0} below this becomes
\begin{equation}
\label{eq:ncvandermonde}
0=\prod_{j\in J} (1-e^{-2\pi i \gamma_{I,j}}).
\end{equation}
It follows from Lemma \ref{lem:nonemptyintersect} below that $I\cap J\neq \emptyset$. If $j\in I\cap J$ then $\gamma_{I,j}\in \ZZ$. This implies \eqref{eq:ncvandermonde}.
\end{proof}

\subsection{Supporting lemmas}
Here we prove some lemmas that were used above. Recall that $W$ was assumed to be quasi-symmetric.

\begin{lemma}
\label{lem:gamma0}
Let $\gamma,\gamma'\in Y(T)_\CC$ be such that $A\gamma=\alpha$, $A\gamma'=\alpha$, $\iota \gamma=0$. Then for all $1\le j \le d$ we have
\begin{equation}
\label{eq:gammap}
\gamma'_j=\gamma_j+\langle b_j,\iota\gamma'\rangle.
\end{equation}
\end{lemma}
\begin{proof} Let $\delta=\gamma'-\gamma$. Then \eqref{eq:gammap} maybe rewritten as 
\[
\delta=B^\ast \iota\delta
\]
which follows from the fact $\iota$ is a splitting of $B^\ast$ and $\delta\in \im B^*$.
\end{proof}

To state the next lemma we introduce some notation. For a facet $F$ of $\Delta$ let $\lambda_F\in Y(T)_\RR$, $c_F\in \RR$ be such that 
$\langle \lambda_F,-\rangle-c_F=0$ is a defining equation for the hyperplane spanned by $F$, which is positive on $\Delta$.
\begin{lemma}\label{lem:union}
Let $C_1,C_2\in \Cscr^0$, $C_0=C_1\wedge C_2$, $\dim C_0=n-1$. 
\begin{enumerate}
\item Put
\[
F=\Delta\setminus \bigcup_{\rho_0\in C_0,\rho_2\in C_2} (\rho_2-\rho_0+\Delta).
\]
This is a facet of $\Delta$ parallel to $C_0$.
\item $\forall \rho_0\in C_0:\Lscr_{C_0}\setminus \Lscr_{C_2}\subset \rho_0+\relint F$.\footnote{We denote by $\relint$ the relative interior.} 
\item $\forall \rho_0\in C_0:\forall \rho_1\in C_1:\forall \rho_2\in C_2:\langle\lambda_F,\rho_2-\rho_0\rangle>0$, $\langle\lambda_F,\rho_1-\rho_0\rangle<0$.
\end{enumerate}
\end{lemma}
\begin{proof}
\begin{enumerate}
\item
The affine spaces spanned by $C\in \Cscr$ are intersections of translated hyperplanes spanned by facets of $-\Delta=\Delta$. In particular if $C\in\Cscr$
is a facet then the hyperplane $L$ spanned by it must be parallel to a facet of $\Delta$. 
 Let $\langle \lambda_L,-\rangle-c_L=0$ be an equation
of $L$ which is strictly positive on $C_2$. It follows that 
\begin{equation}\label{eq:supp1}
\forall \rho_2\in C_2: \rho_1\in C_1:\forall \rho_0\in C_0:\langle \lambda_L,\rho_2-\rho_0\rangle>0,\quad \langle \lambda_L,\rho_1-\rho_0\rangle<0.
\end{equation} 

Let $\langle \lambda_L,-\rangle-c\ge 0$ be a supporting half space for $\Delta$. We claim 
\begin{equation}\label{eq:supp2}
F=\Delta\cap \{\delta\mid \langle \lambda_L,\delta\rangle-c= 0\}.
\end{equation}
So we have to prove
\[
 \{\delta\in \Delta\mid \langle \lambda_L,\delta\rangle-c> 0\}=\Delta \cap\left(\bigcup_{\rho_0\in C_0,\rho_2\in C_2} (\rho_2-\rho_0+\Delta)\right).
\]
The $\supset$ inclusion is clear so we prove the $\subset$ inclusion.
Assume $\delta\in \Delta$ is such that $\langle \lambda_L,\delta\rangle-c:=u>0$. Then there exists $\delta'\in \Delta$ such that $\langle \lambda_L,\delta'\rangle-c<u$. Then $\langle \lambda_L ,\delta-\delta'\rangle>0$. Hence for $0<\epsilon\ll 1$ we may write $\epsilon (\delta-\delta')=\rho_2-\rho_0$ for $\rho_0\in C_0$, $\rho_2\in C_2$.  So $\delta-(\rho_2-\rho_0)=(1-\epsilon)\delta+\epsilon \delta'\in \Delta$.
\item In the last paragraph we showed that if $\delta\in \Delta\setminus F$ then $\delta\in \rho_2-\rho_0+\Delta$, where moreover $\rho_0$ might be chosen arbitrarily. Hence, $(-\rho_0+\Lscr_{C_0})\setminus F\subset -\rho_0+\Lscr_{C_2}$ which implies $\Lscr_{C_0}\setminus\Lscr_{C_2}\subset \rho_0+F$. As, by (1),
\[
\bigcap_{\rho_0\in C_0}(\rho_0+F)= \bigcap_{\rho_0\in C_0}(\rho_0+\relint F),
\]
this suffices. 
\item This follows from \eqref{eq:supp1} as $\lambda_F=\lambda_L$ by \eqref{eq:supp2}.\qedhere
\end{enumerate}
\end{proof}

\begin{lemma}\label{lem:nonemptyintersect}
Let $C_1,C_2\in \Cscr^0$, $C_0=C_1\wedge C_2$. Assume that $\dim C_0=n-1$. 
Let $\chi\in \Lscr_{C_0}\setminus \Lscr_{C_2}$. 
Write
\begin{equation}
\label{eq:chidef}
\chi=\rho_0-(1/2)\sum_{i\in J}b_i+\sum_{i\in J'}\beta_ib_i,
\end{equation}
with $\beta_i\in (-1/2,0)$,
$\rho_0\in C_0$, 
$J'\subset\{1,\dots,d\}\setminus
J$ and $|J|$ minimal.
Then $J=\{i\mid \langle \lambda_F,b_i\rangle >0\}$ for $\lambda_F$ as in Lemma \ref{lem:union}. In particular $J=J_{C_0C_2}$ (see \eqref{eq:JC0C1}).

Moreover,
\[
\forall \rho_1\in C_1:\forall \rho_2\in C_2:\forall I\in\hat{\Iscr}_{\rho_2-\rho_1}:I\cap J\neq\emptyset.
\]
\end{lemma}

\begin{proof}

Let $\lambda=\lambda_F\in Y(T)_\RR$ be 
as in Lemma \ref{lem:union}. Since $F$ is in particular the unique facet of $\rho_0+\Delta$ containing $\chi$
we have $J=\{i\mid \langle \lambda,b_i\rangle >0\}$ (see e.g. \cite[Lemma A.7]{SVdB}).  
Let $\rho_i\in C_i$ for $i=1,2$.
Lemma \ref{lem:union} further implies 
 $\langle\lambda,\rho_2-\rho_1\rangle>0$.  
Let $I\in \hat{\Iscr}_{\rho_2-\rho_1}$. By definition of $\hat{\Iscr}_{\rho_2-\rho_1}$ we can write $\rho_2-\rho_1=\sum_{i\in I} \beta_i b_i$ with $\beta_i>0$. 
Hence there exists $i\in I$ such that $\langle\lambda,b_i\rangle>0$, i.e. $i\in J$. 
\end{proof}
\begin{lemma}\label{lem:inC2}
Let the setting be as in Lemma \ref{lem:nonemptyintersect}. 
Then for any $\emptyset\neq \sJ\subseteq J$ one has $\chi+\sum_{j\in \sJ}b_j\in \Lscr_{C_2}$. 
\end{lemma}

\begin{proof}
  By Lemma \ref{lem:union}, there exists 
  a unique facet $F$ of $\Delta$ such that $\chi\in \rho_0+F$ for
  $\rho_0\in
  C_0$.
  Let $\lambda=\lambda_F$.
  One has $\chi_\sJ:=\chi+\sum_{j\in \sJ}b_j\not \in\rho_0+F$ for
  any $\emptyset\neq \sJ\subseteq J$, as
  $J=\{i\mid \langle\lambda,b_i\rangle
  >0\}$ 
  by Lemma \ref{lem:nonemptyintersect}.  
On the other hand it follows from \eqref{eq:chidef} and quasi-symmetry that
$\chi_{\sJ}\in (\rho_0+\Delta)\cap X(T)=\Lscr_{C_0}$. So $\chi_{\sJ}\in \Lscr_{C_0}\setminus (\rho_0+F)\subset \Lscr_{C_2}$ using Lemma \ref{lem:union}(1).
\end{proof}

\section{Perverse sheaves  on affine hyperplane arrangements}\label{sec:KSperverse}
\subsection{Kapranov-Schechtman data}\label{subsec:KSdata}
In this section we consider a general affine hyperplane arrangement $(V,\Hscr)$ as in \S\ref{sec:genhyp}.
Kapranov and Schechtman provide in \cite{KapranovSchechtman} a
combinatorial description of the abelian category of perverse sheaves on $V_\CC\setminus \Hscr_\CC$. 
\begin{theorem}\cite[Theorem 9.10]{KapranovSchechtman}\label{thm:KSperverse}
\label{th:KSdata}
  The category 
   of perverse sheaves  on $V_\CC$ with respect to the stratification induced by $\Hscr_\CC$ is equivalent to the category
  of diagrams consisting of finite dimensional vector spaces $E_C$,
  $C\in\Cscr$, and linear maps $\gamma_{C'C}:E_{C'}\to E_C$,
  $\delta_{CC'}:E_C\to E_{C'}$ for $C'\leq C$ such that
  $((E_C)_C,(\gamma_{C'C})_{CC'})$ is a representation of $(\Cscr,\leq)$
   and $((E_C)_C,(\delta_{CC'})_{CC'})$ a
  representation of $(\Cscr,\geq)$, and the following 
  conditions are satisfied:
\begin{enumerate}
\item[(m)] $\gamma_{C'C}\delta_{CC'}=\id_{E_C}$ for $C'\leq C$. 
In particular, $\phi_{C_1C_2}:=\gamma_{C'C_2}\delta_{C_1C'}$ for $C'\leq C_1,C_2$ is well defined (i.e. independent of $C'$). 
\item[(i)] $\phi_{C_1C_2}$ is an isomorphism for every $C_1\neq C_2$ which
  are of the same dimension $k$ lying, lie in the same $k$-dimensional
  affine space and share a facet. 
\item[(t)] $\phi_{C_1C_3}=\phi_{C_2C_3}\phi_{C_1C_2}$ for collinear (\S\ref{sec:genhyp})
  triples of faces $(C_1,C_2,C_3)$.
\end{enumerate}
\end{theorem}
We denote the category of data introduced in Theorem \ref{th:KSdata}, except for the requirement $\dim E_C<\infty$, by $\KS(\Hscr)$. It is obviously an abelian category. The full subcategory of $\KS(\Hscr)$ such that $\forall C:\dim E_C<\infty$ is denoted by $\KS^c(\Hscr)$.
If $E\in \KS^c(\Hscr)$
then the associated perverse sheaf on $V_\CC$ is denoted by $\tilde{E}$. Thus we have an equivalence of categories
\[
\KS^c(\Hscr)\r \Perv_{\Hscr_\CC}(V_\CC):E\mapsto \tilde{E}
\]
where $\Perv_{\Hscr_\CC}(V_\CC)$ is the abelian category of perverse sheaves on $V_\CC$ with respect to stratification $\Hscr_\CC$.
For an explicit construction of $\tilde{E}$ starting from $E$ see \cite[\S6.C]{KapranovSchechtman}.

Let $\Pi(\Hscr)$ be as  \S\ref{subsec:fundgd}.
There is a ``restriction'' functor
\begin{equation}
\label{eq:res}
\Res:\KS(\Hscr)\r \Rep(\Pi(\Hscr))
\end{equation}
which associates to $E\in \KS(\Hscr)$ the representation of $\Pi(\Hscr)$ given by $C\mapsto E_C$, $\ggamma_{CC'}\mapsto \phi_{CC'}$.
We have the following result.
\begin{proposition}
\label{prop:monrep}
 Let $E\in \KS^c(\Hscr)$. Then the representation of $\Pi(\Hscr)$
corresponding to $\tilde{E}\mid (V_\CC\setminus \Hscr_\CC)$
is given by $\Res(E)$.
\end{proposition}
\begin{proof}
This follows from Proposition \ref{prop:KSmonodromy}, using the construction of the Kapranov-Schechtman data \cite[(4.13),(3.5),\S 4.C]{KapranovSchechtman}.
\end{proof}
\subsection{Duality}\label{subsec:Vdual}
If $E=((E_C)_C,(\delta_{CC'})_{CC'},(\gamma_{C'C})_{C'C})\in \KS^c(\Hscr)$ then we put
\begin{equation}
\label{eq:verdier}
\DD(E):=((E_C^*)_C,(\gamma_{C'C}^\vee)_{C'C},(\delta^\vee_{CC'})_{CC'})\in \KS^c(\Hscr).
\end{equation}

\begin{remark}\label{rem:dualissue}
Let $\DD$ denote the Verdier dual. It is asserted in \cite[Proposition 4.6]{KapranovSchechtman}\footnote{Note that in \cite{KapranovSchechtman} perverse sheaves are shifted so that local systems live in degree zero. Hence in loc.\ cit.\ the Verdier dual is also shifted.} that
$
\DD(\tilde{E})\cong\widetilde{\DD(E)}
$ 
but it turns out that a slight twist,  similar to the twist in \cite[Proposition 4.6]{KSII},
 is in fact needed to make this statement literally correct \cite{KSmisc}.  In the sequel we will not use the compatibility with the Verdier dual.
\end{remark}
\subsection{Group actions}\label{subsec:equivperv}
Assume that in addition $V$ is equipped with an affine, 
$\Hscr$-preserving, group action by a group $\Gscr$ as in \S\ref{subsubsec:equivariance}.
In that case we can routinely define a
$\Gscr$-equivariant version $\KS(\Gscr,\Hscr)$ of $\KS(\Hscr)$. An object in $\KS(\Gscr,\Hscr)$ consists of an object $E=((E_C)_C,(\delta_{CC'})_{CC'},(\gamma_{C'C})_{C'C})$
in $\KS(\Hscr)$ together with isomorphisms
\[
\phi_{g,C}:E_C\r E_{gC}
\]
satisfying the standard cocycle condition, the requirement that
$\phi_{e,C}$ is the identity for $e\in \Gscr$ the neutral element, and
the obvious compatibility with  $\delta_{CC'}$ and the $\gamma_{C'C}$.
The subcategory $\KS^c(\Gscr,\Hscr)\subset \KS(\Gscr,\Hscr)$ is spanned by the objects in $\KS(\Gscr,\Hscr)$ which lie in $\KS^c(\Hscr)$ if we forget the $\Gscr$-action.

An object $E\in \KS^c(\Gscr,\Hscr)$ defines a perverse sheaf on the stack $V_\CC/\Gscr$ which we denote by $\tilde{E}$.
This yields an equivalence of categories
\[
\KS^c(\Gscr,\Hscr)\cong \Perv_{\Hscr_\CC/\Gscr}(V_\CC/\Gscr).
\]
There is a $\Gscr$-equivariant version of the restriction functor
\begin{equation}
\label{eq:res}
\Res:\KS(\Gscr,\Hscr)\r \Rep(\Pi(\Hscr)\rtimes \Gscr)
\end{equation}
which associates to $E\in \KS(\Gscr,\Hscr)$ the representation of $\Pi(\Hscr)\rtimes \Gscr$ given by $C\mapsto E_C$, $\ggamma_{CC'}\mapsto \phi_{CC'}$, $g_C\mapsto \phi_{g,C}$.
\begin{corollary}
Assume that the group $\Gscr$ acts freely and discretely on $V$
and let $E\in \KS(\Gscr,\Hscr)$. Then the representation of $\Pi(\Hscr)\rtimes \Gscr$ (see \ref{subsubsec:equivariance})
associated to the local system
$\tilde{E}\mid (V_\CC\setminus \Hscr_\CC)/\Gscr$
is given by $\Res(E)$.
\label{cor:monrep}
\end{corollary}
For use below we note that the dual \eqref{eq:verdier} can be lifted to a functor
\[
\DD:\KS^c(\Gscr,\Hscr)\r \KS^c(\Gscr,\Hscr)
\]
where $\phi^{\DD(E)}_{g,C}$ is defined as $(\phi_{g,C}^{-1})^\vee=\phi_{g^{-1},C}^\vee$. 
\subsection{$R$-linear versions}\label{subsec:Rlin}
It is often convenient to consider a version of the category $\KS(\Hscr)$ in which the $E_C$ are modules over a commutative ring $R$. Then we denote the 
corresponding category by $\KS_R(\Hscr)$ and other related notations will be decorated with $R$ as well in a self-explanatory fashion. 
For a ring extension $S/R$ we will use the obvious change of rings functor
\begin{equation}
\label{eq:baseextension}
-\otimes_R S:\KS_R(\Hscr)\r \KS_S(\Hscr).
\end{equation}
By $\KS^c_R(\Hscr)$ we denote the full subcategory of $\KS_R(\Hscr)$ consisting of objects $(E_C)_C$ such that each $E_C$ is a finitely generated projective
$R$-module.
\begin{observation}
\label{obs:stab}
We will encounter the following situation. Assume that in addition $V$ is equipped with an affine, 
$\Hscr$-preserving, group action by a group $\Gscr=\Ascr\times \Bscr$ as in \S\ref{subsec:equivperv} such that $\Ascr$ sends every element of  $\Cscr$ to itself. Then
there is an isomorphism of categories 
\begin{equation}
\label{eq:observation}
\KS_R(\Gscr,\Hscr)\cong \KS_{R[\Ascr]}(\Bscr,\Hscr)
\end{equation}
where $R[\Ascr]$ is the group ring of $\Ascr$.
\end{observation}
\begin{remark} One needs to be a bit careful in using \eqref{eq:observation} since with our current conventions we do not automatically have 
the corresponding ``finite'' statement $\KS^c_R(\Gscr,\Hscr)\cong \KS^c_{R[\Ascr]}(\Bscr,\Hscr)$.
\end{remark}
\section{Perverse schobers  on affine hyperplane arrangements}\label{sec:KSschobers}
\subsection{$\Hscr$-schobers}
\label{subsec:reminder}
\label{sec:XT}
Let the setting be as in \S\ref{sec:KSperverse}.
$\Hscr$-schobers are categorifications of Kapranov-Schechtman data and hence they can be regarded as categorifications of perverse sheaves.
The following brief exposition is more or less literally taken from \cite[\S3]{SVdB10} which in turn is based on \cite{BondalKapranovSchechtman,KapranovSchechtmanSchobers}.
\begin{definition}\label{def:schober}
  An $\Hscr$-schober $\Eescr$ on $V_\CC$  is given by triangulated categories $\Eescr_C$, 
  $C\in \Cscr$,  adjoint
pairs of exact functors $(\delta_{CC'}:\Eescr_C\to \Eescr_{C'},\gamma_{C'C}:\Eescr_{C'}\to \Eescr_C)$ 
   for $C'\leq C$ such that
$(\Eescr_C,(\delta_{C'C})_{C'C})$ defines a pseudo-functor\footnote{To specify a pseudo-functor one also needs to specify suitable natural isomorphisms. As is customary we have suppressed these from the notations.}  from $(\Cscr,\ge )$ to the 
 $2$-category of
  triangulated categories satisfying  the following conditions:
 \begin{enumerate}
 \item[(M)] The unit of the adjunction $(\delta_{CC'},\gamma_{C'C})$
defines a natural isomorphism
   $\id_{\Eescr_C}\xrightarrow{\cong}\gamma_{C'C}\delta_{CC'}$ for $C'\leq C$ , and
   thus $\phi_{C_1C_2}:=\gamma_{C'C_2}\delta_{C_1C'}$ for
   $C'\leq C_1,C_2$ is well defined up to canonical
natural   isomorphism.
\item[(I)]
$\phi_{C_1C_2}$ is an equivalence for every $C_1\neq C_2$ of the same dimension $d$ lying in the same $d$-dimensional affine space which share a facet.
\item[(T)]
For collinear triples of faces $(C_1,C_2,C_3)$ with common face $C_0$ the counit of the adjunction $(\delta_{C_0C_2},\gamma_{C_2C_0})$
 defines a natural isomorphism
$\phi_{C_2C_3}\phi_{C_1C_2}\xrightarrow{\cong} \phi_{C_1C_3}$.
\end{enumerate}
\end{definition}
The 2-category of $\Hscr$-schobers is denoted by $\Schob(\Hscr)$.
$\Hscr$-schobers are categorifications of Kapranov-Schechtman data in the following sense.
\begin{fact} Applying $K^0(-)$ to the data defining an $\Hscr$-schober yields a functor
\[
K^0(-):\Schob(\Hscr)\mapsto \KS_{\ZZ}(\Hscr)
\]
which we call the ``decategorification'' functor.
\end{fact}
Below we use the short hand $K^0_\CC(-)$ for $K^0(-)\otimes_{\ZZ} \CC$.
If $K^0_\CC(\Eescr)\in \KS^c(\Hscr)$ then we will also define $\tilde{K}^0_\CC(\Eescr)=\widetilde{K^0_\CC(\Eescr)}$
and we will refer to $\tilde{K}^0_\CC(\Eescr)$ as a decategorification as well.

\medskip

If $\Eescr$ is an $\Hscr$-schober then we will define a subschober $\Eescr'$ of $\Eescr$ as a collection of (full) triangulated subcategories $\Eescr'_C\subset \Eescr_C$
which are stable under $(\delta_{CC'},\gamma_{C'C})$. In this way $\Eescr'$ becomes tautologically an $\Hscr$-schober. 
We refer to the latter  as a subschober as well and sometimes we use the notation $\Eescr'\subset \Eescr$.
\subsection{Group actions}\label{subsec:equischober}
Assume that in addition $V$ is equipped with an affine, 
$\Hscr$-preserving, group action by a group $\Gscr$ as in \S\ref{subsubsec:equivariance}.
In that case we may define a
$\Gscr$-equivariant version $\Schob(\Gscr,\Hscr)$
of the category $\Schob(\Hscr)$.  A
\emph{$\Gscr$-action} on an $\Hscr$-schober on $V_\CC$ is a collection
of exact functors
\[
\phi_{g,C}:\Eescr_C\r \Eescr_{gC}
\]
for $g\in \Gscr$, $C\in \Cscr$, enhanced with natural isomorphisms $\phi_{h,gC}\phi_{g,C}\cong \phi_{hg,C}$ satisfying the obvious compatibility
for triple products in $\Gscr$, and the requirement that $\phi_{e,C}$ is the identity functor for $e\in \Gscr$. Moreover we should have pseudo-commutative diagrams for every $C'<C$:
\begin{equation}
\label{eq:deltadiagram}
\xymatrix{
\Eescr_C\ar[r]^{\phi_{g,C}}\ar[d]_{\delta_{C,C'}}&\Eescr_{gC}\ar[d]^{\delta_{gC,gC'}}\\
\Eescr_{C'}\ar[r]_{\phi_{g,C'}}&\Eescr_{gC'}
}
\end{equation}
so that the implied natural isomorphism $\delta_{gC,gC'}\phi_{g,C}\cong \phi_{g,C'}\delta_{C,C'}$ should again satisfy a number of obvious compatibilities. 
An $\Hscr$-schober
equipped with a $\Gscr$-action will be called a \emph{$\Gscr$-equivariant $\Hscr$-schober}. The concept of a $\Gscr$-equivariant subschober is defined in the obvious way.

\medskip

We think of a $\Gscr$-equivariant $\Hscr$-schober as a perverse schober on the stack $V_{\CC}/\Gscr$. Its decategorification is a $\Gscr$-invariant perverse sheaf; i.e. we will use the functors
\[
K^0(-):\Schob(\Gscr,\Hscr)\r \KS_\ZZ(\Gscr,\Hscr)
\]
and if $K^0_\CC(\Eescr)\in \KS^c(\Gscr,\Hscr)$ then we will also use the notation $\tilde{K}^0_\CC(\Eescr):=\widetilde{K^0_\CC(\Eescr)}$. 
\section{An $\Hscr$-schober using Geometric Invariant Theory}\label{subsec:construct}
\subsection{Reminder}\label{subsec:10.1}
From now on $\Hscr$, $\Cscr$ will have again their usual meaning (see \S\ref{sec:hyperplane}).
Assume that~$W$ is quasi-symmetric.  
In  \cite{SVdB10} we constructed an $\Hscr$-schober on $X(T)_\CC$ built up from suitable triangulated subcategories of
$D(W/T)$.\footnote{In loc. cit. we actually used $W^\ast/T$. Here, we follow the mirror symmetry convention and use $W/T$. This entails some sign changes
in the definitions and formulas taken from \cite{SVdB10}.
}
We briefly recall this construction.\footnote{In loc. cit. we required, in addition to $W$ being quasi-symmetric,
   that the generic $T$-stabilizer is
  finite. This is satisfied here since $\ZZ B=X(T)$.} 
Let $P_\chi=\chi\otimes \Oscr_W$ and set
\[
P_C:=\bigoplus_{\chi\in -\Lscr_C} P_{\chi},\quad 
\Escr_C:=\langle P_C\rangle \subset D(W/T),
\]
where $\langle S\rangle$ for $S\subset D(W/T)$ denotes the smallest strict, full triangulated subcategory, closed under coproduct, which contains $S$.
Put
\[
\Lambda_C=\End_{W/T}(P_C).
\]
The functor $\RHom_{W/T}(P_C,-)$ defines an equivalence
\begin{equation}
\label{eq:LambdaCeq}
\Escr_C\cong D(\Lambda_C)
\end{equation}
where $D(\Lambda_C)$ is the derived category of \emph{right} $\Lambda_C$-modules. By \cite[Theorem 5.6]{SVdB10}, $\Lambda_C$ has finite global dimension.  For
use below we note that the quasi-inverse to \eqref{eq:LambdaCeq} is
given by
\begin{equation}
\label{eq:rightadjoint}
D(\Lambda_C)\cong \Escr_C:M\mapsto M\otimes_{\Lambda_C}^L P_C.
\end{equation}
For $C'\leq C$ let $\delta_{CC'}:\Escr_C\hookrightarrow \Escr_{C'}$ be the inclusion. Then $\delta_{CC'}$ admits a right adjoint  
\begin{equation}
\label{eq:gamma}
\gamma_{C'C}=\RHom_{W/T}(P_C,-)\otimes_{\Lambda_C}^L P_C.
\end{equation}
Set $\phi_{C_1C_2}=\gamma_{C'C_2}\delta_{C_1C'}$ for $C'\leq C_1,C_2$. 
In addition for $\chi\in X(T)$, let $\phi_{\chi,C}:\Escr_C\to \Escr_{\chi+C}$ be the functor $M\mapsto (-\chi)\otimes M$.

\begin{proposition}\cite[Proposition 5.1]{SVdB10}\label{prop:schober}
$\Sscr:=((\Escr_C)_{C},(\gamma_{C'C})_{C'C},(\delta_{CC'})_{C'C})$  with the $X(T)$-action $(\phi_{\chi,C})_{\chi,C}$ defines an $X(T)$-equivariant $\Hscr$-schober on $X(T)_\CC$.
\end{proposition}
Let $D^c(W/T)$ be the full subcategory of $D(W/T)$ consisting of bounded complexes of $T$-equivariant coherent $\Oscr_T$-modules. 
It was shown in \cite{SVdB10} that $\delta_{C'C}$, $\gamma_{CC'}$ preserve the categories
\[
\Escr^c_C:=\Escr_C\cap D^c(W/T).
\]
We obtain a corresponding $X(T)$-equivariant subschober
\[
\Sscr^c \subset \Sscr.
\]
We recall the following lemma.
\begin{lemma}\cite[\S5.3]{SVdB10}
\label{lem:Dc}
  Under the equivalence \eqref{eq:LambdaCeq}, $\Escr^c_C$ corresponds to the subcategory  $D^c(\Lambda_C)$ of $D(\Lambda_C)$ consisting of bounded complexes
with finitely generated cohomology.
\end{lemma}

In the next section we discuss a subschober of $\Sscr^c$.
\subsection{The finite length subschober}\label{subsec:finitelengthschober}
The nullcone $W^u\subset W$ is defined by
\[
W^u=\{x\in W\mid 0\in \overline{Tx}\}.
\]
Let $D^u(W/T)$ be the full triangulated subcategory of complexes whose cohomology is supported on $W^u$. 
We put
\begin{align*}
\Escr^u_C&=\Escr_C\cap D^u(W/T),\\
\Escr^f_C&=\Escr^u_C\cap \Escr_C^c.
\end{align*}
\begin{lemma} $(\Escr^u_C)_C$, $(\Escr^f_C)_C$ define $X(T)$-equivariant subschobers $\Sscr^u$, $\Sscr^f$ 
of $\Sscr$.
\end{lemma}
\begin{proof}
$X(T)$-equivariance is obvious. 
Furthermore it is sufficient to prove that  $(\Escr^u_C)_C$ is a subschober. Compatibility with $\delta_{CC'}$ is obvious so we have to discuss compatibility with $\gamma_{CC'}$.
From the formula \eqref{eq:gamma} it follows that $\gamma_{C'C}$ is right exact\footnote{It may look strange that a right adjoint is right exact,
but this is not a contradiction. The standard t-structure on $D(W/T)$ does not descend to a $t$-structure on $\Escr_C$.}
 for the standard $t$-structure on $D(W/T)$. Then one quickly finds that it is sufficient to prove that for $M\in \coh(W^u/T)$ the cohomology of
\[
\Hom_{W/T}(P_C,M)\Lotimes_{\Lambda_C} P_C
\]
is supported in $W^u$. By the hypotheses on $M$, $F:=\Hom_{W/T}(P_C,M)$ is supported in the origin of $W\quot T$. 
We now switch to $T$-equivariant $\CC[W]$-modules. Those are supported in $W^u$ if and only if they are supported in the origin of $W\quot T$ when regarded as $\CC[W]^T$-modules. By resolving $P_C$, it thus follows that $F\Lotimes_{\Lambda_C} P_C$ is supported in $W^u$.
\footnote{We are using that for $M$, $N$ respectively a right and a left module over a ring $A$ (possibly non-finitely generated) and $R\subset Z(A)$, it is true that $\supp_R(M\otimes_A N)\subset \supp_R(M)\cap \supp_R(N)$. This is an immediate consequence of the definition of support.}
\end{proof}
\begin{notation} \label{not:decoration}
In case of possible confusion we decorate the KS-data associated to the $\Hscr$-schobers $\Sscr^c$ and $\Sscr^f$ by $(-)^c$, $(-)^f$.
\end{notation}

Let $\mod^f(\Lambda_C)$ be the full subcategory of $\Mod(\Lambda_C)$ consisting
of finite dimensional right $\Lambda_C$-modules supported in the origin of
$W\quot T$.
\begin{lemma} \label{lem:correspondence}
Under the equivalence \eqref{eq:LambdaCeq}, $\Escr^f_C$ corresponds to the full subcategory $D^f(\Lambda_C)$ of $D(\Lambda_C)$  consisting of bounded
complexes with cohomology in $\mod^f(\Lambda_C)$.
\end{lemma}
For $\chi\in X(T)$ let $S_\chi$ be the corresponding 1-dimensional $T$-equivariant $\Oscr_W$- module supported at the origin.
\begin{lemma}
\label{lem:simple}
The simple objects in $\mod^f(\Lambda_C)$ are given by $\Hom_{W/T}(P_C,S_\chi)$
for $\chi\in \Lscr_{-C}$.
\end{lemma}
\begin{proof}
The objects $\Hom_{W/T}(P_C,S_\chi)$ are $1$-dimensional and  contained in the category $\mod^f(\Lambda_C)$. Hence they are certainly simple.  To prove the converse let
$R=\CC[W]^T$, equipped with its natural $\NN$-grading.  Then
the simple objects of $\mod^f(\Lambda_C)$ are the simple modules of
the finite dimensional algebra $\Lambda_C':=\Lambda_C/R_{\ge 1}\Lambda_C$. Since
$\Lambda'_C$ is also $\NN$-graded, $(\Lambda'_C)_{\ge 1}$ is nilpotent. Hence
the simple $\Lambda'_C$-modules correspond to the indecomposable summands of
$(\Lambda'_C)_0=(\Lambda_C)_0$ and they are precisely the $\Hom_{W/T}(P_C,S_\chi)$ for $\chi\in \Lscr_{-C}$.
\end{proof}
Below we will denote the object in $\Escr^f_C$ corresponding to $\Hom_{W/T}(P_C,S_\chi)$
by~$\mathfrak{s}_{C,\chi}$. Using the formula \eqref{eq:rightadjoint} we obtain
\begin{equation}
\label{eq:sformula}
\mathfrak{s}_{C,\chi}=\Hom_{W/T}(P_C,S_\chi)\Lotimes_{\Lambda_C} P_C.
\end{equation}
The following is clear. 
\begin{lemma} 
\label{lem:dual}
For $\chi,\mu \in \Lscr_{-C}$,
$
\Hom_{W/T}(P_\chi,\mathfrak{s}_{C,\mu})=\CC^{\delta_{\chi,\mu} }.
$
\end{lemma}

\subsection{Autoduality}\label{subsec:autoduality}
Below we will need the autoduality functor 
\[\Dscr:D^b(\coh(W/T))\r D^b(\coh(W/T))^\circ,\]
\[
\Dscr=\uRHom_{W/T}(-,\Oscr_W).
\]

We recall some basic properties from \cite[\S 5.4]{SVdB10}.
\begin{lemma}\label{lem:autoduality}
We have $\Dscr(\Escr^*_C)=\Escr^*_{-C}$ for $*\in \{c,f\}$. Moreover, for $C'<C$, $\Dscr\delta_{-C,-C'}\Dscr=\delta_{CC'}$. Thus, $\Dscr\gamma_{-C',-C}\Dscr$ is left adjoint to $\delta_{CC'}$.
\end{lemma}
It will be useful to consider the bifunctor
\begin{multline}
\label{eq:bifunctor}
\RHom_{W/T}(-,-)':=\RHom_{W/T}(\Dscr(-),-):\\\Escr^c_{-C}\times \Escr^f_C\r D^b(\mod(\CC)):(M,N)\mapsto \RHom_{W/T}(\Dscr M,N)
\end{multline}
\begin{lemma} \label{eq:bigadjoints}
Let $C'<C$ in $\Cscr$. Then $(\delta^c_{-C,-C'},\gamma^f_{C'C})$ and $(\gamma^c_{-C',-C},\delta^f_{CC'})$ are adjoint pairs under the bifunctor
$\RHom_{W/T}(-,-)'$. Moreover $(\phi^c_{\chi,-C},\phi^f_{\chi,-\chi+C})$ is an adjoint pair as well.
\end{lemma}
\begin{proof} This follows from Lemma \ref{lem:autoduality}.
\end{proof}
\section{Decategorification of the GIT $\Hscr$-schober}
\subsection{Duality}\label{subsec:GrothendieckGDuality}
In this section we will construct a canonical  duality isomorphism in $\KS_\ZZ(X(T),\Hscr)$, 
\begin{equation}\label{eq:dualK0}
K^0(\Sscr^c)^-\cong \DD_\ZZ(K^0(\Sscr^f)),
\end{equation}
where $(-)^-$ represents the pullback of equivariant KS-data under 
$(X(T),X(T)_\RR)\r (X(T),X(T)_\RR):(\chi,x)\mapsto (-\chi,-x)$. Concretely, 
\begin{equation}
\label{eq:dualK01}
\begin{gathered}
K^0(\Sscr^c)^-(C)=K^0(\Sscr^c)(-C)=K^0(\Escr^c_{-C}),\\
K^0(\delta^{c}_{CC'})^-=K^0(\delta^c_{-C,-C'}), \quad K^0(\gamma^c_{C'C})^-=K^0(\gamma^c_{-C',-C}),\\
 K^0(\phi^c_{\chi,C})^-=K^0(\phi^c_{-\chi,-C}).\\
\end{gathered}
\end{equation}
\begin{lemma}
\label{eq:groth}
The following holds for the Grothendieck groups of $\Escr_C^c$, $\Escr_C^f$.
\begin{enumerate}
\item
$K^0(\Escr_C^c)$ is freely generated by the classes $[P_\chi]$ for $\chi\in -\Lscr_C$.
\item
$K^0(\Escr_C^f)$ is freely generated by the classes $[\mathfrak{s}_{C,\chi}]$ for $\chi\in -\Lscr_C$.
\end{enumerate}
\end{lemma}
\begin{proof}
\hfill 

\begin{enumerate}
\item We have $K^0(\Escr_C^c)=K^0(\Lambda_C)$ and as $\Lambda_C$ is a $\NN$-graded ring of finite
global dimension we obtain by \cite[p. 112, Theorem 7]{Quillen} that 
$\Lambda_C\otimes_{\Lambda_{C,0}} -$ induces an isomorphism $K^0(\Lambda_{C,0})\cong K^0(\Lambda_C)$.
It now suffices to observe $\Lambda_{C,0}$ is semi-simple and its summands are indexed by $\chi \in -C$.
The image of a summand corresponding to $\chi$ is $[P_\chi]$.
\item This follows from Lemmas \ref{lem:correspondence}, \ref{lem:simple}. \qedhere
\end{enumerate}
\end{proof}
For $M\in D^c(W/T)$ and $N\in D^f(W/T)$ the Euler pairing between the classes of $M$, $N$ is defined as
\begin{equation}\label{eq:pairing}
\langle [M], [N]\rangle=\sum_i (-1)^i\dim H^i(\RHom_{W/T}(M,N)).
\end{equation}
\begin{lemma}
\label{eq:groth1}
The Euler pairing yields a perfect duality between the Grothendieck groups $K^0(\Escr^c_C)$ and $K^0(\Escr^f_C)$.
\end{lemma}
\begin{proof} 
This follows from Lemma \ref{lem:dual}.
\end{proof}
We now define a twisted version of the Euler pairing:
\begin{equation}\label{eq:Eulertwist}
\langle-,-\rangle':=\langle K^0(\Dscr)(-),-\rangle:K^0(\Escr^c_{-C})\times K^0(\Escr^f_C)\r \ZZ.
\end{equation}

\begin{lemma}\label{lem:adjointmaps}
Let $C'<C$ in $\Cscr$. Then $\langle-,-\rangle'$ is a perfect duality between $K^0(\Escr^c_{-C})$ and $K^0(\Escr^f_{C})$. Moreover,  
 $(K^0(\delta^c_{-C,-C'}),K^0(\gamma^f_{C'C}))$ and $(K^0(\gamma^c_{-C',-C}),\allowbreak K^0(\delta^f_{CC'}))$ are adjoint pairs for $\langle -,-\rangle'$.
Finally $(K^0(\phi^c_{\chi,-C}),K^0(\phi^f_{\chi,-\chi+C}))$ is an adjoint pair as well.
\end{lemma}
\begin{proof} This follows by combining Lemma \ref{eq:groth1} with Lemma \ref{eq:bigadjoints}.
\end{proof}
\begin{proof}[Proof of \eqref{eq:dualK0}] The map \eqref{eq:dualK0} is
  obtained from the pairing $\langle-,-\rangle'$.  The fact that it a
  well defined map of equivariant KS-data follows from Lemma
  \ref{lem:adjointmaps}.  Note: the fact that a minus sign appears
  in the $X(T)$-action follows from the inverse that appears in the
  description of $\phi_{g,C}$ for the equivariant dual.  See
  \S\ref{subsec:equivperv}.
\end{proof}
\begin{remark}
The need for $(-)^-$ in \eqref{eq:dualK0} and the associated proofs leads to some foundational musing about the definition of $\Hscr$-schobers. 
  As mentioned in \cite[Observation 3.6]{SVdB10} the $\Hscr$-schober 
  $\Sscr$ has some favorable properties not shared by all
  $\Hscr$-schobers. In particular, as we  observed in Lemma \ref{lem:autoduality},
$\delta_{CC'}$ also admits a left
  adjoint. Carrying this further, one may think of the $\Hscr$-schobers as introduced in \cite{KapranovSchechtmanSchobers}
as \emph{right $\Hscr$-schobers} and then 
introduce the dual concept of  \emph{left $\Hscr$-schobers} by
  requiring that $\gamma_{C'C}$ be a {left adjoint} to
  $\delta_{CC'}$. Both left and right $\Hscr$-schobers are uniquely determined by the $\delta$'s and hence it makes sense to use the notations
$\Sscr_l$, $\Sscr_r$ to refer to a left and/or a right $\Hscr$-schober with given $\delta$'s (if both $\Sscr_l$, $\Sscr_r$ exist,
as in our case, then it even makes sense to refer to $\Sscr$ itself as an \emph{$\Hscr$-bischober}).
With these notations the formula \eqref{eq:dualK0} could have been written more elegantly as
\[
  K^0(\Sscr_l^c)\cong\DD_\ZZ(K^0(\Sscr_r^f)).
\]
where the duality is now realized via the Euler form $\langle-,-\rangle$ instead of the twisted version $\langle-,-\rangle'$.
\end{remark}

\subsection{Explicit description of the monodromy isomorphisms}
Let us recall the complex $C_{\lambda,\chi}$ introduced in
\cite[(11.3)]{SVdB}, and used in \cite{SVdB10}, in order to construct
$\gamma_{C'C}$ for $C'<C$.  For $\lambda\in Y(T)$ define
\[
W^{\lambda,+}=\{x\in W\mid \lim_{t\r 0} \lambda(t)x\text{ exists}\}.
\]
We also put $K_\lambda=W/W^{\lambda,+}$,  $d_\lambda=\dim K_\lambda$.  Then by definition
$C_{\lambda,\chi}$ is the Koszul resolution of $\Oscr_{W^{\lambda,+}}$
tensored with $\chi\in X(T)$; i.e. $C_{\lambda,\chi}$ equals the
complex (with the right-most term in degree $0$)
\begin{multline}
\label{eq:Cchi}
0\r \chi\otimes \wedge^{d_{\lambda}} K^*_{\lambda}\otimes \Oscr_W
\r \chi\otimes \wedge^{d_{\lambda}-1} K^*_{\lambda}\otimes \Oscr_W
\r
\cdots
\r 
\chi\otimes \Oscr_W.
\end{multline}
\begin{lemma}\label{cor:Clambdachi}
Assume that $C_1,C_2\in \Cscr^0$ share a facet $C_0$ and let $\chi\in \Lscr_{-C_1}\setminus \Lscr_{-C_2}$. Let  $\langle \lambda,-\rangle-c$ for $\lambda\in Y(T)$ be
a defining equation of the hyperplane spanned by $-C_0$ which is strictly positive on $-C_2$. 
Then $\phi_{C_1C_2}(P_\chi)={\rm{cone}}(P_\chi\to C_{\lambda,\chi})$.
\end{lemma}
\begin{proof}
This follows from \cite[\S5]{SVdB10}. 
Note that $\Escr_C$ was called $\Escr_{-C}$ in loc.\ cit. and here our space is called $W$ instead of $W^\ast$. Neither of these changes has any serious implications
as we are proving an intrinsic statement.

We have
\[
\phi_{C_1C_2}(P_\chi)=\gamma_{C_0C_2}\delta_{C_1C_0}(P_\chi)=\gamma_{C_0C_2}(P_\chi).
\]
By Proposition 5.12 in loc.\ cit.\ we have
a semi-orthogonal decomposition
\begin{equation}
\label{eq:sod}
\Escr_{C_0}=\langle \Escr_{C_0C_2} , \Escr_{C_2}\rangle.
\end{equation}
By Lemma \ref{lem:union}(2) $\Lscr_{-C_0}\setminus \Lscr_{-C_2}$ is contained in the relative interior of the translate of a single facet of $\Delta$.
It then follows from the discussion in loc.\ cit.\ that
\[
\Escr_{C_0C_2}=\langle (\chi\otimes\Oscr_{W^{\lambda,+}})_{\chi\in
  \Lscr_{-C_0}\setminus\Lscr_{-C_2}}\rangle.
\]
Using the fact that $C_{\lambda,\chi}\cong \Oscr_{W^{\lambda,+}}$
it follows from Lemma \ref{eq:complexes} below
and \eqref{eq:sod} that we have 
\[
{\rm{cone}}(P_\chi\to C_{\lambda,\chi})\in \Escr_{C_0}\cap {}^\perp\Escr_{C_0C_2}
\subset \Escr_{C_2}.
\]
This implies $\gamma_{C_0C_2}(P_\chi)={\rm{cone}}(P_\chi\to C_{\lambda,\chi})$.
\end{proof}
We have used the following lemma.\footnote{Versions of this lemma have
  already been used numerous times in our previous work. We give the
  proof since we need some details about the complexes in the proof of
  Proposition \ref{lem:supportschobermonodromy}.}
\begin{lemma}
\label{eq:complexes}
Let $\lambda\in Y(T)$, $\chi\in X(T)$.
\begin{enumerate}
\item The weights $\mu$ of $\chi\otimes \Oscr_{W^{\lambda,+}}$ satisfy $\langle \lambda,\mu \rangle\le \langle \lambda,\chi\rangle$.
\item All terms $P_\mu$ in $C_{\lambda,\chi}$, except $P_\chi$, satisfy $ \langle\lambda,\mu\rangle > \langle\lambda,\chi\rangle$.
\end{enumerate}
\end{lemma}
\begin{proof} 
First we note that  $W^{\lambda,+}$ is spanned by the weight vectors $e_j$ such that
  $\langle\lambda,b_j\rangle \ge 0$.  Hence
  $\CC[W^{\lambda,+}]=\Sym((W^{\lambda,+})^\ast)$ is generated by
  elements of weight $-b_j$ for $\langle\lambda,b_j\rangle \ge 0$.
  This proves the first claim.

For the second claim we note that $K_\lambda$ is generated
by weight vectors $\bar{e}_j$ such that $\langle\lambda,b_j\rangle < 0$
and now we look at the weights of $\wedge^i K_\lambda^\ast$ for $i>0$.
\end{proof}

\begin{example}\label{ex:conifoldcom}
Let $(T,W)$ be as in Example \ref{ex:example1}. Let $C_1=]0,1[$, $C_2=]-1,0[$. Then $\{-1\}=\Lscr_{-C_1}\setminus \Lscr_{-C_2}$. We take $\lambda=1$. Then $K_\lambda=(-1)\oplus (-1)$ and \eqref{eq:Cchi} becomes 
\[
0\to (1)\otimes\Oscr_W\to (0)^{\oplus 2}\otimes \Oscr_W\to (-1)\otimes \Oscr_W. 
\]
Hence $\phi_{C_1C_2}(P_{-1})=(P_1\to P_0^{\oplus 2})$. Analagously, $\phi_{C_2C_1}(P_1)=(P_{-1}\to P_0^{\oplus 2})$.
\end{example}

\section{The GIT $\Hscr$-schober in the $X(\TT)$-equivariant setting}\label{subsec:Hequivariant}
\subsection{More splittings}\label{subsec:moresplittings}
In the rest of this section we will choose a splitting $\jota:H\to \TT$ of $A:\TT\r H$ (see \S\ref{subsec:setting}) such that $\iota\jota=1$ where $\iota:\TT\r T$ is the
splitting of $B^\ast:T\r \TT$ introduced in  \S\ref{subsec:splitting}. Summarizing we now have the following maps
\[
\xymatrix{
&& T\ar@<1ex>[rr]^{B^\ast}&&\TT\ar@<1ex>[ll]^{\iota}\ar@<1ex>[rr]^{A}&&H
\ar@<1ex>[ll]^{\jota}&&
}
\]
where the composable maps form short exact sequences. In other words with our choices we have defined an isomorphism
\begin{equation}
\label{eq:times}
\TT\cong T\times H
\end{equation}
where $A,B^\ast,\iota,\jota$ are given by the appropriate inclusion and projection maps.

\subsection{Lift of the $\Hscr$-schober}\label{subsec:Hequivariant}
The $X(T)$-equivariant $\Hscr$-schober we have constru\-cted lives on
the hyperplane arrangement $\Hscr$ in the real vector spaces
$X(T)_\RR$.  Now observe that $X(T)_\RR$  may be  trivially equipped with a
$X(\TT)$-action via the map $B:X(\TT)\r X(T)$.  It turns out that all the constructions given in \S\ref{subsec:construct} have natural $\TT$-equivariant
versions. Throughout we follow the convention that such $\TT$-equivariant versions are indicated by overlining and sometimes we even omit explicit definitions when they are obvious.

\medskip

First of all a variation of the construction  of $\Sscr$ which yields an object in the category $\Schob(X(\TT),\Hscr)$:
\begin{equation}
\label{eq:barschober}
\bar{\Sscr}=((\bar{\Escr}_C)_C,(\bar{\delta}_{CC'})_{CC'},
(\bar{\gamma}_{C'C})_{C'C},(\bar{\phi}_{\chi,C})_{\chi,C})
\end{equation}
 on $X(T)_\RR$, 
built up from triangulated subcategories
$(\bar{\Escr}_C)_{C\in \Cscr}\subset D(W/\TT)$ which are just the natural lifts of
$(\Escr_C)_C\subset D(W/T)$ under the pushforward functor for the 
stack morphism $W/\TT\r W/T$.
 More precisely for $\chi\in X(\TT)$ put
$\bar{P}_{\chi}:=\chi\otimes \Oscr_W$  and for $C\in \Cscr$ put
\begin{equation}
\label{eq:barL}
\bar{\Lscr}_C:=B^{-1}(\Lscr_C)\subset X(\TT),
\end{equation}
and
\[
\bar{P}_C:=\bigoplus_{\chi\in -\bar{\Lscr}_C} \bar{P}_\chi, \qquad \bar{\Escr}_C:=\langle \bar{P}_C\rangle \subset D(W/\TT).
\]
The functors $(\bar{\delta}_{CC'})_{CC'},
(\bar{\gamma}_{C'C})_{C'C}$ are defined in exactly the same way as for~$\Sscr$, and $\bar{\phi}_{\chi,C}:\bar{\Escr}_C\to \bar{\Escr}_{B\chi+C}$ with $\Mscr\mapsto (-\chi)\otimes \Mscr$.  
It is clear that $\bar{\Sscr}$ has subschobers $\bar{\Sscr}^c$, $\bar{\Sscr}^f$ which are again
defined like $\Sscr^c$ and $\Sscr^f$.

\section{Decategorification of the $X(\TT)$-equivariant GIT $\Hscr$-schober}
We will now summarize the $\TT$-equi\-variant versions of the results in \S\ref{subsec:construct}.
\subsection{Duality}
Below for $M\in D^c(W/\TT)$, $N\in D^f(W/\TT)$, we denote
\begin{equation*}
\langle [M], [N]\rangle=\sum_{i\in\ZZ,\chi\in X(H)} (-1)^i\dim (H^i(\RHom_{W/\TT}(M,\chi^{-1}\otimes N)))\chi\in \ZZ[X(H)],
\end{equation*}
noting that the dimensions are finite, and the sum has a finite number of terms.
Moreover, let $\Dscr=\uRHom_{W/\TT}(-,\Oscr_W)$ and 
$
 \langle -,-\rangle'=\langle K^0(\Dscr)(-),-\rangle.
$ 
We note that since $X(H)\subset X(\TT)$ stabilizes $\Cscr$, $K^0(\bar{\Escr}_C^c)$ and $K^0(\bar{\Escr}_C^f)$ are $\ZZ[X(H)]$-modules.
\begin{lemma} 
\label{lem:groth2}
\hfill
\begin{enumerate}
\item \label{it:five81} 
The form $\langle-,-\rangle'$ is $\ZZ[X(H)]$-linear; i.e.\ 
for $a\in K^0(\bar{\Escr}_{-C}^c)$, $b\in K^0(\bar{\Escr}_C^f)$, $f\in \ZZ[X(H)]$ we have 
\[
f\langle a,b\rangle'=\langle a,fb\rangle'=\langle fa,b\rangle'.
\]
\item\label{it:one81}
 $K^0(\bar{\Escr}_C^c)$ is a free $\ZZ[X(H)]$-module with basis $[\bar{P}_{\iota\chi}]$ for $\chi\in \Lscr_{-C}$.
\item\label{it:two81} 
$K^0(\bar{\Escr}_C^f)$ is a free $\ZZ[X(H)]$-module with basis
$[\bar{\mathfrak{s}}_{C,\iota\chi}]$ for $\chi\in \Lscr_{-C}$.
\item \label{it:three81} 
The classes of $[\bar{P}_{\iota\chi}]_{\chi\in \Lscr_{C}}\in K^0(\bar{\Escr}_{-C}^c)$ and $[\bar{\mathfrak{s}}_{C,\iota\chi}]_{\chi\in \Lscr_{-C}}\in
K^0(\bar{\Escr}_C^f)$ are dual $\ZZ[X(H)]$-bases 
for  $\langle-,-\rangle'$.  Whence $\langle-,-\rangle'$ is a perfect duality between $K^0(\overline{\Escr}^c_{-C})$ and $K^0(\overline{\Escr}^f_{C})$. 
\item \label{item:6} 
For $C'<C$ in $\Cscr$ the pairs 
 $(K^0(\bar{\delta}^c_{-C,-C'}),K^0(\bar{\gamma}^f_{C',C}))$ and $(K^0(\bar{\gamma}^c_{-C',-C}),\allowbreak K^0(\bar{\delta}^f_{C,C'}))$ are adjoint pairs for $\langle -,-\rangle'$. The same holds for the pair  $(K^0(\bar{\phi}^c_{\chi,-C}),K^0(\bar{\phi}^f_{\chi,-B\chi+C}))$ for $\chi\in X(\TT)$, $C\in \Cscr$. 
\item \label{it:four81} 
We have
\[
K^0(\Escr^c_C)\cong K^0(\bar{\Escr}^c_C) \otimes_{\ZZ[X(H)],\bold{1}} \ZZ,\qquad K^0(\Escr^c_C)_{\CC}\cong K^0(\bar{\Escr}^c_C) \otimes_{\ZZ[X(H)],\bold{1}}  \ZZ
\]
where $\bold{1}$ is the ring homomorphism $\ZZ[X(H)]\r \ZZ:\chi\mapsto 1$.
\end{enumerate}
\end{lemma}
\begin{proof} (\ref{it:one81},\ref{it:two81},\ref{it:three81},\ref{item:6}) are proved like Lemmas \ref{eq:groth}, \ref{eq:groth1}, \ref{lem:adjointmaps}. 
(\ref{it:five81},\ref{it:four81}) are clear.
\end{proof}
\begin{convention}
From now on we will consider $K^0(\bar{\Sscr}^\ast)$ for $\ast\in \{c,f\}$ as objects in the category $\KS_{\ZZ[H]}(X(T),\Hscr)$ via Observation \ref{obs:stab} and
the decomposition \eqref{eq:times}.
\end{convention}

\begin{corollary}\label{cor:Hdual} From Lemma \ref{lem:groth2}\eqref{item:6} we obtain a canonical duality  isomorphism in $\KS_{\ZZ[H]}(X(T),\Hscr)$ 
\begin{equation}\label{eq:dualK10}
K^0(\bar{\Sscr^c})^-\cong {}_\tau\DD_{\ZZ[X(H)]}(K^0(\bar{\Sscr}^f)),
\end{equation}
where $(-)^-$ is like in \eqref{eq:dualK01}, and $\tau$ denotes the twist of the $\ZZ[X(H)]$-action by the automorphism $\mu\mapsto\mu^{-1}$ for $\mu\in X(H)$.  
\end{corollary}

\subsection{Monodromy}
\label{subsec:monodromyschober}
\begin{proposition}\label{lem:supportschobermonodromy}
Let $N=\Res_{\ZZ[X(H)]}(K^0(\bar{\Sscr}^c))\in \Rep_{\ZZ[X(H)]}(\Pi(\Hscr)\rtimes X(T))$  (see \eqref{eq:res})
where we consider $K^0(\bar{\Sscr}^c)$ as an object in 
$\KS_{\ZZ[X(H)]}(X(T),\Hscr)$ as above.

\begin{enumerate}
\item For $C\in \Cscr^0$ we have
\[
N(C)=K^0(\bar{\Escr}^c_C).
\]
\item 
For $C_1,C_2\in \Cscr^0$ such that  $\dim C_1\wedge C_2=n-1$, denote $C_0=C_1\wedge C_2$.  
 Set $J=J_{-C_0,-C_2}$ (see  \eqref{eq:JC0C1}). 
 The map $\ggamma_{C_1C_2}$ evaluated on $[\bar{P}_\chi]$ is given by 
\begin{multline*}\label{eq:K0phi}
N(\ggamma_{C_1C_2})([\bar{P}_\chi])=K^0(\bar{\phi}_{C_1C_2})[\bar{P}_\chi]=[\bar{\phi}_{C_1C_2}(\bar{P}_\chi)]=\\
\begin{cases}
[\bar{P}_\chi]&\text{if $\chi\in \bar{\Lscr}_{-C_1}\cap \bar{\Lscr}_{-C_2}$},\\
\sum_{\emptyset\neq \sJ\subset J} (-1)^{|\sJ|+1}[\bar{P}_{\chi-\sum_{j\in \sJ}e_j}]&\text{if $\chi\in \bar{\Lscr}_{-C_1}\setminus \bar{\Lscr}_{-C_2}$}.
\end{cases}
\end{multline*}
\item
$N(\mu_C)([\bar{P}_\chi])=K^0(\bar{\phi}_{\iota\mu,C})([\bar{P}_\chi])=[\bar{P}_{\chi-\iota\mu}]$ for $\mu\in X(T)$ and $\chi\in \Lscr_{-C}$, $C\in \Cscr^0$.
\end{enumerate}
\end{proposition}
\begin{proof}
Everything follows from the definition of the functor $\Res$.
\begin{enumerate}
\item This is a direct application of the definition.
\item This follow from Lemma \ref{cor:Clambdachi} together with the fact, asserted in the proof of Lemma \ref{eq:complexes}, that  $W^{\lambda,+}$ (with $\lambda$ as in Lemma \ref{cor:Clambdachi}) is spanned by the weight vectors $e_j$ such that
  $\langle\lambda,b_j\rangle \ge 0$ which in the current setting is precisely the set $J$.
\item This follows by using the $\TT$-equivariant version of $\phi_{\mu,C}$ (
see Proposition \ref{prop:schober}) combined
with the splitting $\TT\cong T\times H$ (see \eqref{eq:times}) together with Observation \ref{obs:stab}.\qedhere
\end{enumerate}
\end{proof}

\subsection{Specialisation}\label{subsec:specialisation}
For $h\in H$ and $\ast\in \{c,f\}$ we put
\begin{align*}
K_h^0(\bar{\Sscr}^\ast)=K^0(\bar{\Sscr}^\ast)\otimes_{\ZZ[X(H)],h}\CC\in \KS(X(T),\Hscr)
\end{align*}
where 
\begin{enumerate}
\item as above we view $K^0(\bar{\Sscr}^\ast)$ as objects in $\KS_{\ZZ[X(H)]}(X(T),\Hscr)$ using Observation \ref{obs:stab} and \eqref{eq:times};
\item we use the base extension functor \eqref{eq:baseextension} for the ring morphism 
\[
h:\ZZ[X(H)]\r \CC:\chi\mapsto \chi(h).
\]
\end{enumerate}
We record the following trivial lemma. 
\begin{lemma}\label{lem:order}
Let $M\in \KS_{\ZZ[X(H)]}(X(T),\Hscr)$. Then $(M^-)_h=(M_{h})^-$. If $M(C)$ is a finitely generated projective $\ZZ[X(H)]$ module
for all $C\in \Cscr$ then $(\DD_{\ZZ[X(H)]}M)_h=\DD(M_h)$.
\end{lemma}
\subsection{Specialisation and monodromy}
\begin{proposition}\label{lem:supportschobermonodromy1}
  For $h\in
  H$ let $N_h=\Res(K^0_h(\bar{\Sscr}^c))\in \Rep(\Pi(\Hscr)\rtimes
  X(T))$ (see \eqref{eq:res}). 
  We write $[\overline{P}_\chi]_h$
  for the image of $[\overline{P}_\chi]$
  in $K^0_h(\bar{\Sscr}^c)$.
  Choose $\alpha\in
  Y(H)_\CC$ in such a way that $e^{-2\pi
    i\alpha}=h$.  
\begin{enumerate}
\item For $C\in \Cscr^0$ we have
\[
N_h(C)=K^0_h(\bar{\Escr}^c_C).
\]
\item 
For $C_1,C_2\in \Cscr^0$ such that  $\dim C_1\wedge C_2=n-1$, denote $C_0=C_1\wedge C_2$.  
 Set $J=J_{-C_0,-C_2}$ (see  \eqref{eq:JC0C1}). 
 The map $\ggamma_{C_1C_2}$ evaluated on $[\bar{P}_{\iota\chi}]_h$ is given by 
\begin{multline*}\label{eq:K0phi}
N_h(\ggamma_{C_1C_2})([\bar{P}_{\iota\chi}]_h)=K^0_f(\bar{\phi}_{C_1C_2})[\bar{P}_{\iota\chi}]_h=[\bar{\phi}_{C_1C_2}(\bar{P}_{\iota\chi})]_h=\\
\begin{cases}
[\bar{P}_{\iota\chi}]_h&\text{if $\chi\in \Lscr_{-C_1}\cap \Lscr_{-C_2}$},\\
\sum_{\emptyset\neq \sJ\subset J} (-1)^{|\sJ|+1}\left(\prod_{j\in \sJ}e^{-2\pi i \gamma_j}\right)[\bar{P}_{\iota\chi-\sum_{j\in \sJ}\iota b_j}]_h&\text{if $\chi\in \Lscr_{-C_1}\setminus \Lscr_{-C_2}$}.
\end{cases}
\end{multline*}
where  $\gamma$ is the unique element of $Y(\TT)_\CC$ such that $A\gamma=\alpha$ and $\iota \gamma=0$.
\item
$N_h(\mu_C)([\bar{P}_{\iota\chi}]_h)=K^0(\bar{\phi}_{\iota\mu,C})([\bar{P}_{\iota\chi}]_h)=[\bar{P}_{\iota\chi-\iota\mu}]_h$ for $\mu\in X(T)$ and $\chi\in \Lscr_{-C}$, $C\in \Cscr^0$.
\end{enumerate}
\end{proposition}
\begin{proof} Most of the claims follow immediately from Proposition \ref{lem:supportschobermonodromy}. The only claim that requires some thoughts 
is the case $\chi\in \Lscr_{-C_1}\setminus \Lscr_{-C_2}$ in (2). We may write the corresponding equation in Proposition \ref{lem:supportschobermonodromy}(2) as, denoting for simplicity $\{\mu\}:=K^0(\bar{\phi}_{\mu,C})$ for $\mu\in X(\TT)$, 
\begin{equation}
\label{eq:lastline}
\begin{aligned}
N(\ggamma_{C_1C_2})([\bar{P}_{\iota\chi}])&=
\sum_{\emptyset\neq \sJ\subset J} (-1)^{|\sJ|+1}[\bar{P}_{\iota\chi-\sum_{j\in \sJ}e_j}]
\\
&=\sum_{\emptyset\neq \sJ\subset J} (-1)^{|\sJ|+1}
\left(\prod_{j\in \sJ}\{\iota b_j-e_j\}\right)
[\bar{P}_{\iota\chi-\sum_{j\in \sJ}\iota b_j}]
\\
&=\sum_{\emptyset\neq \sJ\subset J} (-1)^{|\sJ|+1}
\left(\prod_{j\in \sJ}\{-A^*\kappa e_j\}\right)[\bar{P}_{\iota\chi-\sum_{j\in \sJ}\iota b_j}]
\end{aligned}
\end{equation}
Here the factor $\{-A^*\kappa e_j\}$ represents the action of 
$\kappa e_j\in X(H)$ (the origin of the sign change is the fact that $\bar{\phi}_{\mu,C}=(-\mu)\otimes-$). Now we compute
\begin{equation}
\label{eq:magickappa}
(\kappa e_j)(h)=e^{-2\pi i\langle \kappa e_j ,\alpha\rangle}=e^{-2\pi i\langle  e_j ,\kappa\alpha\rangle}=e^{-2\pi i\langle  e_j ,\gamma\rangle}
\end{equation}
where in the last equality we have used the easily proved fact that $\gamma=\kappa\alpha$. Specialising \eqref{eq:lastline} at $h$ and substituting 
\eqref{eq:magickappa} yields what we want.
\end{proof}

\begin{example}\label{ex:conifoldHmon}
Let $(T,W)$ be as in Example \ref{ex:example1}. From Proposition \ref{lem:supportschobermonodromy1} we obtain 
\[
N_h(\nu_{C_1C_2})([\bar{P}_{\iota (-1)}]_h)=(e^{-2\pi i\gamma_3}+e^{-2\pi i\gamma_4})[\bar{P}_{\iota (0)}]_h-e^{-2\pi i(\gamma_3+\gamma_4)}[\bar{P}_{\iota (1)}]_h.
\]
Analagously,
\[
N_h(\nu_{C_2C_1})([\bar{P}_{\iota (1)}]_h)=(e^{-2\pi i\gamma_1}+e^{-2\pi i\gamma_2})[\bar{P}_{\iota(0)}]_h-e^{-2\pi i(\gamma_1+\gamma_2)}[\bar{P}_{\iota (-1)}]_h.
\]
Descending $N_h$ from $(\CC\setminus \ZZ)/\ZZ$ to $\CC^*\setminus \{1\}$ we find (applying $N_h(\nu_{C_2C_1})N_h(\nu_{C_1C_2})$) that the monodromy around $1$ in the basis $[\bar{P}_{\iota (0)}]_h,[\bar{P}_{\iota (-1)}]_h$ equals 
\[
\left(
\begin{matrix}
1&e^{-2\pi i \gamma_3}+e^{-2\pi i \gamma_4}-e^{-2\pi i (\gamma_3+\gamma_4)}(e^{-2\pi i \gamma_1}+e^{-2\pi i \gamma_2})\\
0&e^{-2\pi i (\gamma_1+\gamma_2+\gamma_3+\gamma_4)}
\end{matrix}
\right).
\]
Setting $\gamma_1=-a$, $\gamma_2=-b$, $\gamma_3=c-1$, $\gamma_4=0$ we obtain the monodromy   of the Gaussian hypergeometric equation with parameters $(a,b,c)$ (cf. Example \ref{eq:Gauss}) from \cite[Theorem 3.5]{BeukersHeckman}, obtained by \cite[Theorem 1.1]{Levelt}.\footnote{In \cite{BeukersHeckman},  $A_1=-e^{2\pi i a}-e^{2\pi i b}$, $A_2=e^{2\pi i (a+b)}$, $B_1=-e^{2\pi i c}-1$, $B_2=e^{2\pi i c}$.}

See also Example \ref{ex:conifoldcom} for the case $h=1$. 
\end{example}

\subsection{Specialisation and autoduality}
\begin{proposition}\label{prop:dual}
Let $h\in H$. Then $K^0_h(\bar{\Sscr}^c)^-\cong \DD(K^0_{h^{-1}}(\bar{\Sscr}^f))$ in $\KS(X(T),\Hscr)$.
\end{proposition}
\begin{proof}
This follows from Corollary \ref{cor:Hdual}, using Lemma \ref{lem:order}.
\end{proof}
Now we state our crucial technical result.
\begin{proposition}\label{prop:duality} 
Let $h\in H$. The map
\[
K^0_h(\bar{\Sscr}^f)\r K^0_h(\bar{\Sscr}^c) 
\]
obtained by applying the functor $K^0_h(-)$ to the inclusion $\bar{\Sscr}^f\subset \bar{\Sscr}^c$ is an isomorphism if $h\in H^{\nres}$.
\end{proposition}
\begin{proof}
We will prove that
\begin{equation}
\label{eq:bothsides}
K^0(\bar{\Sscr}^f)\r K^0(\bar{\Sscr}^c) 
\end{equation}
becomes an isomorphism after inverting an element $F\in \CC[X(H)]\cong \CC[H]$ which defines the non-resonant locus. Note
that using our general conventions $\mu\in X(H)$ acts by $\mu^{-1}\otimes-$ on both sides of \eqref{eq:bothsides} but in this context using 
$\mu^{-1}$ is irrelevant since it follows from the definition of the non-resonant locus that it is invariant under $\mu\mapsto \mu^{-1}$.
So in the rest of this proof we reduce the amount of sign confusion by having $X(H)$ act via $\mu\mapsto \mu\otimes-$.

\medskip

If $M,N\in D(W/\TT)$ then $\Hom_{W/T}(M,N)$ is a $\TT/T=H$-representation, or equivalently a $X(H)$-graded vector space
which is
moreover  finite dimensional in every degree whenever $M$ and $N$ are coherent.
If we put 
\[
\bar{\Lambda}_C:=\Hom_{W/T}\left(\bigoplus_{\chi\in -\Lscr_C} \bar{P}_{\iota \chi}\right)
\]
then we see that $\bar{\Lambda}_{C}$ is an $X(H)$-graded ring which is isomorphic to $\Lambda_C$ if we forget the $X(H)$-grading. Since $\Lambda_C$ has finite global dimension by  \cite[Theorem 5.6]{SVdB10},
we deduce from \cite[Corollary I.7.8]{NVO} that $\bar{\Lambda}_C$ has finite $X(H)$-graded global dimension.

\medskip

We discuss the $X(H)$-grading in more detail.  Note that
$\bar{\Lambda}_C$ is itself a finitely generated $X(H)$-graded
$\CC[W]^T=\Sym(W^\ast)^T$-module. Let $\sigma$ be the cone in $Y(H)_\RR$ spanned by $(a_i)_{i=1,\ldots,d}$
and let $\sigma^\vee\subset X(H)_\RR$ be the dual cone. Then by standard toric geometry (as $\sigma$ is full dimensional and hence $\sigma^\vee$ is strongly convex), see e.g. \cite[(5.1.4)]{CoxLittleSchenck}, 
\begin{equation}
\label{eq:support}
\{\mu \in X(H)\mid \Sym(W^\ast)^T_{\mu}\neq 0\}=-\sigma^\vee\cap X(H)
\end{equation}
(the $-$ sign is because the weights of $W^\ast$ are $(-b_i)_{i=1,\ldots,d}$) and moreover 
\begin{equation}\label{eq:support0}
\dim \Sym(W^\ast)^T_\mu\in \{0,1\}\; \text{for $\mu\in X(H)$}.
\end{equation}
 For use below we let $\mm_1,\ldots \mm_\ell$
be monomials in $\CC[W]^T$ (with respect to the canonical basis of $W^\ast\cong \CC^d$) whose degrees $m_j:=|\mm_j|\in X(H)$ are generators  for the one dimensional
cones in the boundary of $-\sigma^\vee$. By  \eqref{eq:support}, \eqref{eq:support0}, $\CC[W]^T$ is finitely generated over 
$\CC[\mm_1,\ldots,\mm_d]$.

Below we need to be able to associate a Poincare series to a
finitely generated $X(H)$-graded $\bar{\Lambda}_C$-module. 
Choose $\theta\in \relint \sigma$ (e.g.\ $\theta=\sum_i a_i$)  and let $\ZZ[X(H)]\,\hat{}\,\,$ be the completion of $\ZZ[X(H)]$ for 
the filtration on $X(H)$ induced by $-\langle \theta,-\rangle$.  
For $M$ a finitely generated $X(H)$-graded $\bar{\Lambda}_C$ module
we define its ``Poincare series'' as
\[
H(M)=\sum_{\mu\in X(H)} \dim M_{\mu}\,\mu \in \ZZ[X(H)]\,\hat{}.
\]

\medskip

One easily sees
that \eqref{eq:LambdaCeq} can be enhanced to
\[
\bar{\Escr}_C\cong D(\Gr_{X(H)}(\bar{\Lambda}_C))
\]
where $\Gr_{X(H)}(\bar{\Lambda}_C)$ denotes the category of $X(H)$-graded $\bar{\Lambda}_C$-modules. Moreover we have analogues of Lemmas \ref{lem:Dc} and \ref{lem:correspondence}
\[
\bar{\Escr}^c_C\cong D^b(\gr_{X(H)}(\bar{\Lambda}_C)), \qquad \bar{\Escr}^f_C\cong D^b(\gr^f_{X(H)}(\bar{\Lambda}_C))
\]
where $\gr_{X(H)}(\bar{\Lambda}_C)$ and
$\gr_{X(H)}^f(\bar{\Lambda}_C)$ denote respectively the categories of
$X(H)$-graded $\bar{\Lambda}_C$-modules with finitely generated and
finite dimensional cohomology.\footnote{Note that finite dimensional
  $X(H)$-graded $\CC[W]^T$-modules are automatically supported in
  the origin of $W\quot T$. So the definition of
  $\gr^f_{X(H)}(\bar{\Lambda}_C)$ is a bit simpler than the definition
  of $\mod^f(\Lambda_C)$.}
Hence we have to prove that
\begin{equation}
\label{eq:K0statement}
K^0(\gr^f_{X(H)} (\bar{\Lambda}_C))\r K^0(\gr_{X(H)} (\bar{\Lambda}_C))
\end{equation}
becomes an isomorphism when restricted to $h\in H^{\nres}$.
We may use  the projective resolutions of the simple objects in $\gr^f_{X(H)} (\bar{\Lambda}_C)$ to express a basis  of simples of $K^0(\gr^f_{X(H)} (\bar{\Lambda}_C))$
in terms of a basis of projectives of $K^0(\gr_{X(H)} (\bar{\Lambda}_C))$.
The resulting base change matrix $\Phi$ has entries in $\ZZ[X(H)]$ and
is 
equal to the inverse of the matrix of $X(H)$-graded Poincare
series:
\begin{equation}\label{eq:Poincare_matrix}
\Psi:=H(\Hom_{W/T}(\bar{P}_{\iota\chi_1},\bar{P}_{\iota\chi_2}))_{\chi_1,\chi_2\in -\Lscr_C}\in M_{|\Lscr_C|\times |\Lscr_C|} (\ZZ[X(H)]\,\hat{}\,).
\end{equation}

Put $F=\prod_i (1-m_i)\in \ZZ[X(H)]$. Since
every $\Hom_{W/T}(P_{\chi_1},\allowbreak P_{\chi_2})$ is a finitely
generated $\CC[W]^T$-module, and hence a finitely generated $\CC[\mm_1,\ldots,\mm_\ell]$-module
 it follows that 
\[
\Phi^{-1}=\Psi \in M_{|\Lscr_C|\times |\Lscr_C|} (\ZZ[X(H)]_F).
\]
Hence the specialisation of $\Phi$ at $h$ will be invertible whenever $F(h)\neq 0$. Let
$\alpha$ be a lift of $h$ under the map $Y(H)_\CC\r H:\alpha\mapsto \exp(2\pi i \alpha)$. Then
we have 
\[
|\mm_i|(h)=e^{\displaystyle 2\pi i \langle |\mm_i|,\alpha\rangle}.
\]
It now suffices to invoke 
Lemma  \ref{lem:fanvsgit} below.
\end{proof}

\begin{corollary}
If $h\in H^{\nres}$ 
then 
$
\DD(K_h^0(\bar{\Sscr}^c))\cong (K^0_{h^{-1}}(\bar{\Sscr}^c))^-
$ in $\KS(X(T),\Hscr)$.
\end{corollary}

\begin{proof}
This follows from by combining Propositions \ref{prop:dual} and \ref{prop:duality}.
\end{proof}

The following lemma was used in the proof of Proposition \ref{prop:duality}.
\begin{lemma}\label{lem:fanvsgit} 
  $\alpha\in Y(H)_\CC$.  Then $\alpha$ is non-resonant if and only if
  $\langle m_j,\alpha\rangle\not \in \ZZ$ for $j=1,\ldots,\ell$.
\end{lemma}

\begin{proof}
Put $M=X(H)$, $N=Y(H)$. Using \eqref{eq:support} we identify the monomials in $\CC[W]^T$ with $m\in M$ such that $\langle m,a_i\rangle\leq 0$, $1\leq i\leq d$. 

Assume first that $\alpha$ is resonant. Then there exist $m\in M$, $n\in N$, $S\subset \{1,\dots,d\}$, $\gamma'_i\in \CC$ s.t.  
$\alpha=n+\sum_{i\in S^c}\gamma'_i a_i$,  $\langle m,a_i\rangle=0$ for $i\in S^c$, $\langle m,a_i\rangle<0$ for $i\in S$. 
Then  $\langle m,a_i\rangle \le 0$, $1\leq i\leq d$. 
Therefore we can identify $m$ 
with an element of $\CC[W]^T$. 
Choose $m_j$ in such a way that $m_j$ divides $pm$ in $\CC[W]^T$ for $p\in \NN_{>0}$.
Since $0\geq \langle m_j,a_i\rangle \geq p\langle m,a_i\rangle$ we have $\langle m_j,a_i\rangle=0$ for $i\in S^c$.
Hence, $\langle m_j,\alpha\rangle=\langle m_j,n\rangle\in \ZZ$.

\medskip

On the other hand assume there exists $m_j$ such $\langle m_j,\alpha\rangle:=z\in\ZZ$. 
Since $m_j$ is a generator of a 1-dimensional cone, $m_j$ is  is not a multiple of a $m'\in M$, hence there exists $n\in N$ such that $\langle m,n\rangle=z$. Then 
\[
\langle m_j,\alpha-n\rangle=0. 
\]
Since by duality $m_j$ defines a supporting hyperplane of $\sigma$ this implies that $\alpha$ lies in $N+\Iscr_0$; i.e. $\alpha$ is resonant. 
\end{proof}

\section{The decategorified GIT $\Hscr$-schober and the GKZ system}\label{sec:decatGKZ}
Now for $h\in H$ we put
\[
\tilde{K}^0_h(\bar{\Sscr}^c):=K^0_h(\bar{\Sscr}^c)\,\tilde{}\in \Perv(X(T)_\CC)
\]
using the notation $\tilde{?}$ introduced in \S\ref{subsec:equivperv} and we let $S^c(h)$ be the corresponding perverse sheaf on $X(T)_\CC/X(T)\cong T^\ast$. The following theorem,
which is the main
result of this paper, shows that $S^c(h)$, obtained by decategorifying and specializing the 
GIT $\Hscr$-schober, is, up to a suitable translation, equal to the solution sheaf of a GKZ system, whenever $h$ is non-resonant.

\medskip

Denote $\bar{\kkappa}=e^{2\pi i \kkappa}$, where $\kkappa$ is as introduced in \S\ref{subsec:GKZdislocus}, let $\bar{\tau}_{\bar{\kkappa}^{-1}}$ denote 
 the translation $T^*\to T^*$,  $x\mapsto \bar{\kkappa}^{-1} x$.

\begin{theorem}\label{thm:connect} 
Assume that $\alpha\in\mathfrak{h}= Y(H)_\CC$ is non-resonant. 
Then 
\[
\iP(\alpha)\cong \bar{\tau}_{\bar{\kkappa}^{-1}}^*S^c(e^{-2\pi i\alpha}). 
\] 
\end{theorem}
The proof of this theorem will be given in the rest of this section. 
As before $V(E_A)\subset T^\ast$ is the GKZ discriminant locus defined in \S\ref{subsec:GKZdislocus} and $j:T^*\setminus V(E_A)\hookrightarrow T^*$ is the inclusion.
\subsection{Perverse sheaves as intermediate extension}
\label{sec:intermediate}
We now show that $S^c(h)$ for $h\in H^{\nres}$ is obtained as the intermediate extension of its restriction to the regular locus, $T^*\setminus V(E_A)$.

\begin{corollary}\label{cor:interext}
One has $\bar{\tau}_{\bar{\kkappa}^{-1}}^*S^c(h)\cong j_{!*}(j^* \bar{\tau}_{\bar{\kkappa}^{-1}}^*S^c(h))$ on $T^*$.
\end{corollary}

\begin{proof}
Write $\bar{E}^c_C=K^0(\bar{\Escr}^c_C)$.
From Lemma \ref{lem:union0} below, for $C\in \Cscr\setminus \Cscr^0$ we have $\bar{E}^c_{C}=\sum_{C<C^0\in\Cscr^0}\bar{E}^c_{C^0}$, which remains true after the specialisation.  
Hence, $S^c(h)$ has no \emph{perverse quotient sheaf} supported on
$\bar{\zeta} V(E_A)$.  Since the cokernel of the map
${}^pj_!(j^* \bar{\tau}_{\bar{\kkappa}^{-1}}^*S^c(h))\to \bar{\tau}_{\bar{\kkappa}^{-1}}^*S^c(h)$ is supported on $V(E_A)$ it follows that 
this must be an epimorphism of perverse
sheaves. Applying Proposition \ref{prop:duality} with $h$ replaced by $h^{-1}$
(using the
assumption that $h\in H^{\nres}$ and hence $h^{-1}$ in $H^{\nres}$), it follows that $S^c(h)$ has no
\emph{perverse subsheaf} supported on $\bar{\zeta}V(E_A)$, which implies that
$\bar{\tau}_{\bar{\kkappa}^{-1}}^*S^c(h)\to {}^pj_*(j^*\bar{\tau}_{\bar{\kkappa}^{-1}}^*S^c(h))$ is also a monomorphism.  Hence
\[
j_{!*}(j^*\bar{\tau}_{\bar{\kkappa}^{-1}}^*S^c(h)):={}^p\im ({}^pj_!(j^*\bar{\tau}_{\bar{\kkappa}^{-1}}^*S^c(h))\r {}^pj_*(j^*\bar{\tau}_{\bar{\kkappa}^{-1}}^*S^c(h)))
\cong \bar{\tau}_{\bar{\kkappa}^{-1}}^*S^c(h).\qedhere
\]  
\end{proof}

We have used the following combinatorial lemma. 
\begin{lemma}\label{lem:union0}
Let $C\in \Cscr\setminus \Cscr^0$. Then 
\begin{align*}
\Lscr_C&=\bigcup_{C_0\in \Cscr^0,C_0>C}\Lscr_{C_0},\\
\bar{\Lscr}_C&=\bigcup_{C_0\in \Cscr^0,C_0>C}\bar{\Lscr}_{C_0}.
\end{align*}
\end{lemma}

\begin{proof}
The second claim follows from the first one by taking inverse images under~$B$. So we now check the first claim.
The inclusion $\bigcup_{C_0\in \Cscr^0,C_0>C}\Lscr_{C_0}\subset \Lscr_C$ follows from \eqref{eq:inclusion_Lscr_C}.

To show the converse let $\rho\in C$. Assume that
$\chi\in \Lscr_{C}=(\rho+\Delta)\cap X(T)$. 
We have
$\rho\in \chi+\Delta$ (using $\Delta=-\Delta$). Since $\chi+\Delta$ is convex of dimension
$n$
it must intersect some $C_0>C$ for $C_0\in \Cscr^0$.
Let $\rho_0 \in (\chi+\Delta)\cap C_0$. Then $\chi\in \rho_0+\Delta$ (using again $\Delta=-\Delta$)
and hence $\chi\in \Lscr_{C_0}$.
\end{proof}

\subsection{Comparison of the monodromy}\label{subsec:comparemonodromy}
We first compare the corresponding  local systems on 
$T^*\setminus V(E_A)=(X(T)_\CC\setminus (\Hscr_\CC+\zeta))/X(T)$. 
\begin{proposition}\label{cor:connect}
  Assume that $\alpha\in \sum_i \RR_{<0}
  a_i$ 
  is non-resonant. 
Then  $j^*\iP(\alpha)\cong j^*\bar{\tau}^*_{\bar{\kkappa}^{-1}}S^c(e^{-2\pi i\alpha})$. 
The isomorphism corresponds to  the 
 isomorphism of the corresponding 
$\Pi(\Hscr)\rtimes X(T)$-representations $M$ (Theorem \ref{th:mainth1}), $N_h$ (Proposition \ref{lem:supportschobermonodromy1}) given by 
$
\hat{M}^\alpha_\chi\mapsto [\bar{P}_{-\iota\chi}]_h
$ for $h=e^{-2\pi i\alpha}$. 
\end{proposition}
\begin{proof}
This follows from Theorem \ref{th:mainth1} and Proposition \ref{lem:supportschobermonodromy1}.
\end{proof}

\subsection{Proof of  Theorem \ref{thm:connect}}\label{subsec:perversecomparison}

Let $\alpha'\in (\alpha+N)\cap \sum_i \RR_{<0}a_i$. 
From Proposition \ref{prop:iPPgeneral}(4) and Proposition \ref{cor:connect} we obtain isomorphisms
\[
j^*\iP(\alpha)\cong j^*\iP(\alpha')\cong j^*\bar{\tau}^*_{\kkappa^{-1}}S^c(e^{-2\pi i\alpha'})=j^*\bar{\tau}^*_{\kkappa^{-1}}S^c(e^{-2\pi i\alpha}).
\]

Since $\iP(\alpha)$ and $S^c(e^{-2\pi i\alpha})$ are intermediate extensions of the restrictions to $T^*\setminus V(E_A)$ by Corollaries \ref{cor:intermediateextension}, \ref{cor:interext}, the result follows.

\begin{remark}\label{rem:referee}
The referee has suggested a simplification of our proof that we outline here. It would be worthwhile to  work out the details. 

To prove Theorem \ref{thm:connect} it suffices to show, besides Proposition
\ref{cor:connect} and the claim, proved in the first paragraph of
Corollary \ref{cor:interext}, that $S^c(h)$ has no perverse quotient
sheaf supported on $\bar{\zeta}V(E_A)$, 
that the sum of the multiplicites in the characteristic cycles of
$S^c(h)$ and the GKZ perverse sheaf $P(\alpha)$ coincide.

By the construction of the KS data \cite[\S(0.3)(d)]{KapranovSchechtman}, together with \cite[Theorem]{Takeuchi}, see also \cite[Proposition 1.3.3(b)]{KapranovSchechtmanShuffle}, the sum of multiplicites for $S^c(h)$  is given by $|\Lscr_{C_0}|$ where $C_0\in\Cscr$ is a point.

For the GKZ perverse sheaf the characteristic cycle is computed in \cite[\S2.1, Theorem 5]{GKZhyper}. The multiplicities are expressed in terms of some polytopes associated to the faces of the convex hull of $A$. The subtle point here is that in loc.cit. the GKZ system is considered as living on $\CC^N$, and not on $\TT^*$ as we do. Hence, the characteristic variety in loc.cit. has more irreducible components. Therefore the counting must be adapted appropriately (we have have not carried out this step).
\end{remark}

\appendix

\section{Descent for weakly equivariant $\Dscr$-modules}
In order to have a convenient reference in the body of the paper, we summarize a few facts on weakly equivariant $\Dscr$-modules. No originality is intended.
All schemes are separated of finite type over an algebraically closed field
of characteristic zero. The main result is Corollary \ref{cor:maincor}.
\subsection{Weakly equivariant $\Dscr$-modules}
\subsubsection{Generalities}
\label{sec:weq}
Let $Y/k$ be smooth scheme and let $G$ be a reductive group acting on $Y$. Put
$\mathfrak{g}=\Lie(G)$. Let $\alpha\in \mathfrak{g}^\ast$ be a character of $\mathfrak{g}$, i.e. $\alpha([\mathfrak{g},\mathfrak{g}])=0$.

If $\Mscr$ is a $G$-equivariant $\Oscr_Y$-module then differentiating the $G$-action yields a Lie algebra action\footnote{One way to obtain this action
is to extend the base ring to $k[\epsilon]/(\epsilon^2)$ and to use that $\mathfrak{g}=\ker(G(k[\epsilon])\r G(k))$. The $\mathfrak{g}$ action
is then obtained from the $G(k[\epsilon])$-action on $\Mscr[\epsilon]$.}
\[
\gamma:\mathfrak{g}\otimes\Mscr\r\Mscr.
\]
If $\Mscr=\Oscr_Y$ then this action is by derivations and hence we obtain an ``anchor'' morphism $\rho:\mathfrak{g}\r \Gamma(Y,\Tscr_Y)$.
Finally the $\mathfrak{g}$-action and the $\Oscr_Y$-action on $\Mscr$ are related by the Leibniz identity.

\begin{definition} \label{def:character}
A \emph{weakly $G$-equivariant} $\Dscr_Y$-module  is a quasi-coherent $(G,\Dscr_Y)$-module $\Mscr$. For such $\Mscr$ we say that it has character $\alpha$
if
for every $v\in \mathfrak{g}$ and for every local section $m$ of $\Mscr$ one has $\gamma(v,m)=(\rho(v)-\alpha(v))m$ where on the right-hand side we
used the action obtained via the inclusion $\Oscr_Y\oplus \Tscr_Y\subset \Dscr_Y$.
\end{definition}
We will denote the category of weakly $G$-equivariant $\Dscr_Y$-modules by $\Qch(G,\Dscr_Y)$. The full subcategory of weakly $G$-equivariant $\Dscr_Y$-modules with character $\alpha$
is denoted by $\Qch_\alpha(G,\Dscr_Y)$.

\subsection{The canonical weakly equivariant $\Dscr$-module}
\label{sec:canonical}
We keep notation as in \S\ref{sec:weq}. Here and below we write
$\mathfrak{g}-\alpha(\mathfrak{g})$ for the vector space $\{v-\alpha(v)\mid v\in\mathfrak{g}\}\subset k\oplus \mathfrak{g}$. We identify $\mathfrak{g}-\alpha(\mathfrak{g})$ with the corresponding
sections of $\Oscr_Y\oplus \Tscr_Y$, via the map $\rho$ introduced above (which does not have to be injective).
\begin{lemma}\label{lem:A2} The $\Dscr_Y$-module $\Dscr_{Y,\alpha}=\Dscr_Y/\Dscr_Y(\mathfrak{g}-\alpha(\mathfrak{g}))$
is weakly $G$-equivariant with character $\alpha$.
\end{lemma}
\begin{proof} Put $\Mscr=\Dscr_Y/\Dscr_Y(\mathfrak{g}-\alpha(\mathfrak{g}))$. It is clear
that $\Mscr$ is weakly $G$-equivariant. Let $D$ be a local section of $\Dscr_Y$ and let
$\bar{D}$ be the corresponding local section of $\Mscr$.

One checks that the differential of the $G$-action on $\Dscr_Y$ is given by
$\gamma(v,D)= [\rho(v),D]$. Whence we also have $\gamma(v,\bar{D})=\overline{[\rho(v),D]}$. 
So we have to prove for all $v\in \mathfrak{g}$,
$\overline{[\rho(v),D]}=\overline{(\rho(v)-\alpha(v))D}$.
We compute
\begin{align*}
[\rho(v),D]&=\rho(v)D-D\rho(v)\\
&\cong \rho(v)D-D\alpha(v)\qquad \text{modulo $\Dscr_Y(\mathfrak{g}-\alpha(\mathfrak{g}))$}\\
&=(\rho(v)-\alpha(v))D.\qedhere
\end{align*}
\end{proof}
\subsection{Descent for weakly equivariant $\Dscr$-modules}
\subsubsection{Principal homogeneous spaces}
In the rest of this appendix we keep the notations as in \S\ref{sec:weq} but now we consider a principal homogeneous $G$-space $\pi:Y\r X$.
We put $\Dscr_{X,\alpha}:=(\pi_\ast \Dscr_{Y,\alpha})^G$. Note that we have a short exact sequence
\begin{equation}
\label{eq:surjection}
0\r (\pi_\ast \Dscr_Y(\mathfrak{g}-\alpha(\mathfrak{g})))^G\r
(\pi_\ast \Dscr_Y)^G\r \Dscr_{X,\alpha}\r 0.
\end{equation}
\begin{lemma} $(\pi_\ast \Dscr_Y(\mathfrak{g}-\alpha(\mathfrak{g}))^G$ is a two sided
ideal in $(\pi_\ast \Dscr_Y)^G$. Hence
 multiplication of local sections in $\Dscr_{Y}$ induces a multiplication
on  $\Dscr_{X,\alpha}$ via \eqref{eq:surjection}. 
\end{lemma}
\begin{proof}
Let $v\in\mathfrak{g}$ and let $D$ be a local section of $(\pi_\ast \Dscr_Y)^G$. We have
to show that $(\rho(v)-\alpha(v))D$ is a local section of $(\pi_\ast \Dscr_Y(\mathfrak{g}-\alpha(\mathfrak{g}))^G$. This follows from the fact that the differentiated $G$-action on $\Dscr_Y$
is given by $v\mapsto [\rho(v),-]$ (as was already used in the proof of Lemma \ref{lem:A2}), and hence, since $D$ is $G$-invariant, we have $[\rho(v),D]=0$.
\end{proof}
We observe that $\Dscr_{Y,\alpha}$ is a $G$-equivariant $(\Dscr_Y,\pi^{-1}\Dscr_{X,\alpha})$-bimodule.

\begin{proposition} \label{prop:eq}
There are inverse equivalences
\begin{gather}
\Qch_\alpha(G,\Dscr_Y)\r \Qch(\Dscr_{X,\alpha}):\Mscr\mapsto (\pi_\ast \Mscr)^G,\label{eq:inv}\\
\Qch(\Dscr_{X,\alpha})\r \Qch_\alpha(G,\Dscr_Y) :\Nscr\r \Dscr_{Y,\alpha}\otimes_{\pi^{-1}\Dscr_{X,\alpha}} \pi^{-1}\Nscr.
\end{gather}
\end{proposition}
\begin{proof}
We have for $\Nscr\in \Qch(\Dscr_{X,\alpha})$:
\[
(\pi_\ast(\Dscr_{Y,\alpha}\otimes_{\pi^{-1}\Dscr_{X,\alpha}} \pi^{-1}\Nscr))^G=(\pi_\ast\Dscr_{Y,\alpha})^G \otimes_{\Dscr_{X,\alpha}} \Nscr=\Nscr
\]
(to check this it is convenient to reduce to the case that $X$, and hence $Y$ are affine).

Now assume $\Mscr\in \Qch_\alpha(G,\Dscr_Y)$. We have to prove that the canonical map
\[
\Dscr_{Y,\alpha}\otimes_{\pi^{-1} \Dscr_{X,\alpha}} \pi^{-1}(\pi_\ast \Mscr)^G\r \Mscr
\]
is an isomorphism. Let $\Kscr$, $\Cscr$ be its kernel and cokernel. Applying the exact functor $\pi_\ast(-)^G$ we obtain by the first part of the proof  $(\pi_\ast \Kscr)^G=(\pi_\ast \Cscr)^G=0$.
It then follows by descent that $\Kscr=\Cscr=0$.
\end{proof}
\subsubsection{The special case that $X$ is a point}
\begin{lemma} \label{lem:point}
Assume $Y=G$ and $X=\Spec k$. Then the inclusion $\Oscr_G\hookrightarrow \Dscr_G$ induces an isomorphism $\Oscr_G\cong \Dscr_{G,\alpha}$. In particular $\Dscr_{X,\alpha}=(\pi_\ast \Oscr_G)^G=\Oscr_X$.
\end{lemma}
\begin{proof} The anchor map $\Oscr_G\otimes_k \mathfrak{g}\r \Tscr_G$
  is an isomorphism in this case. If we filter $\Dscr_G$ by order of
  differential operators we obtain from this
  $\gr \Dscr_{G,\alpha}=\Oscr_G$. This yields the claim in the lemma.
\end{proof}
\begin{corollary} \label{cor:unit} The equivalences in Proposition \ref{prop:eq} specialize to an equivalence $\Qch_\alpha(G,\Dscr_G)\cong \Mod(k)$. Moreover the functor  $\Qch_\alpha(G,\Dscr_G)\r\Mod(k)$
defined by \eqref{eq:inv}
is naturally isomorphic to $i_e^\ast$ where $e\in G$ is the unit element.
\end{corollary}
\begin{proof} 
The first statement is just the specialisation of Proposition \ref{prop:eq} to the case $Y=G$, $X=\Spec k$. 
For the second claim note that both functors, the one defined by \eqref{eq:inv} and  $i_e^\ast$, factor through $\Qch_\alpha(G,\Dscr_G)\r \Qch(G,\Oscr_G)$. Hence it is sufficient to prove that $(-)^G$ and $i^\ast_e$ define naturally isomorphic functors
on $\Qch(G,\Oscr_G)$. This is standard descent.
\end{proof}
\begin{corollary} $i^\ast_e:\Qch_\alpha(G,\Dscr_G)\r \Qch(\Oscr_e)\cong \Mod(k)$ is also an equivalence of
categories.
\end{corollary}
\begin{remark} \label{rem:invertible}
The fact that $\Dscr_{G,\alpha}\cong \Oscr_G$ in $\Qch(G,\Oscr_G)$ shows in particular that $\Dscr_{G,\alpha}$ is invertible for the tensor product of $\Dscr$-modules. 
Write
$\theta_\alpha$ for the image of $1\in \Gamma(G,\Dscr_G)$. Then $\Dscr_{G,\alpha}=\Oscr_G\theta_\alpha$ where $\theta_\alpha$ is $G$-invariant and satisfies $\rho(v)\theta_\alpha=\alpha(v)\theta_\alpha$
for $v\in \mathfrak{g}$.

Let us specialize to the case that $T$ is the torus $(\CC^*)^n$. In that case $\mathfrak{t}:=\Lie(T)=\CC^n$ and $\rho(e_i)=\partial/\partial x_i$ for $e_i$ the canonical $i$'th basis vector of $\CC^n$.
 Let $\alpha\in \mathfrak{t}^\ast=\CC^n$ be given by $(\alpha_1,\ldots,\alpha_n)$.
Then $\theta_\alpha$ is the multi-valued function given by
$\theta_\alpha(t_1,\ldots,t_n)=t_1^{\alpha_1}\cdots t_n^{\alpha_n}$ for $t_i\in \CC^\ast$. More intrinsically we have $\theta_\alpha\circ e^{2\pi i-} =e^{2\pi i\langle \alpha,-\rangle}$.
\end{remark}
\subsubsection{The case of split principal homogeneous spaces}
\begin{lemma} Assume $Y=G\times X$. 
Then the inclusion
  $\Oscr_G\boxtimes\Dscr_X\hookrightarrow \Dscr_{G\times X}$ induces an isomorphism
  $\Oscr_G\boxtimes\Dscr_X\cong \Dscr_{G\times X,\alpha}$. In particular, 
  $\Dscr_{X,\alpha}\cong(\pi_\ast (\Oscr_G\boxtimes \Dscr_X ))^G\cong\Dscr_X$. 
\end{lemma}
\begin{proof} This follows from Lemma \ref{lem:point}, taking into account $\Dscr_{G\times X}=\Dscr_G\boxtimes \Dscr_X$, $\Dscr_{G\times X,\alpha}=\Dscr_{G,\alpha}\boxtimes \Dscr_X$.
\end{proof}
\begin{corollary} \label{cor:previous}
Assume $Y=G\times X$. 
The equivalences in Proposition \ref{prop:eq} specialize to an equivalence $\Qch_\alpha(G,\Dscr_{G\times X})\cong \Qch(\Dscr_X)$. Moreover, the associated functor  $\Qch_\alpha(G,\Dscr_G)\r\Qch(\Dscr_X)$ 
is naturally isomorphic to $i_e^\ast$ where $e:X\r G\times X$ is the unit section.
\end{corollary}
\begin{proof} This is a generalization of Corollary \ref{cor:unit} with an extra $\Dscr_X$-action.
\end{proof}
\begin{corollary} \label{cor:maincor}
 Assume that $\pi:Y\r X$ is a split principal homogeneous $G$-space. Let $i:X\r Y$ be a splitting for $Y$. Then the
$\Dscr$-module inverse image functor $i^\ast:\Qch_\alpha(G,\Dscr_Y)\r \Qch(\Dscr_X)$ is an equivalence of (abelian) categories.  
\end{corollary}
\begin{proof} This follows from Corollary \ref{cor:previous}, after making the identification $Y\cong G\times X$ using the splitting $i$. Then $i=i_e$.
We use \cite[\S4.5]{Borel} for the fact that $i_e^\ast$ does indeed compute the usual $\Dscr$-module inverse image for $X\r G\times X$ (note that in loc.\ cit.\ the $\Dscr$-module inverse image is 
denoted by $(-)^\circ$).
\end{proof}
Corollary \ref{cor:maincor} implies in particular that $i^\ast$ is exact. This also follows from the following lemma.
\begin{lemma} \label{lem:inverse}
In the setting of Corollary \ref{cor:maincor}  one has $L^ji^\ast\Mscr=0$ for $j>0$
and $\Mscr \in \Qch(G,\Oscr_Y)$.
\end{lemma}
\begin{proof}  We may assume $Y=G\times X$ with $i=i_e$. By descent $\Mscr=\pi^\ast \Nscr$ for $\Nscr\in \Qch(X)$. Hence $Li^\ast(\Mscr)=(Li^\ast\circ \pi^\ast)(\Nscr)=(Li^\ast\circ L\pi^\ast)(\Nscr)
=\Nscr$, where we have used that $Y/X$ is flat and $\pi\circ i$ is the identity.
\end{proof}
\subsection{Changing the section}
\begin{lemma} \label{the-sections-they-are-a-changin}
Let $\pi:Y\r X$ be a split $G$-principal homogeneous space and let $i_1,i_2:X\r Y$ be two sections of $\pi$. Define $\delta:X\r G$ via
 $\delta(x)i_1(x)=i_2(x)$ for $x\in X$. Then
for $\Mscr\in \Qch_\alpha(G,\Dscr_Y)$ we have
\[
i_2^\ast \Mscr\cong 
\delta^\ast\Dscr_{G,\alpha}  \otimes_{\Oscr_{X}}         i_1^\ast\Mscr.
\]
\end{lemma}
\begin{proof} We may assume $Y=G\times X$ and $i_1=i_e$, $i_2(x)=(\delta(x),x):=i_\delta(x)$.
 By Corollary \ref{cor:previous} we have $\Mscr=\Dscr_{G,\alpha}\boxtimes \Nscr$ for $\Nscr\in \Qch(\Dscr_X)$ so that
\begin{align*}
i_\delta^\ast\Mscr&=i_{\delta}^\ast(\Dscr_{G,\alpha}\boxtimes \Nscr)\\
&=i_\delta^\ast(\Dscr_{G,\alpha}\boxtimes \Oscr_X) \otimes_{\Oscr_{X}} i_\delta^\ast\pr_2^\ast\Nscr.
\end{align*}
The composition $\pr_2\circ i_\delta$ is the identity so that $i_\delta^\ast\pr_2^\ast\Nscr=\Nscr=i_e^\ast\Mscr$ in $\Dscr_X$. On the other hand $i_\delta^\ast(\Dscr_{G,\alpha}\boxtimes \Oscr_X)=
i_\delta^\ast \pr_1^\ast \Dscr_{G,\alpha}=\delta^\ast \Dscr_{G,\alpha}$.
\end{proof}

\begin{remark} Following up on Remark \ref{rem:invertible} we can think of $\delta^\ast \Dscr_{G,\alpha}$ as $\Oscr_X \theta_{\alpha,\delta}$ with $\theta_{\alpha,\delta}=\theta_\alpha\circ \delta$ where the
derivatives of $\theta_\alpha\circ\delta$ are determined by the chain rule. In the torus example where we represented $\theta_\alpha$ with an explicit multi-valued function, the notation
$\theta_\alpha\circ \delta$ can be taken literally.
\end{remark}

\newpage
\section*{List of Notations}\label{lon}

\vspace{-0.1cm}

\begin{longtable}{@{}cp{0.64\textwidth}p{0.067\textwidth}}
$X(-)$ & characters  & \S\ref{sec:id}\\
$Y(-)$ & 1-parameter subgroups (cocharacters)& \S\ref{sec:id}\\
$L,N$ & free abelian groups such that $N=\ZZ^d/L$& \S\ref{subsec:setting}\\
$T$, $\TT$, $H$ & associated tori such that $\TT\cong (\CC^\ast)^d$, $H=\TT/T$&\S\ref{subsec:setting}\\
$d$, $n$& $\dim \TT$, $\dim T$&\S\ref{subsec:setting}\\
$W$ & $\cong \CC^d$,  tautological representation of $\TT$ & \S\ref{subsec:setting}\\
$A,B$ & toric data, many incarnations & \S\ref{subsec:setting}\\
$(b_i)_i$ & T-weights of $W$, describing $B$ &\S\ref{subsec:setting}\\
$V$ & a real vector space &\S\ref{sec:genhyp}\\
$\Hscr$ & a hyperplane arrangement &\S\ref{sec:genhyp}\\
$\Hscr$ & a hyperplane arrangement constructed from $B$ &\S\ref{sec:hyperplane}\\
$\Hscr_0$ &  central hyperplane arrangement corresponding to $\Hscr$ &\S\ref{sec:genhyp}\\
$\Cscr$ & the faces of the connected components of $V\setminus \Hscr$&\S\ref{sec:genhyp}\\
$\Cscr^0$ & connected components of $V\setminus \Hscr$ (i.e.\ chambers in $\Cscr$)&\S\ref{sec:genhyp}\\
$\Sigma$ & a zonotope constructed from $B$ &\S\ref{sec:hyperplane}\\
$\Delta$ & $(1/2)\bar{\Sigma}$ &\S\ref{sec:hyperplane}\\
$(H_i)_i$& affine hyperplanes defining the facets of $\Delta$ &\S\ref{sec:hyperplane}\\
$C$ & an arbitrary element in $\Cscr$ &\S\ref{sec:hyperplane}\\
$\Lscr_C$ & $(\nu+\Delta)\cap X(T)$ for $\nu\in C$ &\S\ref{sec:hyperplane}\\
$\mathfrak{h}$ & $\Lie(H)$ ($\cong Y(H)_\CC$) & \S\ref{sec:nonresonance}\\
$\Iscr_0,\Iscr'_0$ & central arrangements in $Y(H)_\RR$ built from $A$& \S\ref{sec:nonresonance}\\
$\Iscr,\Iscr'$ & associated affine hyperplane arrangements & \S\ref{sec:nonresonance}\\
$\mathfrak{h}^{\nres}$ & the complement of $\Iscr_\CC$ (the ``non-resonant'' part of $\mathfrak{h}$)& \S\ref{sec:nonresonance}\\
$H^{\nres}$ & the image of $\mathfrak{h}^{\nres}$ under $\exp(2\pi i-)$& \S\ref{sec:nonresonance}\\
$\Dscr_X$ & the sheaf of differential operators on $X$ & \S\ref{ssec:gkz}\\
$E_\phi$ & a derivation of $\Oscr_{\TT^\ast}$ corresponding to $\phi\in\mathfrak{h}^*$ & \S\ref{ssec:gkz}\\
$\alpha$ & an element of $Y(H)_\CC$ (often implicit) & \S\ref{ssec:gkz}\\
$\Pscr(\alpha)$ & the GKZ $\Dscr_{\TT^\ast}$-module with parameter $\alpha$ & \S\ref{ssec:gkz}\\
$\mathfrak{v}$ & the normalised volume of the convex hull of $A$ &\S\ref{subsec:nonresonance}\\
$\iota$ & a splitting of $B^*:T\to \TT$ (also used for derived maps) & \S\ref{subsec:splitting}\\
$\bar{\Pscr}(\alpha)$ & $\iota^*\Pscr(\alpha)$ & \S\ref{subsec:splitting}\\
$E_A$ & the principal $A$-discriminant & \S\ref{subsec:GKZdislocus}\\
$\kkappa$ &  $-\frac{i}{2\pi}\sum_i (\log{|n_j|})b_j\in X(T)_\CC$ for suitable $n_j\in \ZZ$ 
  & \S\ref{subsec:GKZdislocus}\\
$\Mod_{\rm rh}(\Dscr_X)$ & regular holonomic $\Dscr_X$-modules &\S\ref{sec:reminder}\\
$D^b_{\rm rh}(\Dscr_X)$ & corresponding bounded derived category &\S\ref{sec:reminder}\\
$\Perv(X)$ & the category of perverse sheaves on $X^{\an}$ &\S\ref{sec:reminder}\\
${}^pH(-)$ & perverse cohomology &\S\ref{sec:reminder}\\
$\Sol_X$ & the solution functor &\S\ref{sec:reminder}\\
$\Ma(\alpha)$ & $\Sol_{\TT^\ast}(\MMa(\alpha))[d]$ &\S\ref{subsec:gkzperverse}\\
${}^p \iota^*$ & ${}^pH^0(L\iota^{\an,\ast}(?))$ &\S\ref{subsec:gkzperverse}\\
$\bar{P}(\alpha)$ & ${}^p\iota^\ast P(\alpha)$ &\S\ref{subsec:gkzperverse}\\
$j$ & the inclusion $T^*\setminus V(E_A)\hookrightarrow T^*$&\S\ref{subsec:gkzperverse}\\
$j_{!*}$ & intermediate extension functor corresponding to $j$ &\S\ref{subsec:gkzperverse}\\
$\gamma$ & an element in $Y(\TT)_\CC=\CC^d$ with $\alpha=A\gamma$ & \S\ref{sec:MB}\\
$M$ & the Mellin-Barnes integral (a solution to GKZ)& \S\ref{sec:MB}\\
$\hat{M}$ & the single-valued version of $M$ & \S\ref{subsubsec:MBsingle}\\
$M^\alpha$, $\hat{M}^\alpha$ & MB integrals corresponding to $\alpha$ &\S\ref{subsubsec:otherGKZrelations}\\
$\hat{M}_\chi(x)$ & $\hat{M}(x-\chi)$ & \S\ref{subsubsec:basisofsolutions}\\
$\Mscr_C$ & $\{\hat{M}_{\chi} | \chi \in \mathcal{L}_{C}\}$ & \S\ref{subsubsec:basisofsolutions}\\
$\Phi_\gamma$ & a power series solution for the GKZ system & \S\ref{subsec:powerseries}\\
$\hat{\Phi}_\gamma$ & the single-valued power series solution & \S\ref{subsec:powerseries}\\
$\gamma_I$ & an element in $\CC^d$ with $\alpha=A\gamma_I$, $\gamma_{I,i}\in \ZZ$, $i\in I$  & \S\ref{subsec:powerseries}\\
$D_\rho$ & a domain in $(\CC^*)^d$ & \S\ref{subsec:powerseries}\\
$\hat{D}_\rho$ & a domain in $\CC^d$ & \S\ref{subsec:powerseries}\\
$\Iscr_\rho, \hat{\Iscr}_\rho$ & (multi)sets of subsets of $\{1,\dots,d\}$ & \S\ref{subsec:powerseries}\\
$\Phi_I$, $\hat{\Phi}_I$ & $\Phi_{\gamma_I}$, $\hat{\Phi}_{\gamma_I}$ & \S\ref{subsec:powerseries}\\
$\ggamma_{C_1C_2}$ & a path from $\rho_1$ to $\rho_2$, $\rho_i\in C_i$ &\S\ref{subsec:fundgd}\\
$\Pi(\Hscr)$ & an incarnation of the fundamental groupoid of $\Hscr^c_\CC$&\S\ref{subsec:fundgd}\\
$\Mscr\rtimes \Gscr$ & semi-direct product of a groupoid $\Mscr$ and a group $\Gscr$ &\S\ref{subsubsec:equivariance}\\
$J_{C'C}$ & a specific subset of $\{1,\dots,d\}$ &\S\ref{subsec:mainresult}\\
$M$ & the representation of $\Pi(\Hscr)\rtimes X(T)$ given by $\bar{\Pscr}(\alpha)$&\S\ref{subsec:mainresult}\\
$\hat{\Phi}_I^s$ & a normalisation of $\hat{\Phi}_I$ & \S\ref{subsec:connecting}\\
$\Pscr_{\rho}$ & $\{\hat{\Phi}^s_I\mid I\in \hat{\Iscr}_\rho\}$ & \S\ref{subsec:connecting}\\
$\Perv_{\Hscr_\CC}(V_\CC)$ & perverse sheaves on $V_\CC$, stratified using $\Hscr_\CC$& \S\ref{subsec:KSdata}\\
$E_C,\delta_{CC'},\gamma_{C'C}$ & KS-data for objects in $\Perv_{\Hscr_\CC}(V_\CC)$& \S\ref{subsec:KSdata}\\
$\phi_{C_1C_2}$ & $\gamma_{C'C_2}\delta_{C_1C'}$ for $C'\leq C_1,C_2$& \S\ref{subsec:KSdata}\\
$\KS(\Hscr)$ & the category of $\KS$-data (allowing $\dim E_C=\infty$) & \S\ref{subsec:KSdata}\\
$\KS^c(\Hscr)$ & $\{E\in \KS(\Hscr)| \dim(E_C)<\infty\}$ & \S\ref{subsec:KSdata}\\
$\tilde{E}$ & the perverse sheaf associated to $E\in \KS^c(\Hscr)$& \S\ref{subsec:KSdata}\\
$\Res$ & functor $\KS(\Hscr)\to \Rep(\Pi(\Hscr))$ computing monodromy& \S\ref{subsec:KSdata}\\
$\DD(E)$ & dual $\KS$-data for $E\in \KS^c(\Hscr)$ & \S\ref{subsec:KSdata}\\
$\phi_{g,C}$ & isomorphism $E_C\to E_{gC}$, part of equivariant $\KS$-data & \S\ref{subsec:equivperv}\\
$\KS(\Gscr,\Hscr)$, $\KS^c(\Gscr,\Hscr)$ & $\Gscr$-equivariant $\KS$-data & \S\ref{subsec:equivperv}\\
$\Res$ & functor $\KS(\Gscr,\Hscr)\to \Rep(\Pi(\Hscr)\rtimes \Gscr)$ (equivariant $\Res$)& \S\ref{subsec:equivperv}\\
$\DD$ & functor $\KS^c(\Gscr,\Hscr)\to \KS^c(\Gscr,\Hscr)^\circ$ (equivariant $\DD$)& \S\ref{subsec:equivperv}\\
$\KS_R(\Hscr)$, $\KS_R^c(\Hscr)$  & $R$-linear $\KS$-data &\S\ref{subsec:Rlin}\\
$\Eescr$ & a general $\Hscr$-schober & \S\ref{sec:XT}\\
$\Eescr_C,\delta_{CC'},\gamma_{C'C}$ & $\Hscr$-schober data (categorified $\KS$-data)& \S\ref{sec:XT}\\
$\phi_{C_1C_2}$ & $\gamma_{C'C_2}\delta_{C_1C'}$ for $C'\leq C_1,C_2$ & \S\ref{sec:XT}\\
$\Schob(\Hscr)$ & the 2-category of $\Hscr$-schobers & \S\ref{sec:XT}\\
$K_\CC^0(-)$ & $K^0(-)\otimes_\ZZ \CC$ (decategorification) & \S\ref{sec:XT}\\
$\tilde{K}^0_\CC(\Eescr)$ & $\widetilde{K^0_\CC(\Eescr)}$  (decategorification)& \S\ref{sec:XT}\\
$\phi_{g,C}$ & equivalence $\Eescr_C\to \Eescr_{gC}$ (for equivariant schobers)& \S\ref{subsec:equischober}\\
$\Schob(\Gscr,\Hscr)$ & $\Gscr$-equivariant $\Hscr$-schobers & \S\ref{subsec:equischober}\\
$\Sscr$ & a particular $\Hscr$-schober (the ``GIT $\Hscr$-schober'')& \S\ref{subsec:10.1}\\
$P_\chi$ & $\chi\otimes \Oscr_W$ (for $\chi\in X(T)$)& \S\ref{subsec:10.1}\\
$P_C$ & $\oplus_{\chi\in \Lscr_{-C}} P_\chi$ & \S\ref{subsec:10.1}\\
$\Sscr_C$ & $\langle P_C\rangle$ & \S\ref{subsec:10.1}\\
$\Lambda_C$ & $\End_{W/T}(P_C)$ & \S\ref{subsec:10.1}\\
$D(\Lambda_C)$ & the derived category of right $\Lambda_C$-modules & \S\ref{subsec:10.1}\\
$\Escr_C,\gamma_{C'C},\delta_{CC'}$ & $\Hscr$-schober data for $\Sscr$ & \S\ref{subsec:10.1}\\
$\phi_{\chi,C}$ & the functor $\Sscr_C\to \Sscr_{\chi+C}$, $M\mapsto (-\chi)\otimes M$ & \S\ref{subsec:10.1}\\
$D^c(W/T)$ & bounded coherent complexes in $D(W/T)$ &\S\ref{subsec:10.1}\\
$\Escr^c_C$ & $\Escr_C\cap D^c(W/T)$ &\S\ref{subsec:10.1}\\
$W^u$ & the nullcone in $W$& \S\ref{subsec:finitelengthschober}\\
$D^u(W/T)$ & complexes in $D(W/T)$ supported on $W^u$ & \S\ref{subsec:finitelengthschober}\\
$\Sscr^c,\Sscr^f$ & subschobers of $\Sscr$ & \S\ref{subsec:finitelengthschober}\\
$\Sscr_C^u$ & $\Escr_C\cap D^u(W/T)$ & \S\ref{subsec:finitelengthschober}\\
$\Escr^f_C$ & $\Escr^u_C\cap \Escr_C^c$ & \S\ref{subsec:finitelengthschober}\\
$\mod^f(\Lambda_C)$ & a subcategory of $\Mod(\Lambda_C)$ associated with $\Sscr^f_C$ & \S\ref{subsec:finitelengthschober}\\
$\mathfrak{s}_{C,\chi}$ & ``simple'' objects in $\Sscr^f_C$ & \S\ref{subsec:finitelengthschober}\\
$\Dscr$ & $\uRHom_{W/T}(-,\Oscr_W)$ &\S\ref{subsec:autoduality}\\
$\RHom_{W/T}(-,-)'$ & $\RHom_{W/T}(\Dscr(-),-)$ &\S\ref{subsec:autoduality}\\
$(-)^-$ & a pullback of equivariant KS-data &\S\ref{subsec:GrothendieckGDuality}\\
$\langle -, -\rangle$ & generic notation for the Euler form &\S\ref{subsec:GrothendieckGDuality}\\
$\langle-,-\rangle'$ & $\langle K^0(\Dscr)(-),-\rangle$ &\S\ref{subsec:GrothendieckGDuality}\\
$\kappa$ & a splitting of $A:\TT\to H$, $A^*:X(H)\to X(\TT)$ & \S\ref{subsec:moresplittings}\\
$\bar{\Sscr}$ & a lift of $\Sscr$ to $\Schob(X(\TT),\Hscr)$ & \S\ref{subsec:Hequivariant}\\ 
$\bar{\Escr}_C,\bar{\gamma}_{C'C},\bar{\delta}_{CC'}$ & $\Hscr$-schober data for $\bar{\Sscr}$ & \S\ref{subsec:Hequivariant}\\
$\bar{\phi}_{\chi,C}$ & a lift of $\phi_{\chi,C}$& \S\ref{subsec:Hequivariant}\\
$N$ & $\Res_{\ZZ[X(H)]}(K^0(\bar{\Sscr}^c))$ &\S\ref{subsec:monodromyschober}\\
$K_h^0(\bar{\Sscr}^\ast)$, $\ast\in \{c,f\}$  &$K^0(\bar{\Sscr}^\ast)\otimes_{\ZZ[X(H)],h}\CC$ for $h\in H$ (decategorification)& \S\ref{subsec:specialisation}\\
$N_h$ & $\Res(K^0_h(\bar{\Sscr}^c))$ &\S\ref{lem:supportschobermonodromy1}\\
$\tilde{K}^0_h(\bar{\Sscr}^c)$ & $K^0_h(\bar{\Sscr}^c)\,\tilde{}\in \Perv(X(T)_\CC)$ &\S\ref{sec:decatGKZ}\\
$S^c(h)$ & the perverse sheaf on $T^*$ corresponding to $\tilde{K}^0_h(\bar{\Sscr}^c)$ &\S\ref{sec:decatGKZ}\\
$S^c$ & $S^c(\bold{1})$, $\bold{1}\in H$ the unit element &\S\ref{sec:HLS}\\
$\bar{\zeta}$ & $e^{2\pi i \zeta}\in T^\ast$ & \S\ref{sec:decatGKZ}\\
$\bar{\tau}_{u}$ & the translation $T^*\to T^*$, $x\mapsto ux$ for $u\in T^\ast$& \S\ref{sec:decatGKZ}\\
$\Qch(G,\Dscr_Y)$ & weakly $G$-equivariant $\Dscr_Y$-modules & \S\ref{sec:canonical}\\
$\Qch_\alpha(G,\Dscr_Y)$ & weakly $G$-equivariant $\Dscr_Y$-modules  with character $\alpha$ & \S\ref{sec:canonical}\\
\end{longtable}

\bibliography{nccr}
\bibliographystyle{amsalpha} 

\end{document}